\documentclass[a4paper]{amsart}

\usepackage[T1]{fontenc}
\usepackage[utf8x]{inputenc}
\usepackage{yfonts}
\usepackage{dsfont}
\usepackage{amscd,amssymb,amsmath,amsthm, amsfonts}
\usepackage{mathtools}
\usepackage{accents}
\usepackage{graphicx}
\usepackage{mathrsfs}
\usepackage[all,cmtip]{xy}
\usepackage{calc}
\usepackage[cal=boondox,scr=boondoxo]{mathalfa}
\usepackage{hyperref}
\hypersetup{colorlinks=true,pageanchor=false,linkcolor=blue,citecolor=red}
\usepackage[all]{hypcap}
\usepackage{marginnote}
\usepackage{enumitem}
\usepackage{epsfig}
\usepackage{placeins}

\newtheorem{theorem}{Theorem}[section]
\newtheorem{corollary}[theorem]{Corollary}
\newtheorem{proposition}[theorem]{Proposition}
\newtheorem{lemma}[theorem]{Lemma}

\theoremstyle{definition}

\theoremstyle{remark}
\newtheorem{remark}[theorem]{Remark}

\DeclareMathSymbol{\widetildesym}{\mathord}{largesymbols}{"65}
\newcommand\lowerwidetildesym{
  \text{\smash{\raisebox{-1.3ex}{
    $\widetildesym$}}}}
\newcommand\fixwidetilde[1]{
  \mathchoice
    {\accentset{\displaystyle\lowerwidetildesym}{#1}}
    {\accentset{\textstyle\lowerwidetildesym}{#1}}
    {\accentset{\scriptstyle\lowerwidetildesym}{#1}}
    {\accentset{\scriptscriptstyle\lowerwidetildesym}{#1}}
}

\newcommand{\rcoev}{\stackrel{\longrightarrow}{\operatorname{coev}}}
\newcommand{\rev}{\stackrel{\longrightarrow}{\operatorname{ev}}}
\newcommand{\lev}{\stackrel{\longleftarrow}{\operatorname{ev}}}
\newcommand{\lcoev}{\stackrel{\longleftarrow}{\operatorname{coev}}}
\newcommand{\piv}{g}
\newcommand{\X}{X}
 
\newcommand{\kk}{\Bbbk}

\newcommand{\wt}{\fixwidetilde}
\newcommand{\bp}[1]{{\left(#1\right)}}

\newcommand{\cat}{\mathscr{C}}
\newcommand{\brk}[1]{{{\left\langle{#1}\right\rangle}}}

\newcommand{\C}{\ensuremath{\mathbb{C}}}
\newcommand{\Z}{\ensuremath{\mathbb{Z}}}
\newcommand{\R}{\ensuremath{\mathbb{R}}}
\newcommand{\N}{\ensuremath{\mathbb{N}}}

\newcommand{\slt}{\ensuremath{\mathfrak{sl}_2}}

\newcommand{\End}{\operatorname{End}}
\newcommand{\Hom}{\operatorname{Hom}}

\renewcommand{\sl}{\mathfrak{sl}}

\newcommand{\ptr}{\operatorname{ptr}}

\newcommand{\Proj}{{\operatorname{Proj}}}
\newcommand{\calB}{\mathcal B}
\newcommand{\calC}{\mathcal C}
\newcommand{\calD}{\mathcal D}
\newcommand{\calE}{\mathcal E}

\newcommand{\calL}{\mathcal L}
\newcommand{\calR}{\mathcal R}
\newcommand{\calS}{\mathcal S}

\newcommand{\calV}{\mathcal V}
\newcommand{\calX}{\mathcal X}

\renewcommand{\leq}{\leqslant} 
\renewcommand{\geq}{\geqslant}
\newcommand{\Vect}{\mathrm{Vect}}
\newcommand{\id}{\operatorname{id}}

\renewcommand{\mod}{\operatorname{mod}} 
\newcommand{\Hmod}{\ensuremath{H\text{-}\mod}}
\newcommand{\Span}{\operatorname{Span}}

\newcommand{\rmt}{\operatorname{t}}

\newcommand{\rmH}{\mathrm{H}}

\newcommand{\rmV}{\mathrm{V}}
\newcommand{\rmZ}{\mathrm{Z}}
\newcommand{\HKR}{{\operatorname{H}_H}}

\newcommand{\Hm}{\mathrm{H}'_{\calC}}
\newcommand{\HH}{\operatorname{HH}}
\newcommand{\GK}{\operatorname{GK}}
\newcommand{\Cob}{\mathrm{Cob}}
\newcommand{\adCob}{\check{\mathrm{C}}\mathrm{ob}}
\DeclareRobustCommand{\one}{\mathbin{\text{\includegraphics[height=\heightof{$\mathbf{1}$}]{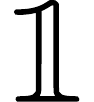}}}}
\DeclareRobustCommand{\bbSigma}{\mathbin{\text{\includegraphics[height=\heightof{$\mathbf{\Sigma}$}]{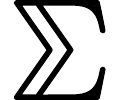}}}}
\newcommand{\bbA}{\mathbb A}
\newcommand{\bbB}{\mathbb B}
\newcommand{\bbD}{\mathbb D}
\newcommand{\bbI}{\mathbb I}
\newcommand{\bbM}{\mathbb M}
\newcommand{\bbS}{\mathbb S}

\newcommand{\op}{\mathrm{op}}

\renewcommand{\Sigma}{\varSigma}
\renewcommand{\epsilon}{\varepsilon}

\newcounter{exo} \newcounter{numexercice}
\renewcommand{\theexo}{\arabic{exo}}

\begin{document}

\raggedbottom

\title{Renormalized Hennings Invariants and 2+1-TQFTs}

\author[M. De Renzi]{Marco De Renzi} \address{Université Paris Diderot -- Paris 7, Sorbonne Paris Cité, IMJ-PRG, UMR 7586 CNRS, F-75013
  Paris, France}  \email{marco.de-renzi@imj-prg.fr}

\author[N. Geer]{Nathan Geer}
\address{Mathematics \& Statistics\\
  Utah State University \\
  Logan, Utah 84322, USA} \email{nathan.geer@gmail.com}

\author[B. Patureau-Mirand]{Bertrand Patureau-Mirand}
\address{Univ. Bretagne - Sud, UMR 6205, LMBA, F-56000 Vannes, France}
\email{bertrand.patureau@univ-ubs.fr}

\thanks{All the authors would like to thank F. Costantino and the CIMI
  Excellence Laboratory for organizing the conference ``Quantum Topology
  and Geometry'' during the thematic semester ``Invariants in
  Low Dimensional Geometry and Topology'' in Toulouse, May 2017, where
  this work was initiated.}\

\begin{abstract}
 We construct non-semisimple $2+1$-TQFTs yielding mapping class group representations in Lyubashenko's spaces.
 In order to do this, we first generalize
 Beliakova, Blanchet and Geer's logarithmic Hennings invariants based on quantum $\mathfrak{sl}_2$ to the setting of 
 finite-dimensional non-degenerate unimodular ribbon Hopf algebras. The tools used for this construction
 are a Hennings-augmented Reshetikhin-Turaev functor and modified traces.  When the Hopf algebra is factorizable, we
 further show that the universal construction of Blanchet, Habegger, Masbaum and Vogel produces a $2+1$-TQFT on
 a not completely rigid monoidal subcategory of cobordisms. 
\end{abstract}

\maketitle
\setcounter{tocdepth}{3}

\date{\today}

\section{Introduction}

 Following Atiyah \cite{A88} a $d+1$-dimensional Topological Quantum
 Field Theory (TQFT) assigns to every closed oriented $d$-dimensional
 manifold $\Sigma$ a vector space $V(\Sigma)$ 
 and assigns to every compact oriented $d+1$-dimensional cobordism $M$ from $\Sigma$ to $\Sigma'$ a linear map
 $V(M): V(\Sigma) \to V(\Sigma')$. These vector spaces and maps should satisfy certain conditions, including tensor multiplicativity 
 with respect to disjoint union and functoriality with respect to gluing of cobordisms.
 Atiyah's axioms can then be stated as follows: a $d+1$-TQFT is a symmetric monoidal functor
 $V$ from a category of cobordisms to the category of vector spaces over a field $\Bbbk$.

 Understanding TQFTs in low dimensions has been especially successful.
 In particular, the Witten-Reshetikhin-Turaev 3-manifold invariants
 \cite{RT91} coming from quantized simple Lie algebras are known to
 extend to $2+1$-TQFTs, see \cite{T94}.  These invariants and their
 TQFT extensions rely on 
 a semisimple version of the representation theory of these algebraic structures.  
 In \cite{H96}, Hennings constructed an invariant of 3-manifolds from 
 certain finite-dimensional quasitriangular Hopf algebras $H$ which need not be semisimple
 (see also \cite{KR95}, where Kauffman and Radford gave a reformulation of the
 invariant avoiding the use of orientations).  This construction is
 similar to the Reshetikhin-Turaev construction, except that Hennings
 uses the algebra $H$ directly instead of its representation theory.
 Hennings' invariant recovers the Reshetikhin-Turaev invariant when $H$
 is semisimple (see Lemma 1 of \cite{K96}).

 If $H$ is not semisimple then Hennings' invariant is zero for manifolds
 with positive first Betti number, see \cite{K98,O95}.  In particular,
 Hennings invariant vanishes on $S^1 \times S^2$, so it cannot be extended
 to a TQFT with Atiyah's axioms and cobordisms, see \cite{K96}. However,
 Lyubashenko and Kerler \cite{L95a, L95b, KL01, K96} showed that this 
 invariant can still be associated to representations of the genus $g$
 mapping class group in the space of invariants of the $g$-fold 
 tensor product of the coadjoint representation of $H$ and, more generally, to a TQFT for
 cobordisms between connected surfaces with one boundary component.
 In the general case, Kerler \cite{K98} defined a
 \emph{half-projective} TQFT which has weaker functoriality and
 monoidality properties.  According to Kerler, the relation with the
 Reshetikhin-Turaev TQFT could pass through homological TQFT (see
 \cite{K03,FN92}).

 More recently, motivated by the volume conjecture, Jun Murakami combined
 the Hennings invariant associated with quantum $\slt$ with the techniques
 of renormalized Reshetikhin-Turaev invariants to define a
 \emph{generalized Kashaev invariant} of links in a 3-manifold \cite{M17}. 
 These ideas were further generalized by Beliakova, Blanchet and the second author in
 \cite{BBG17a}.

 This paper is organized into two main sections: 
 \begin{enumerate}
  \item In the first part we generalize the logarithmic Hennings invariant of closed 3-manifolds given in 
   \cite{BBG17a} to the setting of finite-dimensional non-degenerate unimodular ribbon Hopf algebras;
  \item In the second part we extend the previous 3-manifold invariant
   to a full $2+1$-TQFT satisfying all
   functoriality and monoidality properties in the case of finite-dimensional factorizable ribbon Hopf algebras. 
 \end{enumerate}
 To explain this work more precisely we will recall some past results.

 Our construction relies on the \emph{modified trace} defined in
 \cite{GKP11, GKP13, GPV13}.  In non-semisimple representation theory
 the categorical trace can vanish for many modules. The modified
 trace is designed to replace the categorical trace in such
 situations.  It is used to define Reshetikhin-Turaev-style 3-manifold invariants and
 TQFTs from the representation theory of the unrolled
 quantum group of $\slt$, see \cite{CGP14, BCGP16, D17}. These TQFTs are based on the
 universal construction of Blanchet, Habegger, Masbaum and Vogel given
 in \cite{BHMV95}.

 For the restricted version of quantum $\slt$, Beliakova, Blanchet and the second author construct a family
 of invariants of 3-manifolds endowed with bichrome colored links by combining
 Hennings' approach with the modified trace methods (see \cite{BBG17a}). 
 As we will now explain, our first major result is to give a
 further generalization and a reformulation of this construction.

 We start with a finite-dimensional non-degenerate unimodular ribbon Hopf algebra $H$.  
 Hopf algebras have been studied extensively, see
 for example the books \cite{R11, S69a, M93}.  Of particular
 interest for us are the references \cite{CW08, D89, LN15, M90, LO16, L90, N17, R75, R93, RS88, R88, R98, S69b}.
 These papers contain many examples of the kinds of Hopf algebras we
 are considering.

 If $\calC$ denotes the category of finite-dimensional left $H$-modules, then the 
 associated Reshetikhin-Turaev functor $F_\calC$ maps $\calC$-colored ribbon graphs to $H$-module morphisms in
 $\calC$. We think of edges of such graphs as being 
 ``blue''.  In Subsection \ref{string_link_graphs}, we define an
 extension $F_\lambda$ of $F_\calC$ to bichrome graphs where we also
 allow ``red'' edges.  These red edges are colored with the regular
 representation of $H$ and have to be evaluated with a special element
 $\lambda \in H^*$ called the right integral, as in the Hennings invariant.
 We call $F_\lambda$ the Hennings-Reshetikhin-Turaev functor.
 A closed $\calC$-colored bichrome graph is {\em admissible} if it features at least one
 blue edge whose color is a projective module.  By renormalizing with the
 modified trace, we also define a renormalized invariant
 $F'_{\lambda}$ of admissible closed $\calC$-colored bichrome graphs in $S^3$.

 We use $F'_{\lambda}$ to define an invariant of
 closed 3-manifolds $M$ containing admissible closed $\calC$-colored bichrome graphs $T$.
 If $L \subset S^3$ is a surgery presentation for $M$
 then we can think of its $\ell$ components as being red and colored with the regular representation $H$, 
 so that $L \cup T$ becomes a $\calC$-colored bichrome graph in $S^3$. Then we define
 \[
  \Hm(M,T) =\calD^{-1 - \ell} \delta^{- \sigma(L)} F'_{\lambda}(L \cup T)
 \]
 where the first two factors are scalars depending only on the linking
 matrix of $L$, see Subsection \ref{3-manifold_invariants}.
 Properties of the Hennings-Reshetikhin-Turaev functor, of the modified trace and of the integral imply that
 this assignment is an isotopy invariant of $L \cup T$.  Moreover, the
 integral and the normalization factor assure it is independent of the choice of the surgery link
 $L$.  Thus, $\Hm$ is a well-defined invariant of the pair $(M,T)$ which can be thought of as a combination of Hennings'
 algebraic invariant with Reshetikhin and Turaev's categorical invariant
 which is renormalized by the modified trace.

 As we mentioned before, the invariant $\Hm$ is a generalization of the work of \cite{BBG17a} to the setting of finite-dimensional 
 non-degenerate unimodular ribbon Hopf algebras. More precisely, in Subsection \ref{SS:Log_Hennings} we show the following:
 when $H$ is (a quasi-triangular extension of) the restricted version of quantum $\slt$, 
 if the bichrome graph $T$ has all of its blue edges
 colored with the regular representation $H$, and if it only has $(1,1)$-coupons, then the invariant $\Hm$ recovers the
 \emph{logarithmic Hennings invariant} of \cite{BBG17a}.
 In addition, if the bichrome graph is a blue knot
 colored by the Steinberg-Kashaev representation of $\sl_2$, the invariant $\Hm$
 recovers Jun Murakami's \emph{generalized Kashaev invariant} \cite{M17} of links in 3-manifolds, see
 Subsection \ref{SS:Gen_Kashaev}.  
 Finally, our invariant contains via connected sums the Hennings invariant  
 $\HKR$ associated with $H$: 
 if $M$ and $M'$ are closed connected 3-manifolds and $T'$ is an admissible closed $\calC$-colored bichrome graph inside
 $M'$ then
 \[
  \Hm(M \# M',T) = \HKR(M) \Hm(M',T').
 \]

 As explained above, when $H$ is not semisimple the Hennings invariant
 can not be extended to a TQFT (in Atiyah's strict definition).  However, the admissibility requirement
 on $T$ and the modified trace allow us to obtain non-trivial vectors
 which are zero for the Hennings-Kerler-Lyubashenko TQFT.  
 As we will now explain, this is the main tool we use to produce a fully monoidal functor.

 In \cite{BHMV95} Blanchet, Habegger, Masbaum and Vogel provide a
 universal TQFT construction which is completely determined, once a quantum invariant of closed 3-manifolds has been fixed, 
 by the choice of the source cobordism category.  
 In \cite{BCGP16, D17} this procedure is applied to a setting of quantum invariants arising from generically semisimple categories. 
 Suitable restrictions on the cobordism category allow for the integration of the 
 modified trace in the process. Here we show that the universal construction can also be combined with the use of the integral 
 to construct TQFTs from Hopf algebras with no semisimplicity requirements.

 We consider a category of decorated cobordisms satisfying a certain admissibility condition, see Subsection
 \ref{admissible_cobordisms}. Loosely speaking, a decorated cobordism is obtained by
 generically cutting a 3-manifold containing an admissible graph along surfaces. Therefore 
 objects in our category, denoted $\bbSigma$, are surfaces decorated with marked points corresponding to where the graph is cut.    
 We use the admissibility requirements to dissymmetrize the category of cobordisms: there are more restrictions for cobordisms
 with empty incoming boundary than for cobordisms with empty outgoing boundary. In other words, there are less morphisms
 of the form $\bbM : \varnothing \rightarrow \bbSigma$ than morphisms of the form $\bbM' : \bbSigma \rightarrow \varnothing$.
 This breaks Atiyah's involutory axiom 
 asking that the TQFT space of $\overline{\bbSigma}$ be dual to the TQFT space of $\bbSigma$.
 The closed 3-manifold invariant $\Hm$ induces a bilinear pairing
 \[
  \langle \cdot,\cdot \rangle_{\bbSigma} : \calV'(\bbSigma) \times \calV(\bbSigma) \rightarrow \Bbbk
 \]
 on the vector spaces
 \[
  \calV(\bbSigma) = \Span_{\Bbbk} \{ \bbM : \varnothing \rightarrow \bbSigma \}, \quad
  \calV'(\bbSigma) = \Span_{\Bbbk} \{ \bbM' : \bbSigma \rightarrow \varnothing \}.
 \]
 The universal construction defines the vector spaces
 $\rmV_{\calC}(\bbSigma)$ and $\rmV'_{\calC}(\bbSigma)$ as the quotients
 of $\calV(\bbSigma)$ and $\calV'(\bbSigma)$ with respect to the right and left radicals of $\langle \cdot,\cdot \rangle_{\bbSigma}$
 respectively. We prove that properties of $\Hm$ make
 $\rmV_{\calC}$ and $\rmV'_{\calC}$ into symmetric monoidal functors, and hence TQFTs.

 Usually the vector spaces produced by the universal construction are
 not easy to determine. This is not the case in our situation.  In
 fact they are isomorphic to the images of Lyubashenko's modular functor
 for $H$ (see \cite{L95b}). In particular, we denote with $\X$ the Hopf algebra $H$ equipped with the left $H$-module structure whose
 dual is the coadjoint representation ($X^*$ is the coend for the functor mapping every pair $(V,V')$ 
 of left $H$-modules to $V^* \otimes V'$).
 Then the space assigned to a genus $g$ surface with no
 marked points is isomorphic to the space of $H$-invariant vectors in $\X^{\otimes g}$.
 More generally, the TQFT vector space of an object $\bbSigma$ given by a genus $g$ surface equipped with $k$ marked points
 labeled by finite-dimensional $H$-modules $V_1,\ldots,V_k$ is isomorphic to
 $\Hom_\calC(V_1 \otimes \cdots \otimes V_k,\X^{\otimes g})$.
 Indeed, in Subsection \ref{S:identification_TQFT_spaces} we give a bilinear pairing 
 \[
  \langle \cdot,\cdot \rangle_{\calX} : 
  \calX'_{g,\underline{V}} \times \tilde{\calX}_{g,\underline{V}} \rightarrow \Bbbk
 \]
 on the algebraic spaces
 \[
  \tilde{\calX}_{g,\underline{V}} = \Hom_\calC(H,\X^{\otimes g} \otimes \underline{V}), \quad 
  \calX'_{g,\underline{V}} = \Hom_\calC((\X^*)^{\otimes g} \otimes \underline{V},\one)
 \]
 where $\underline{V} = V_1 \otimes \ldots \otimes V_k$.
 Then the TQFT vector spaces $\rmV_{\calC}(\bbSigma)$ and $\rmV'_{\calC}(\bbSigma)$ are isomorphic to the quotients
 of $\tilde{\calX}_{g,\underline{V}}$ and $\calX'_{g,\underline{V}}$ with respect to the right and left radicals of 
 $\langle \cdot,\cdot \rangle_{\calX}$ respectively, see Corollary \ref{C:identification}. 
 Since this pairing has trivial left radical we have isomorphisms
 \[
  \calX'_{g,\underline{V}} \cong \rmV'_{\calC}(\bbSigma) \cong \rmV_{\calC}(\bbSigma)^*.
 \]
 If one of the modules $V_i$ is projective then $\bbSigma$ is dualizable 
 and there is a Verlinde formula (see Remark \ref{R:Verlinde}) expressing the dimension of Lyubashenko's spaces:
 \[
  \dim_{\Bbbk} \left( \Hom_\calC((\X^*)^{\otimes g} \otimes \underline{V},\one) \right) = \Hm(\bbS^1 \times \bbSigma).
 \]
 
 With this in mind, we ask if there are applications of our work in the area of two-dimensional conformal quantum field theory (CFT). Modular (semisimple) tensor categories and their associated quantum invariants have been very useful tools in studying rational CFTs. As explained in \cite{FS10}, the analysis of non-rational CFTs is less understood. However, recently there has been several results proven in this area, see for example \cite{BGT11, BFGT09, CRW14, FGST06, FHST04, FS10, FS17, FSS13, FSS14, GSTF06, GR06}. We can ask:  does the non-semisimple TQFT of this paper play a similar role in non-rational CFTs as does the modular TQFT in rational CFTs?

 Finally, the action of a Dehn twist along a curve $\gamma \subset \Sigma$ on $\rmV'_{\calC}(\bbSigma)$ is given by the insertion of a red surgery curve parallel to $\gamma$ with appropriate framing inside a cobordism from $\Sigma$ to $\varnothing$, see Remark \ref{R:Dehn_twist}. These encircling red curves look similar to the operators introduced by Lyubashenko to define its modular functor. This leads us to conjecture that the modular functor induced by $\rmV'_{\calC}$ is equivalent to Lyubashenko's modular functor.

\section{3-Manifold invariants from unimodular ribbon Hopf algebras}\label{S:3-manifold_invariants}

In this section we generalize the logarithmic Hennings 3-manifold invariants of \cite{BBG17a} 
via a construction which applies to every finite-dimensional non-degenerate unimodular ribbon Hopf algebra.

\subsection{Modified traces and stabilized categories of left modules}\label{S:modified_traces}

We start by recalling classical results on ribbon Hopf algebras.
Standard references for the theory are provided by \cite{S69a} and \cite{R11}.

Let us fix for this section a finite-dimensional non-degenerate unimodular ribbon Hopf algebra $H$ over a field $\Bbbk$.
We briefly recall some of the definitions involved. 
The Hopf algebra $H$ is a finite-dimensional vector space endowed with a multiplication $m : H \otimes H \rightarrow H$, 
a unit $\eta : \Bbbk \rightarrow H$, a coproduct 
$\Delta : H \rightarrow H \otimes H$, a counit $\epsilon : H \rightarrow \Bbbk$,
an antipode $S : H \rightarrow H$, an R-matrix $R = \sum_{i=1}^r a_i \otimes b_i \in H \otimes H$ and a ribbon element $v \in \rmZ(H)$.
We denote with $u$ the Drinfeld element $\sum_{i=1}^r S(b_i)a_i \in H$ and with $\piv$ the pivotal
element $uv^{-1} \in H$.
For the coproduct we will use Sweedler's notation $\Delta^{n-1}(x) = x_{(1)} \otimes \ldots \otimes x_{(n)}$
which hides the summation. As a consequence of the finite-dimensionality of $H$, the antipode $S$ is invertible.

A right integral of $H$ is a linear form $\lambda \in H^*$ satisfying $\lambda f = f(1_H) \cdot \lambda$
for every $f \in H^*$. 
This means that $(\lambda f)(x)=(\lambda \otimes f)(\Delta(x))=f(1_H)\lambda(x)$ for every $f\in H^*$ and every $x \in H$,
or equivalently that $\lambda(x_{(1)}) \cdot x_{(2)}=\lambda(x) \cdot 1_H$ for every $x \in H$.  
A left cointegral of $H$ is a vector $\Lambda \in H$ satisfying 
$x \Lambda = \epsilon(x) \Lambda$
for every $x \in H$. Since $H$ is finite-dimensional,
right integrals form a 1-dimensional ideal in $H^*$ and
left cointegrals form a 1-dimensional ideal in $H$. Moreover every non-zero right integral $\lambda \in H^*$ and every non-zero 
left cointegral $\Lambda \in H$ satisfy $\lambda(\Lambda) \neq 0$. We fix therefore for the rest of the paper
a choice of a right integral $\lambda \in H^*$ and of a left cointegral $\Lambda \in H$ satisfying $\lambda(\Lambda) = 1$.  
The Hopf algebra $H$ being \textit{unimodular} means that $S(\Lambda) = \Lambda$. 
Unimodularity of $H$ implies $\lambda$ is a \textit{quantum character}: $\lambda(xy)=\lambda(S^2(y)x)$ for all $x,y \in H$.  
Finally, the fact that $H$ is \textit{non-degenerate} means that $\Delta_+ \Delta_- \neq 0$ where
\[
 \Delta_+ := \lambda(v^{-1}), \quad \Delta_- := \lambda(v).
\]

Let $\calC$ be the ribbon linear category $H$-$\mod$ of
finite-dimensional left $H$-modules.   For $V$ an object of $\calC$, we
denote by $\rho_V:H\to\End_\kk(V)$ the associated representation.  The
unit of $\calC$ is $\one=\kk$ with $\rho_{\one}(h)=\epsilon(h) \cdot \id_{\kk}$ for every $h \in H$, and the dual of
$V$ is $V^*=\Hom_\kk(V,\kk)$ with $\rho_{V^*}(h)= \left( \rho_V(S(h)) \right)^*$ for every $h \in H$. 
The duality structural morphisms of $\calC$ are denoted
\begin{gather*}
 \lev_V : V^* \otimes V \to \one, \quad \lcoev_V : \one \to V \otimes V^*  \\
 \rev_V : V \otimes V^* \to \one, \quad \rcoev_V : \one \to V^* \otimes V
\end{gather*}
where the first two maps are given by the duality in $\Vect_{\Bbbk}$, while the second two 
are twisted with the action of the pivotal element $g$.

We will denote with $H$ the \textit{regular representation of $H$}, which is the left $H$-module structure on $H$ itself determined by the action $L : H \to \End_\kk(H)$ given by 
$L_h(x) = hx$ for all $h,x \in H$. 
Both $H$ and its dual left module $H^*$ are free rank one modules
generated by $1_H$ and $\lambda$ respectively. They are thus
isomorphic via the \textit{Radford map}
\[
 \begin{array}{rccc}
  \varphi : & H & \rightarrow & H^* \\
  & x & \mapsto & L_{S(x)}^*(\lambda).
 \end{array}
\]
It is well known that $\varphi(\Lambda) = \epsilon$ and that $\varphi^{-1}(f) = f(\Lambda_{(1)}) \cdot \Lambda_{(2)}$ 
for every $f \in H^*$, see \cite[Section 10.2]{R11}.

We will now recall the theory of modified traces.
The \textit{right partial trace} of an endomorphism $f \in \End_{\calC}(V \otimes V')$ 
is the endomorphism $\ptr(f) \in \End_{\calC}(V)$ given by
\[
 \ptr(f) := (\id_V \otimes \rev_{V'}) \circ (f \otimes \id_{V'^*}) \circ (\id_V \otimes \lcoev_{V'}).
\]

A \textit{modified trace} $\rmt$ on the ideal of projective modules $\Proj(\calC)$ is a family 
\[
 \rmt := \{ \rmt_V : \End_{\calC}(V) \rightarrow \C \mid V \in \Proj(\calC) \}
\]
of linear maps satisfying:
\begin{enumerate}
 \item $\rmt_{V}(f' \circ f) = \rmt_{V'}(f \circ f')$ 
  for all objects $V,V'$ of $\Proj(\calC)$ and
  all morphisms $f \in \Hom_{\calC}(V,V')$ and $f' \in \Hom_{\calC}(V',V)$,
 \item $\rmt_{V \otimes V'} (f)= \rmt_V(\ptr(f))$ for all objects
  $V$ of $\Proj(\calC)$ and $V'$ of $\calC$ and every morphism
  $f \in \End_{\calC}(V \otimes V')$.
\end{enumerate}

Since $\calC$ is ribbon then it follows that a modified trace also satisfies an analogous left
version of second condition of the definition, see \cite{GKP11,GPV13}.    
The existence of a modified trace was known under some conditions on $H$
(see \cite{GR17,GKP13}).  The following result is
a weaker version, for ribbon categories, of Theorem 1 in \cite{BBG17b} by Beliakova, Blanchet and Gainutdinov.

\begin{theorem}\label{T:BBGa}
 There exists a unique modified trace $\rmt$ on $\Proj(\calC)$ satisfying
 \[
  \rmt_H(f) = \lambda(\piv f(1_H))
 \]
 for every endomorphism $f$ of the projective left $H$-module $H$,
 where $\piv$ is the pivotal element of $H$. Furthermore, $\rmt$ is
 non-degenerate.\footnote{There might exist modified traces on more
   general ideals of $\calC$, but the non-degeneracy of $\rmt$ is
   related to the fact that $\Proj(\calC)$ is the smallest non-zero
   ideal of $\calC$.}
\end{theorem}

We fix from now on the choice of this modified trace. Remark that we get
$\rmt_H(\Lambda \circ \epsilon) = \epsilon(g) \lambda(\Lambda) = 1$ where 
$\Lambda : \Bbbk \to H$ denotes the unique morphism determined by $\Lambda(1) = \Lambda$.

We denote with $[n]\calC$ the \textit{$n$-th stabilized subcategory of $\calC$}, which is the linear
subcategory of $\calC$ whose objects are of the form
$[n]V := H^{\otimes n} \otimes V$ for some object $V$ of $\calC$ and
whose morphisms from $[n]V$ to $[n]V'$ are linear combinations of the form
$\sum_{i=1}^m L_{\underline{x_i}} \otimes f_i$ for some elements
$\underline{x_1}, \ldots, \underline{x_m} \in H^{\otimes n}$ and for some linear maps
$f_1, \ldots, f_m : V \rightarrow V'$. Remark that if $\sum_{i=1}^m L_{\underline{x_i}} \otimes f_i$ 
is a morphism of $[n]\calC$ then neither $L_{\underline{x_i}}$, nor $f_i$, nor their tensor product 
is required to be $H$-invariant for any $i \in \{1,\ldots,m\}$, but the linear combination 
$\sum_{i=1}^m L_{\underline{x_i}} \otimes f_i$ is.

If $V$ and $V'$ are objects of $\calC$ and if $n > 0$ then for every morphism
$\sum_{i=1}^m L_{\underline{x_i}} \otimes f_i$ in $\Hom_{[n]\calC}([n]V,[n]V')$ let us define
\[
\int_{\calC} \left( \sum_{i=1}^m L_{\underline{x_i}} \otimes f_i \right) := \sum_{i=1}^m (\lambda
\otimes \id_{[n-1]V'}) \circ ( L_{\underline{x_i}} \otimes f_i) \circ (\eta \otimes \id_{[n-1]V})
\]
where $\eta : \Bbbk \rightarrow H$ is the unit of $H$.

\begin{lemma}
 The above assignment defines a linear map 
 \[
  \int_{\calC} : \Hom_{[n]\calC}([n]V,[n]V') \rightarrow \Hom_{[n-1]\calC}([n-1]V,[n-1]V').
 \]
\end{lemma}

\begin{proof} 
 Every morphism in $\Hom_{[n]\calC}([n]V,[n]V')$ can be written as $\sum_{i = 1}^m L_{x_i} \otimes f_i$ for some $x_1, \ldots, x_m \in H$ and some linear maps $f_1, \ldots, f_m : [n-1]V \rightarrow [n-1]V'$, and we need to show that $\int_\calC (\sum_{i = 1}^m L_{x_i} \otimes f_i)$ is $H$-invariant. To do this we will use the properties of the Hopf algebra $H$: for every $h \in H$ we have 
 \[
  \epsilon (h_{(1)}) \cdot h_{(2)} =  h, \quad S(h_{(1)}) h_{(2)} = \epsilon(h) \cdot 1_H, \quad 
  h_{(2)}S^{-1}(h_{(1)}) = \epsilon(h) \cdot 1_H.
 \]

 Now for every $h \in H$ we have
 \begin{align*}
  \int_{\calC} \left(\sum_{i = 1}^m L_{x_i} \otimes f_i\right) \circ \rho_{W}(h) 
  &= \sum_{i = 1}^m \lambda(x_i) \cdot f_i \circ \rho_{W}(h) \\
  &= \sum_{i = 1}^m \epsilon(h_{(1)})\lambda(x_i) \cdot f_i \circ \rho_{W}(h_{(2)}) \\
  &= \sum_{i = 1}^m \lambda \left( x_i h_{(2)}S^{-1}(h_{(1)}) \right) \cdot f_i \circ \rho_{W}(h_{(3)}) \\
  &= \sum_{i = 1}^m \lambda(S(h_{(1)})x_i h_{(2)}) \cdot f_i \circ \rho_{W}(h_{(3)})  \\
  &= \sum_{i = 1}^m \lambda(S(h_{(1)})h_{(2)}x_i) \cdot \rho_{W'}(h_{((3)}) \circ f_i \\
  &= \sum_{i = 1}^m \epsilon(h_{(1)}) \lambda(x_i) \cdot \rho_{W'}(h_{(2)}) \circ f_i \\
  &= \sum_{i = 1}^m \lambda(x_i) \cdot \rho_{W'}(h) \circ f_i \\
  &= \rho_{W'}(h) \circ \int_{\calC} \left(\sum_{i = 1}^m L_{x_i} \otimes f_i\right),
 \end{align*}
 where the fourth equality follows from the fact that $\lambda$ is a quantum character
 and the fifth equality results from $\sum_{i = 1}^m L_{x_i} \otimes f_i$ being $H$-invarant.
 \end{proof}

\subsection{Hennings-Reshetikhin-Turaev functor for string link graphs}\label{string_link_graphs}

In this subsection we construct a family of functors defined on certain categories of $\calC$-colored ribbon graphs featuring 
red and blue edges. Red edges are related to the Hennings invariant: they are colored with the regular representation of $H$ and when
they form closed components they are evaluated using the right integral $\lambda$. 
Blue edges are somewhat more standard: they can be colored with any representation of $H$ and they are evaluated using the 
Reshetikhin-Turaev functor $F_{\calC}$.
The standard reference for ribbon graphs, ribbon categories and their associated Reshetikhin-Turaev functors is
\cite{T94}.

By a \emph{closed manifold} we mean a compact manifold without boundary. Every manifold we will consider in this paper will be oriented, every diffeomorphism of manifolds will be positive, and every link will be oriented and framed.  
If $Y$ is a manifold then we denote with $\overline{Y}$ the 
manifold obtained from $Y$ by reversing its orientation. The interval $[0,1]$ will always be denoted $I$.

An \textit{$n$-string link} is an $(n,n)$-tangle whose $i$-th incoming
boundary vertex is connected to the $i$-th outgoing boundary vertex by
an edge directed from bottom to top for every $1 \leq i \leq n$.

A \textit{bichrome graph} is a ribbon graph with edges divided into two groups, red and blue, satisfying the following condition:
for every coupon there exists a number $k \geq 0$ such that the first $k$ input legs and
the first $k$ output legs are red with positive orientation, meaning incoming and outgoing respectively, while all the other ones are blue. Red edges will be represented graphically by dashed-dotted lines. This will allow the reader to distinguish them from blue edges also in black and white versions of the paper.

The \textit{smoothing} of a bichrome graph is the red tangle obtained by throwing away every blue edge and 
by replacing every coupon with red vertical strands connecting every
red input leg with the corresponding red output leg as shown in Figure \ref{F:CouponSmoothing}.

\begin{figure}[hbt]
 \centering
 \includegraphics{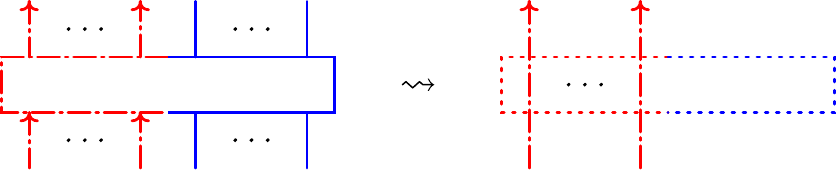}
 \caption{Smoothing of a coupon.}
 \label{F:CouponSmoothing}
\end{figure}

An \textit{$n$-string link graph} is a bichrome graph satisfying the following conditions:
\begin{enumerate}
 \item the first $n$ incoming boundary vertices and the first $n$ outgoing boundary vertices 
  are red, while all the other ones are blue;
 \item the red tangle obtained by smoothing is an $n$-string link.
\end{enumerate}
For an example of a 2-string link graph see Figure \ref{F:Example2Stringlink}.  

\begin{figure}[hbt]
  \centering
  \includegraphics{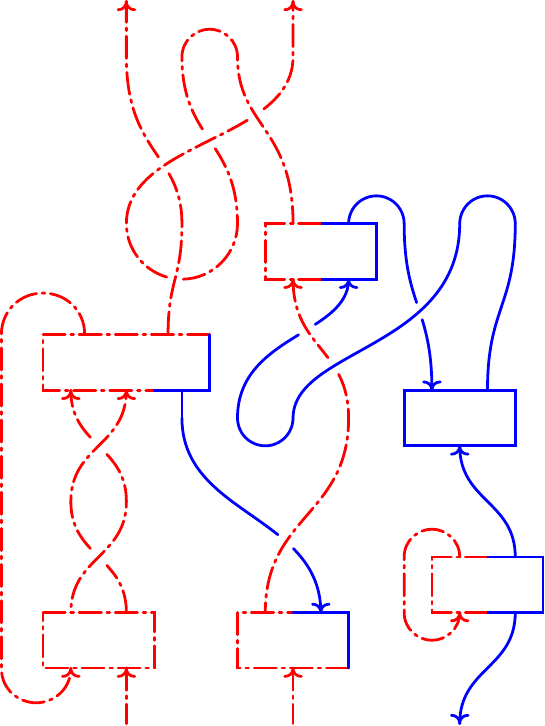}
  \caption{A 2-string link graph.}
  \label{F:Example2Stringlink}
 \end{figure}

Next we define the \textit{category $[n]\calR_{\lambda}$ of $\calC$-colored $n$-string link graphs}.
An object $[n](\underline{\epsilon},\underline{V})$ of $[n]\calR_{\lambda}$ is a finite sequence 
\[
 ( \underbrace{(+,H), \ldots, (+,H)}_n,(\epsilon_1,V_1), \ldots, (\epsilon_k,V_k) )
\]
where $\epsilon_i\in\{ +, - \}$ is a sign  and $V_i$ is an object of
$\calC$ for all $i \in \{1, \ldots, k\}$.  
A morphism
$T : [n](\underline{\epsilon},\underline{V}) \rightarrow
[n](\underline{\epsilon'},\underline{V'})$
of $[n]\calR_{\lambda}$ is an isotopy class of embeddings in
$\R^2\times[0,1]$ of $\calC$-colored $n$-string link graphs from
$[n](\underline{\epsilon},\underline{V})$ to
$[n](\underline{\epsilon'},\underline{V'})$, where
$\calC$-colorings of bichrome graphs are required to assign
the color $H$ to every red edge and to assign a morphism of $[k]\calC$
to every coupon having $2k$ red legs. The category $[0]\calR_{\lambda}$ will just be denoted $\calR_{\lambda}$.

A morphism of $[n]\calR_{\lambda}$ is \textit{open} if its smoothing features no closed component.
We denote with $[n]\calR_{\calC}$ the subcategory of $[n]\calR_{\lambda}$ having the same objects 
but featuring only open morphisms.

\begin{remark}
 The forgetful functor from $[n]\calR_{\calC}$ to the ribbon category $\calR_{\calC}$ of $\calC$-colored ribbon graphs 
 allows us to define the Reshetikhin-Turaev functor $F_{\calC}$ on $[n]\calR_{\calC}$ by dropping the distinction between 
 red and blue edges.
\end{remark}

\begin{proposition}\label{image_RT}
 The Reshetikhin-Turaev functor $F_{\calC} : [n]\calR_{\calC} \rightarrow \calC$ factors through $[n]\calC$ 
 for every $n \in \N$.
\end{proposition}

\begin{proof}
 Let us consider the category $\calB_{\Vect_{\Bbbk}}$ defined as the quotient of the free linear category generated by 
 $\calR_{\Vect_{\Bbbk}}$
 with respect to the Reshetikhin-Turaev functor $F_{\Vect_{\Bbbk}}$. This means $\calB_{\Vect_{\Bbbk}}$ has the same objects as 
 $\calR_{\Vect_{\Bbbk}}$, while 
 vector spaces of morphisms of $\calB_{\Vect_{\Bbbk}}$ are given by quotients of free vector spaces generated by morphism spaces of 
 $\calR_{\Vect_{\Bbbk}}$ with respect to kernels of the linear maps defined by $F_{\Vect_{\Bbbk}}$. 
 In the category $\calB_{\Vect_{\Bbbk}}$ we 
 represent certain morphisms in bead notation. A bead is a dot labeled with an element of $H$ lying on an edge colored with
 an object of $\calC$ inside a morphism of $\calB_{\Vect_{\Bbbk}}$.
 It represents a coupon in $\calR_{\Vect_{\Bbbk}}$ determined as follows: 
 if $\rho_V : H \rightarrow \End_{\Bbbk}(V)$ is a finite-dimensional representation of $H$ then a bead labeled with 
 $x \in H$ on a strand colored with $V$ represents a coupon colored with $\rho_V(x)$ if the strand is directed upwards, while 
 it represents a coupon colored with $\rho_V(S(x))^*$ if the strand is directed downwards. See Figure \ref{beads} for a graphical
 representation. Remark that these coupons are not in general coupons in $\calR_{\calC}$ because $\rho_V(x)$ may not be an 
 $H$-module morphism. 
 
 \begin{figure}[hbt]
  \centering
  \includegraphics{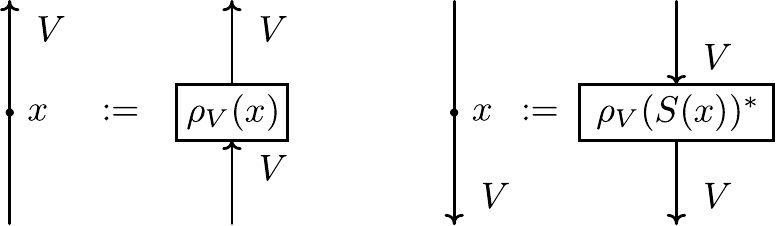}
  \caption{Beads.}
  \label{beads}
 \end{figure}
 
 We claim the functor $F_{\calC} : [n]\calR_{\calC} \rightarrow \calC$ factors through $\calB_{\Vect_{\Bbbk}}$. 
 Indeed we have a bead functor 
 $B_{\calC} : [n]\calR_{\calC} \rightarrow \calB_{\Vect_{\Bbbk}}$ which is the identity on objects and which has the following behaviour on 
 morphisms: given a diagram for a morphism of $[n]\calR_{\calC}$ which is presented as a composition of tensor products of 
 elementary ribbon graphs, the functor $B_{\calC}$ introduces beads on crossings, caps and cups and takes linear combinations of
 the morphisms thus obtained. Figure \ref{beads_functor} contains a graphical definition for the image under $B_{\calC}$
 of some of the generating morphisms of $[n]\calR_{\calC}$. 
 
 \begin{figure}[hbt]
  \centering
  \includegraphics{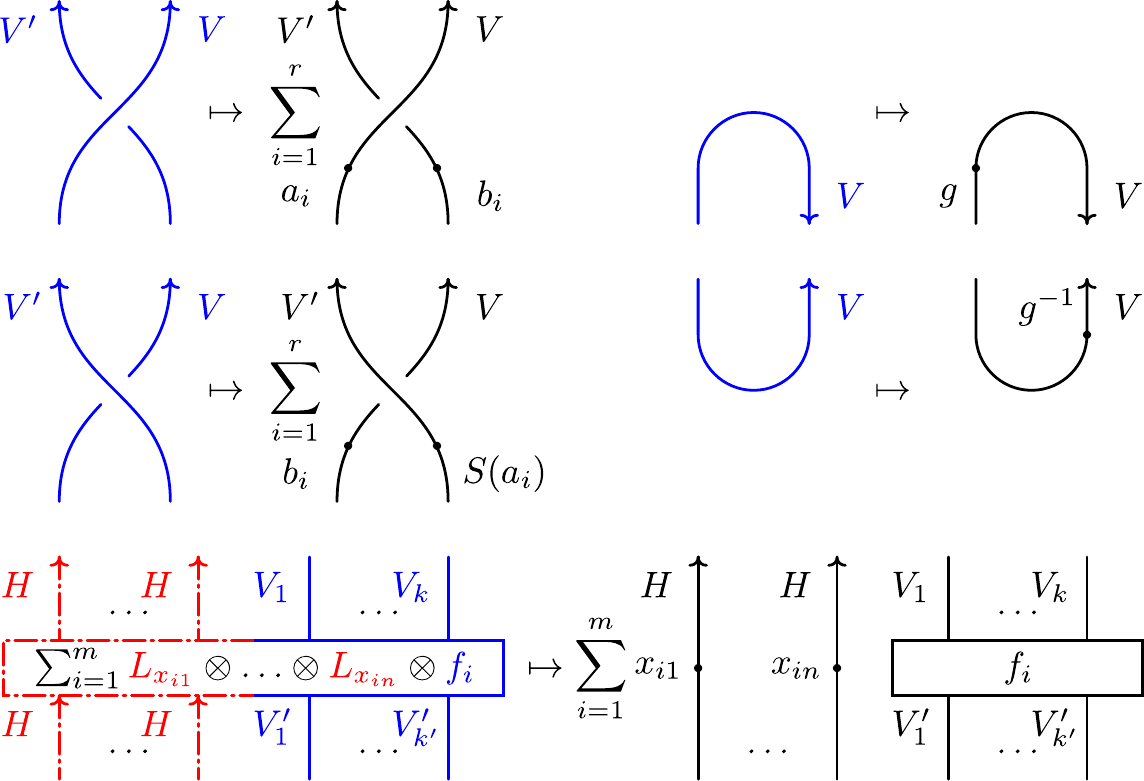}
  \caption{Graphical definition of the functor $B_{\calC}$.}
  \label{beads_functor}
 \end{figure}
 
 The functor $B_{\calC}$ introduces no bead on left caps and cups. 
 If a crossing is obtained from one of the two represented in Figure \ref{beads_functor} by 
 reversing the orientation of an edge then the label of the corresponding bead has to be replaced with its image under 
 the antipode $S$. The definition is the same for crossings involving red edges and for red caps and cups.
 
 Now the composition of the Reshetikhin-Turaev functor $F_{\calC} : [n]\calR_{\calC} \rightarrow \calC$ 
 with the forgetful functor from $\calC$ to $\Vect$ 
 can be computed as the composition of the bead functor $B_{\calC} : [n]\calR_{\calC} \rightarrow \calB_{\Vect_{\Bbbk}}$ with the functor
 from $\calB_{\Vect_{\Bbbk}}$ to $\Vect$ induced on the quotient by $F_{\Vect_{\Bbbk}}$. This makes it clear that
 morphisms in the image of $F_{\calC}$ are linear combinations of the form $\sum_{i=1}^m L_{\underline{x_i}} \otimes f_i$ for some elements $\underline{x_1},\ldots,\underline{x_m}$ of $H^{\otimes n}$ and for some linear maps $f_1,\ldots,f_m$. 
\end{proof}

Now if $(\underline{\epsilon},\underline{V})$ and $(\underline{\epsilon'},\underline{V'})$ are objects of $\calR_{\calC}$ and $n > 0$
let us consider the map 
\[
 \int_{\calR} : \Hom_{[n]\calR_{\lambda}}([n](\underline{\epsilon},\underline{V}),[n](\underline{\epsilon'},\underline{V'}))
 \rightarrow \Hom_{[n-1]\calR_{\lambda}}([n-1](\underline{\epsilon},\underline{V}),[n-1](\underline{\epsilon'},\underline{V'}))
\]
defined by the braid closure of the leftmost red strand represented in Figure \ref{integral}.

\begin{figure}[hbt]
 \centering
 \includegraphics{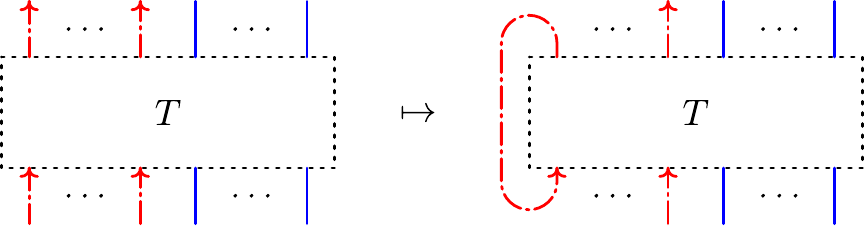}
 \caption{Graphical definition of the map $\int_{\calR}$.}
 \label{integral}
\end{figure}

\begin{proposition}\label{HRT_functor_proposition}
 There exists a unique family of functors $F_{\lambda} : [n]\calR_{\lambda} \rightarrow \calC$ extending 
 $F_{\calC} : [n]\calR_{\calC} \rightarrow \calC$ for every $n \in \N$ and satisfying
 \[
  \int_{\calC} \circ \ F_{\lambda} = F_{\lambda} \circ \int_{\calR} .
 \]
\end{proposition}

\begin{proof}
 When $T$ is a morphism of $[n]\calR_{\lambda}$ then we say a
 morphism $T'$ of $[n+k]\calR_{\calC}$ is obtained by \textit{opening $T$} if
 \[
  \underbrace{\int_{\calR} \cdots \int_{\calR}}_k T' = T.
 \]
 We want to construct a family of functors $F_{\lambda} : [n]\calR_{\lambda} \rightarrow \calC$ 
 with the desired properties.
 Let us define it as follows:
 if $[n](\underline{\epsilon},\underline{V})$ is an object of $[n]\calR_{\lambda}$ then we set
 $F_{\lambda}([n](\underline{\epsilon},\underline{V})) := F_{\calC}([n](\underline{\epsilon},\underline{V}))$. 
 If $T$ is a morphism of $[n]\calR_{\lambda}$ then we set 
 \[
  F_{\lambda}(T) := \underbrace{\int_{\calC} \cdots \int_{\calC}}_k F_{\calC}(T')
 \]
 where $T'$ is a morphism of $[n+k]\calR_{\calC}$ obtained by opening $T$ for some $k \geq 0$.

 The proof that $F_{\lambda}$ is well-defined requires a little preparation, as we need to introduce some terminology.
 First of all, we need to consider a stabilization functor 
 \[
  [1] : [n]\calR_{\lambda} \rightarrow [n+1]\calR_{\lambda}
 \]
 mapping every object $[n](\underline{\epsilon},\underline{V})$ of $[n]\calR_{\lambda}$ to the object
 $[n+1](\underline{\epsilon},\underline{V})$ of $[n+1]\calR_{\lambda}$ and mapping every morphism $T$ of 
 $\Hom_{[n]\calR_{\lambda}}([n](\underline{\epsilon},\underline{V}),[n](\underline{\epsilon'},\underline{V'}))$ to the morphism
 of $\Hom_{[n+1]\calR_{\lambda}}([n+1](\underline{\epsilon},\underline{V}),[n-1](\underline{\epsilon'},\underline{V'}))$
 represented in Figure \ref{stabilization}.
 We denote with $[k]$ the $k$-fold composition of the functor $[1]$.

 \begin{figure}[hbt]
  \centering
  \includegraphics{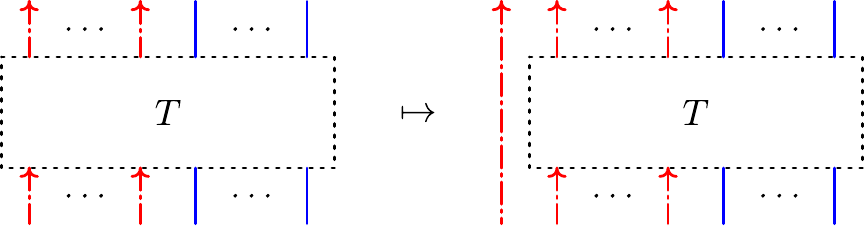}
  \caption{Graphical definition of the functor $[1]$.}
  \label{stabilization}
 \end{figure}

 We say a set $C$ of red edges of a morphism $T$ of $[n]\calR_{\lambda}$ is a \textit{chain} if all of its elements are contained in 
 one and the same component of the smoothing of $T$.
 A maximal chain in $T$ is called a \textit{cycle} if its corresponding component in the smoothing of $T$ 
 is closed, and it is called a \textit{relative cycle} otherwise.
 
 \begin{figure}[ht]
  \centering
  \includegraphics{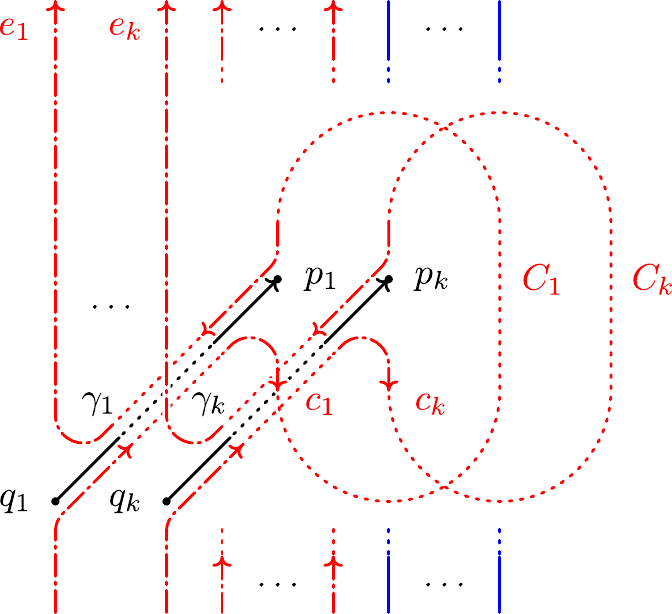}
  \caption{Diagram representing $([k]T)_{\underline{\gamma}}$.}
  \label{HRT_functor_figure}
 \end{figure}
 
 We want to show that if $T$ is a morphism of $[n]\calR_{\lambda}$ featuring exactly $k$ cycles $C_1, \ldots, C_k$
 then there exists a morphism of $[n+k]\calR_{\calC}$ which is obtained by opening $T$. 
 To define it we choose a point $p_i$ along some red edge $c_i \in C_i$ and a point $q_i$ along the red edge $e_i$ 
 connecting the $i$-th incoming boundary vertex to the $i$-th outgoing boundary vertex of $[k]T$ for all $i \in \{1, \ldots, k\}$.
 Then we consider pairwise disjoint embeddings $\iota_1, \ldots, \iota_k$ of $D^1 \times D^1$ into $\R^2 \times I$ with 
 $\iota_i((D^1 \smallsetminus \partial D^1) \times D^1)$ contained in the complement of $[k]T$, with 
 $(\iota_i(\{ -1 \} \times D^1),\iota_i(-1,0)) \subset (c_i,p_i)$ and with 
 $(\iota_i(\{ 1 \} \times \overline{D^1}),\iota_i(1,0)) \subset (e_i,q_i)$ for all $i \in \{1,\ldots, k\}$ (see the beginning of this section for our notations and conventions about orientations).
 We denote with $\gamma_i$ the path $\iota_i(D^1 \times \{ 0 \})$ from $q_i$ to $p_i$
 and we let $([k]T)_{\underline{\gamma}}$ be the morphism of $[n+k]\calR_{\lambda}$ obtained from $[k]T$ by index 1 surgery along
 $\iota_1, \ldots, \iota_k$. Up to isotopy, this morphism can be represented by a diagram like the one depicted in 
 Figure \ref{HRT_functor_figure}. Then by construction
 \[
  \underbrace{\int_{\calR} \cdots \int_{\calR}}_k ([k]T)_{\underline{\gamma}} = T.
 \] 
 
 \begin{figure}[ht]
  \centering
  \includegraphics{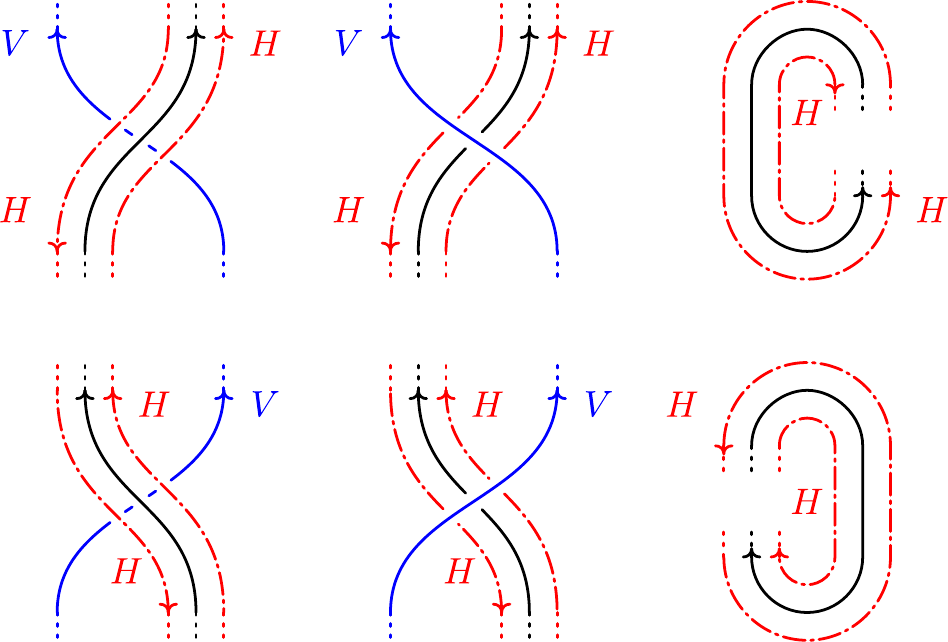}
  \caption{Local appearance of the surgered morphism $([k]T)_{\underline{\gamma}}$ around $\gamma_i$.}
  \label{independence_of_path_1}
 \end{figure}
 
 We first have to prove that $F_{\lambda}(T)$ does not depend on the choice of the embeddings $\iota_1, \ldots, \iota_k$.
 Up to isotopy we can suppose regions of the diagram around $\gamma_i$ locally look like Figure \ref{independence_of_path_1},
 possibly up to replacing blue strands with red strands. 
 Just like in the proof of Proposition \ref{image_RT} we can compute $F_{\calC}(([k]T)_{\underline{\gamma}})$ by passing
 through $\calB_{\Vect_{\Bbbk}}$.
 Let us follow the $i$-th relative cycle $(C_i \cup e_i)_{\gamma_i}$ of $([k]T)_{\underline{\gamma}}$ 
 obtained from $C_i$ and $e_i$ by surgery along $\iota_i$. We first collect beads along a parallel copy of 
 $\gamma_i$ which compose to give an element $x_{\gamma_i(2)}$. 
 Then we meet a bead labeled with the pivotal element $g$ and, moving on along the cycle $C_i$, we
 collect an element $x_{C_i}$. Finally the travel along a parallel copy of $\overline{\gamma_i}$
 contributes with an element $S(x_{\gamma_i(1)})$. All of this follows from the analysis of the beads associated with local models
 coming from Figure \ref{independence_of_path_1}, as summarized in Figure \ref{independence_of_path_2}.
 
 \begin{figure}[hb]
  \centering
  \includegraphics{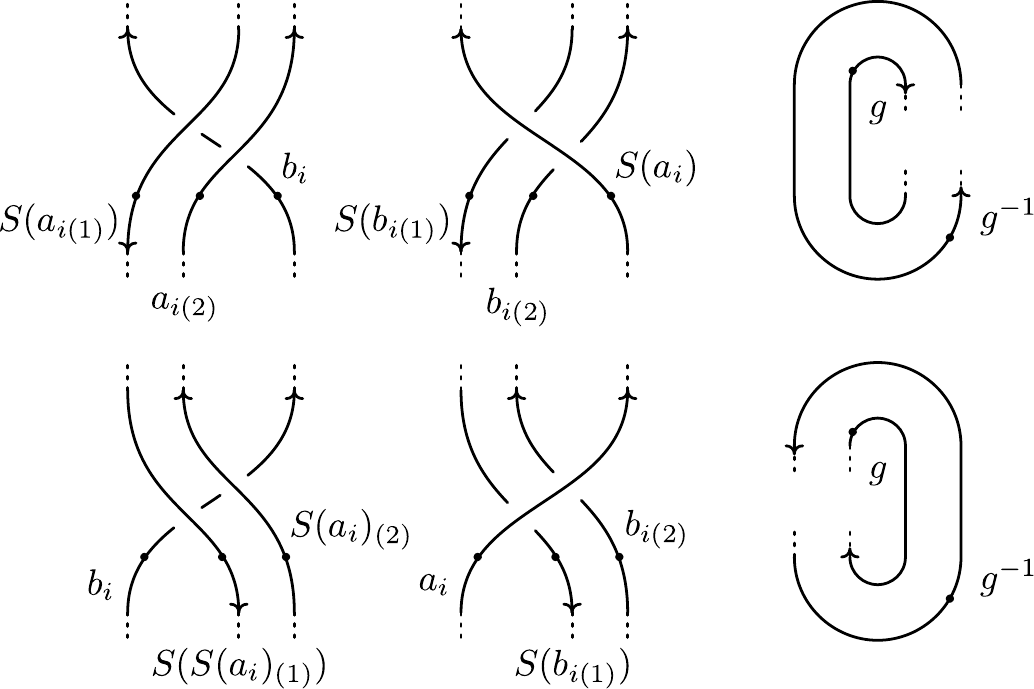}
  \caption{Beads around $\gamma_i$.}
  \label{independence_of_path_2}
 \end{figure}

 Let us denote with $([k]T)_{\underline{\gamma}} \smallsetminus (\underline{C} \cup \underline{e})_{\underline{\gamma}}$ 
 the morphism of $\calB_{\Vect_{\Bbbk}}$ obtained from $B_{\calC}(([k]T)_{\underline{\gamma}})$
 by removing all edges corresponding to relative cycles $(C_i \cup e_i)_{\gamma_i}$ for all $i \in \{1, \ldots, k\}$.
 Analogously, let us denote with $T \smallsetminus \underline{C}$ 
 the morphism of $\calB_{\Vect_{\Bbbk}}$ obtained from $B_{\calC}(T)$
 by removing all edges corresponding to cycles $C_i$ for all $i \in \{1, \ldots, k\}$.
 Then
 \begin{align*}
  \underbrace{\int_{\calC} \cdots \int_{\calC}}_k F_{\calC}(([k]T)_{\underline{\gamma}}) 
  &= \left( \prod_{i=1}^k \lambda(S(x_{\gamma_i(1)})x_{C_i}gx_{\gamma_i(2)}) \right) \cdot 
  F_{\Vect_{\Bbbk}}(([k]T)_{\underline{\gamma}} \smallsetminus (\underline{C} \cup \underline{e})_{\underline{\gamma}}) \\ 
  &= \left( \prod_{i=1}^k \epsilon(x_{\gamma_i}) \lambda(x_{C_i}g) \right) \cdot  
  F_{\Vect_{\Bbbk}}((([k]T)_{\underline{\gamma}} \smallsetminus (\underline{C} \cup \underline{e})_{\underline{\gamma}}) \\ 
  &= \left( \prod_{i=1}^k \lambda(x_{C_i}g) \right) \cdot 
  F_{\Vect_{\Bbbk}}(T \smallsetminus \underline{C}). 
 \end{align*}
 The second equality follows from $\lambda$ being a quantum character, and the last equality follows from the fact that the pivotal element and the R-matrix of a quasitriangular Hopf algebra satisfy
 \begin{gather*}
  \epsilon(g) = \epsilon(g^{-1}) = 1, \\
  (\epsilon \otimes \id_H)(R) = (\id_H \otimes\ \epsilon)(R) = 
  (\epsilon \otimes \id_H)(R^{-1}) = (\id_H \otimes\ \epsilon)(R^{-1}) = 1_H.
 \end{gather*}
 This proves that the definition of $F_{\lambda}(T)$ is actually independent of the choice of the embeddings 
 $\iota_1, \ldots, \iota_k$.
 The fact that it does not depend on the choice of the points $p_i$ either follows from the fact that $\lambda \circ R_g^*$
 is a character of $H$.

 To prove the equality in the statement
 let us consider a morphism $T$ in $[n]\calR_{\lambda}$, and let us denote by $\sum_{i = 1}^m L_{x_i} \otimes L_{\underline{x_i}} \otimes f_i$ its image under $F_{\lambda}$, with $x_i \in H$ and $\underline{x_i} \in H^{\otimes n-1}$ for every $i=1,\ldots,m$. Then 
 \[
  \int_{\calC} F_{\lambda}(T) = \int_{\calC} \left( \sum_{i = 1}^m L_{x_i} \otimes L_{\underline{x_i}} \otimes f_i \right) = 
  \sum_{i=1}^m \lambda(x_i) \cdot (L_{\underline{x_i}} \otimes f_i).
 \]
 On the other hand
 \[
  F_{\lambda} \left( \int_{\calR}T \right) = \sum_{i=1}^m \lambda(x_ig^{-1}g) \cdot (L_{\underline{x_i}} \otimes f_i).
 \]

 The unicity of the family of functors $F_{\lambda} : [n]\calR_{\lambda} \rightarrow \calC$ follows from the fact
 that every morphism $T$ of $[n]\calR_{\lambda}$ admits a morphism $T'$ of $[n+k]\calR_{\lambda}$ 
 obtained by opening $T$. Then every family of functors $F_{\lambda}$ satisfying
 \[
  \int_{\calC} \circ \ F_{\lambda} = F_{\lambda} \circ \int_{\calR}
 \]
 also satisfies
 \[
  F_{\lambda} (T) = F_{\lambda} \left( \underbrace{\int_{\calR} \cdots \int_{\calR}}_k T' \right) =
 \underbrace{\int_{\calC} \cdots \int_{\calC}}_k F_\calC(T'). \qedhere
 \]
\end{proof}

\FloatBarrier

For every $n \geq 0$ the functor $F_{\lambda} : [n]\calR_{\lambda} \rightarrow \calC$ 
is called the \textit{Hennings-Reshetikhin-Turaev functor} associated with $\lambda$.

\begin{remark}
 The definition we gave for $F_{\lambda} : [n]\calR_{\lambda} \rightarrow \calC$
 in terms of $F_{\calC}$ and $\int_{\calC}$ makes it clear that this functor too factors through $[n]\calC$ .
\end{remark}

\subsection{Renormalized Hennings invariant of closed 3-manifolds}\label{3-manifold_invariants}

This subsection is devoted to the construction of a closed 3-manifold invariant relying on two main ingredients: 
the modified trace, whose existence is ensured by Theorem \ref{T:BBGa}, and the Hennings-Reshetikhin-Turaev functor,
which was introduced in Subsection \ref{string_link_graphs}. 
These tools allow for the definition of an invariant of closed $\calC$-colored 
bichrome graphs satisfying a certain admissiblity condition. 
Indeed, in order to be able to compute the modified trace, we need a blue edge whose color is a projective object of $\calC$. 
With this in place, we can define a renormalized Hennings invariant of 
closed 3-manifolds equipped with admissible closed $\calC$-colored bichrome graphs.
Theorems \ref{T:F'Exists} and \ref{T:3ManInvMainT} prove the existence of such invariants.

A bichrome graph featuring no boundary vertex is called a
\textit{closed bichrome graph}. 
When $T$ is a closed $\calC$-colored bichrome graph and $V$ is a projective object we say an endomorphism $T_V$ of $(+,V)$ in 
$\calR_{\lambda}$ is a \textit{cutting presentation of $T$} if
\[
\rev_{(+,V)} \circ\ (T_V \otimes \id_{(-,V)})\ \circ \lcoev_{(+,V)}\ = T.
\]

We say a $\calC$-colored bichrome graph is \textit{admissible} if it features a
blue edge whose color is a projective $H$-module.

\begin{theorem}\label{T:F'Exists}
 If $T$ is an admissible closed $\calC$-colored bichrome graph and 
 $T_V$ is a cutting presentation of $T$ then
 \[
  F'_{\lambda}(T) := \rmt_V(F_{\lambda}(T_V))
 \]
 is an invariant of the isotopy class of $T$.
\end{theorem}

\begin{proof}
 The proof is similar to the one provided by \cite{GPT09}: indeed, the category
 $\calR_{\lambda}$ is a ribbon category with respect to the
 monoidal structure induced by concatenation of objects and
 disjoint union of morphisms. If $T_{V}$ and $T_{V'}$ are two different cutting
 presentations of $T$, and $c_{V,V'}$ is the braiding morphism associated with objects $(+,V)$ and $(+,V')$, 
 we can find an endomorphism $T_{V,V'}$ of $((+,V),(+,V'))$ such that $\ptr(T_{V,V'}) = T_V$ and $\ptr(c_{V,V'} \circ T_{V,V'} \circ c_{V,V'}^{-1}) = T_{V'}$.
 Then the properties of the modified trace imply
 \begin{align*}
  \rmt_V\bp{F_{\lambda}(T_V)} &= \rmt_{V \otimes V'}\bp{F_{\lambda}(T_{V,V'})} =  \rmt_{V' \otimes V} \bp{F_{\lambda}(c_{V,V'} \circ T_{V,V'} \circ c_{V,V'}^{-1})} \\
  &= \rmt_{V'}\bp{F_{\lambda}(T_{V'})}. \qedhere
 \end{align*}
\end{proof}

We call $F'_{\lambda}$ the \textit{renormalized invariant of admissible closed $\calC$-colored bichrome graphs}.

\begin{proposition}\label{P:FF'}
 Let $T,T'$ be two closed $\calC$-colored bichrome graphs. If $T'$ is admissible then 
 \[
  F'_{\lambda}(T \otimes T') = F_{\lambda}(T) F'_{\lambda}(T').
 \]
\end{proposition}

\begin{proof}
 If $T'_V$ is a cutting presentation of $T'$ then $T \otimes T'_V$ is a
 cutting presentation of $T \otimes T'$ and the proposition follows
 from the fact that $F_{\lambda}(T \otimes T'_V) = F_{\lambda}(T)F_{\lambda}(T'_V)$.
\end{proof}

Recall the definition of the coefficients $\Delta_+ = \lambda(v^{-1})$ and $\Delta_- = \lambda(v)$ given at the beginning
of Subsection \ref{S:modified_traces}. We fix now a choice of a square root $\calD$ of $\Delta_- \Delta_+$ and we define
$\delta := \frac{\calD}{\Delta_-} = \frac{\Delta_+}{\calD}$.
Remark that this is the first place we use the non-degeneracy of $H$.

\begin{theorem}\label{T:3ManInvMainT}    
 If $M$ is a closed connected 3-manifold and $T$ is an admissible closed $\calC$-colored bichrome graph inside $M$ then
 \[
  \Hm(M,T) := \calD^{-1 - \ell} \delta^{- \sigma(L)} F'_{\lambda}(L \cup T)
 \]
 only depends on the diffeomorphism class of the pair $(M,T)$, with
 $L$ being a surgery presentation for $M$ given by a red $H$-colored $\ell$-component link inside $S^3$ and 
 $\sigma(L)$ being the signature of the linking matrix of $L$.
\end{theorem}

\begin{proof}
 The proof follows the argument of Reshetikhin and Turaev by showing that the quantity
 $\calD^{-1 - \ell} \delta^{- \sigma(L)} F'_{\lambda}(L \cup T)$
 remains unchanged under orientation reversal of components of $L$ and under Kirby moves.
 These properties are proved in Subsection \ref{SS:PropForProof}:
 first of all, Proposition \ref{P:orient} implies that $F'_{\lambda}(L \cup T)$ is independent of the choice of the orientation
 of the surgery link $L$. Then, thanks to Proposition \ref{P:slide}, $F'_{\lambda}(L \cup T)$ is also 
 invariant under handle slides, known as the Kirby II move. 
 Finally, the invariance of $\Hm(M,T)$ under stabilizations, known as the Kirby I move,
 follows from the choice of the normalization factor $\calD^{-1 - \ell} \delta^{- \sigma(L)}$, which is made possible by the 
 non-degeneracy of $H$.
\end{proof}

We call $\Hm$ the \textit{renormalized Hennings invariant of admissible decorated closed 3-manifolds}.

\begin{remark}
 When we write $F'_{\lambda}(L \cup T)$ we are using a slightly
 abusive notation because $T$ is actually contained in $M$.  What we 
 mean is that we have a diffeomorphism between $S^3(L)$ and $M$, and
 that $T$ can be isotoped to be inside the image of the exterior of
 $L$ in $S^3$ under this diffeomorphism. We can therefore pull back
 $T$ to an admissible closed bichrome graph inside $S^3$ which is
 disjoint from $L$ and which we still denote with $T$.
\end{remark}

The renormalized Hennings invariant is related to the standard Hennings invariant $\HKR$ as follows.  

\begin{proposition}
 If $M$ and $M'$ are closed connected 3-manifolds and $T'$ is an admissible $\calC$-colored bichrome graph inside $M'$
 then 
 \[
  \Hm(M \# M',T')= \HKR(M) \Hm(M',T').
 \]
\end{proposition}

\begin{proof}
 It is enough to apply Proposition \ref{P:FF'} to a surgery presentation for
 $M \# M'$ which is a disjoint union of two surgery presentations for $M$ and $M'$.
\end{proof}

\subsection{Proof of the invariance} \label{SS:PropForProof} 

We conclude Section \ref{S:3-manifold_invariants} with two results which were announced in the proof of Theorem \ref{T:3ManInvMainT}.

\begin{proposition}\label{P:orient} 
 Let $T$ be a morphism of $[n]\calR_{\lambda}$, let $K$ be a closed
 red component of $T$ disjoint from coupons and let $T'$ be the morphism of $[n]\calR_{\lambda}$ obtained by
 reversing the orientation of $K$. Then
 \[
  F_{\lambda}(T) = F_{\lambda}(T').
 \]
 Similarly, if $T$ is closed and admissible then $F'_{\lambda}(T) = F'_{\lambda}(T')$.
\end{proposition}

\begin{proof}
  The proof follows from the analogous result for the Hennings invariant, and is very similar to the proof of Propositions
  \ref{HRT_functor_proposition}. Indeed, we can
  first choose a diagram for $T$ presenting $K$ as the closure of a braid. 
  Then we can compute $F_{\lambda}(T)$ and $F_{\lambda}(T')$ by passing
  through the bead functor $B_{\calC}$.
  The contribution of $K$ to $F_{\lambda}(T)$ is computed by picking a base point on
  $K$, by collecting all the beads we meet whilst travelling along $K$
  according to its orientation in order to obtain an element $x_K$ of
  $H$ and by evaluating the integral $\lambda$ against $x_Kg$.
  Therefore if $T \smallsetminus (K \cup \underline{C})$ denotes the morphism of $\calB_{\Vect_{\Bbbk}}$ obtained from $B_{\calC}(T)$ 
  by removing all components corresponding to cycles $K, C_1, \ldots, C_k$ we get
  \[
   F_{\lambda}(T) = \lambda(x_Kg) \left( \prod_{i=1}^k \lambda(x_{C_i}g) \right)
   \cdot F_{\Vect_{\Bbbk}}(T \smallsetminus (K \cup \underline{C})).
  \]
  Now $F_{\lambda}(T')$ is obtained from $F_{\lambda}(T)$ by applying $S$ to all the beads we meet along $K$
  and by reversing the order of the multiplications. Then, since the morphism
  $T' \smallsetminus (\overline{K} \cup \underline{C})$ of $\calB_{\Vect_{\Bbbk}}$ obtained from $B_{\calC}(T')$ 
  by removing all components corresponding to cycles $\overline{K}, C_1, \ldots, C_k$ coincides with $T \smallsetminus (K \cup \underline{C})$,
  we get 
  \begin{align*}
   F_{\lambda}(T') &= \lambda(S(x_{K})g) \left( \prod_{i=1}^k \lambda(x_{C_i}g) \right)
   \cdot F_{\Vect_{\Bbbk}}(T' \smallsetminus (\overline{K} \cup \underline{C})) \\
   &= \lambda(x_{K}g) \left( \prod_{i=1}^k \lambda(x_{C_i}g) \right) \cdot F_{\Vect_{\Bbbk}}(T \smallsetminus (K \cup \underline{C})) 
  \end{align*}
  where the last equality follows from Proposition 4.2 of \cite{H96}.
  The property for $F'_{\lambda}$ follows now from the property for $F_{\lambda}$ applied to cutting presentations.
\end{proof} 

\begin{proposition}\label{P:slide}
 Let $T$ be a morphism of $[n]\calR_{\lambda}$, let $K$ be a closed
 red component of $T$ disjoint from coupons and let $e$ be an edge of
 $T$. Let $T'$ be the morphism of $[n]\calR_{\lambda}$ obtained by
 sliding $e$ over $K$.  Then
 \[
  F_{\lambda}(T) = F_{\lambda}(T').
 \]
 Similarly, if $T$ is closed and admissible then  $F'_{\lambda}(T) = F'_{\lambda}(T')$.
\end{proposition}

\begin{proof}
  Up to isotopy the slide of $e$ over $K$ is the
  operation which transforms the diagram represented in the left-hand side of Figure
  \ref{slide_1} into the right-hand one.
  
 \begin{figure}[ht]
  \centering
  \includegraphics{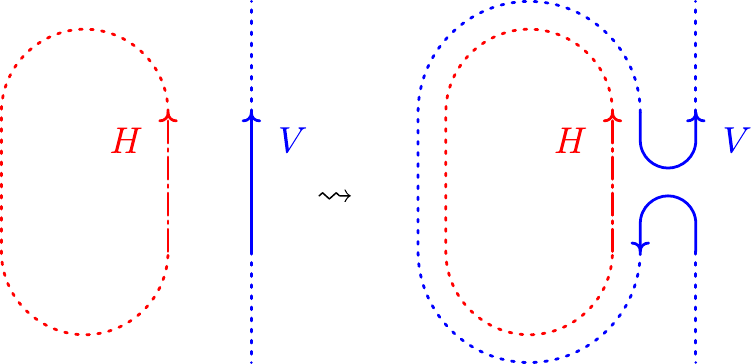}
  \caption{Slide of a $V$-colored blue edge $e$ over an $H$-colored red component $K$. 
  The edge $e$ is also allowed to be red, in which case $V = H$.}
  \label{slide_1}
 \end{figure}

 If we choose a diagram for $T'$ presenting $K$ as the closure of a braid, regions around the edge $e'$ resulting from the slide of $e$ over $K$ will
 locally look like Figure \ref{slide_2}, possibly up to replacing $V'$-colored blue strands with $H$-colored red strands.

 \begin{figure}[ht]
  \centering
  \includegraphics{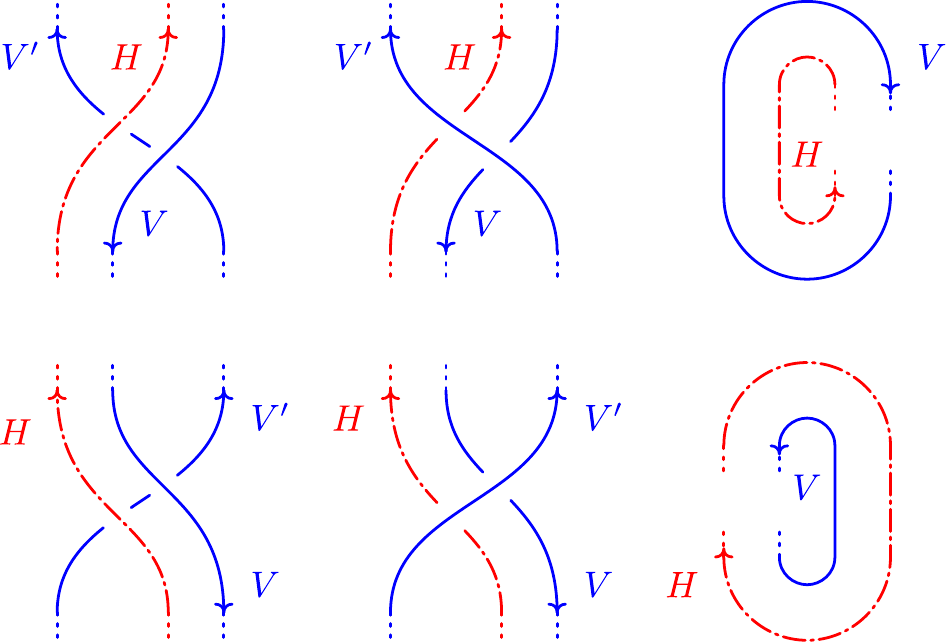}
  \caption{Local appearance of $T'$ around the edge $e'$.}
  \label{slide_2}
 \end{figure}

 Just like in the proof of Proposition \ref{HRT_functor_proposition},
 we can compute $F_{\lambda}(T)$ and $F_{\lambda}(T')$ by passing
 through $\calB_{\Vect_{\Bbbk}}$. As we showed in the proof of
 Proposition \ref{HRT_functor_proposition}, we can compute the
 contribution of $K$ to $F_{\lambda}(T)$ by picking a base point on
 $K$, by collecting all the beads we meet whilst travelling along $K$
 according to its orientation in order to obtain an element $x_K$ of
 $H$, and by evaluating the integral $\lambda$ against $x_Kg$.
 Therefore if $T \smallsetminus (K \cup \underline{C})$ denotes the
 morphism of $\calB_{\Vect_{\Bbbk}}$ obtained from $B_{\calC}(T)$ 
 by removing all components corresponding to cycles
 $K, C_1, \ldots, C_k$ then we get
 \[
  F_{\lambda}(T) = \lambda(x_Kg) \left( \prod_{i=1}^k \lambda(x_{C_i}g) \right)
  \cdot F_{\Vect_{\Bbbk}}(T \smallsetminus (K \cup \underline{C})).
 \]
 Let us see how the slide of $e$ over $K$ affects this computation.
 If we follow $e'$ we collect an element $S(x_{K(2)})$
 and then we meet a bead labeled with $g^{-1} = S(g)$. At the same time the contribution of $K$ to
 $F_{\lambda}(T')$ has changed to $\lambda(x_{K(1)}g)$.
 All of this follows again from the analysis of the beads associated with local models
 coming from Figure \ref{slide_2}, as summarized in Figure \ref{slide_3}.

 \begin{figure}[hb]
  \centering
  \includegraphics{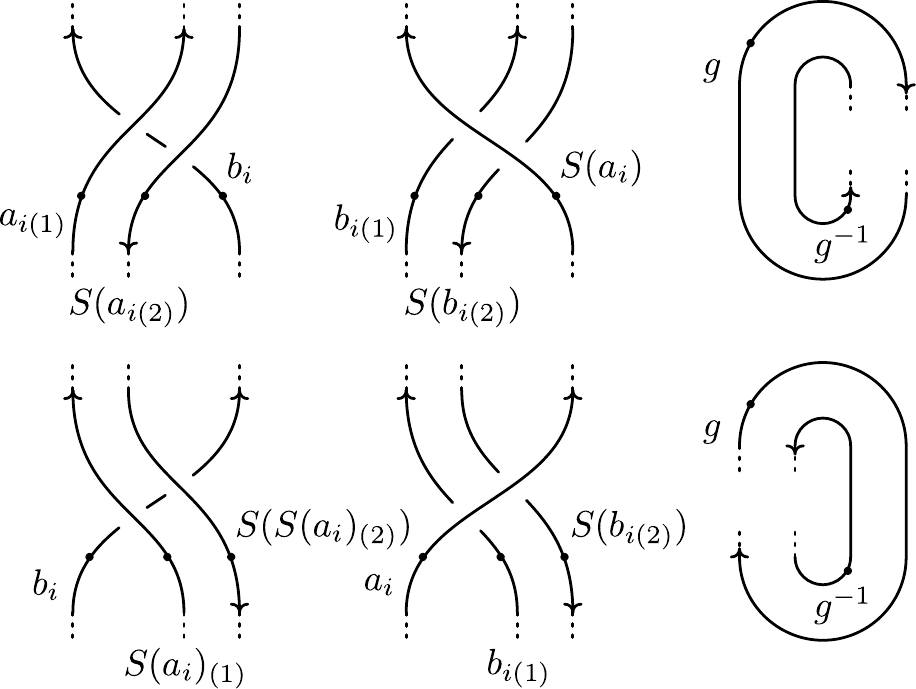}
  \caption{Beads around $e'$.}
  \label{slide_3}
 \end{figure}

 Therefore if $T' \smallsetminus (K \cup \underline{C})$ denotes the morphism of
 $\calB_{\Vect_{\Bbbk}}$ obtained from $B_{\calC}(T')$ 
 by removing all components corresponding to cycles $K, C_1, \ldots, C_k$ we get
 \begin{align*}
  F_{\lambda}(T') &= \lambda(x_{K(1)}g) \left( \prod_{i=1}^k \lambda(x_{C_i}g) \right)
  \cdot F_{\Vect_{\Bbbk}}(T' \smallsetminus (K \cup \underline{C})) \\
  &= \lambda(x_{K}g) \left( \prod_{i=1}^k \lambda(x_{C_i}g) \right) \cdot F_{\Vect_{\Bbbk}}(T \smallsetminus (K \cup \underline{C})), 
 \end{align*}
 where the last equality follows from the fact that
 \[
  \lambda(x_{K(1)}g) \cdot S(x_{K(2)}g) = \lambda(x_{K}g) \cdot S(1_H) = \lambda(x_{K}g) \cdot 1_H
 \]
 because $F_{\Vect_{\Bbbk}}(T' \smallsetminus (K \cup \underline{C}))$ carries an $S(x_{K(2)}g)$-labeled bead.
 The property for $F'_{\lambda}$ follows now from the property for $F_{\lambda}$ applied to cutting presentations.
\end{proof}

\FloatBarrier

\section{2+1-TQFTs from factorizable Hopf algebras}\label{S:TQFT}

In this section we extend the renormalized Hennings invariants of Section \ref{S:3-manifold_invariants} 
to $2+1$-TQFTs in the case of finite-dimensional factorizable ribbon Hopf algebras. We also give an explicit characterization
of the resulting TQFT vector spaces.

\subsection{Algebraic TQFT spaces}\label{pairing_section}

We start with the definition of a family of vector spaces which will be later identified with the family of TQFT vector spaces
coming from the functorial extension of the invariant $\Hm$. In order to do this, we first need some preliminary work.
Let us fix for this whole section a finite-dimensional factorizable ribbon Hopf algebra $H$. 
We recall that $H$ being \textit{factorizable} means that the Drinfeld map 
\[
 \begin{array}{rccc}
  \psi : & H^* & \rightarrow & H \\
  & f & \mapsto & \displaystyle \sum_{i,j = 1}^r f(b_ja_i) \cdot a_jb_i
 \end{array}
\]
is an isomorphism. Then $H$ is automatically unimodular, see for example \cite{R11}.  
Moreover $H$ is also non-degenerate, as proven in Proposition 7.1 of \cite{H96}.
In particular, we have a renormalized Hennings invariant $\Hm$ associated with the category $\calC = H$-$\mod$.
We will denote with $\X$ the \textit{dual coadjoint representation of $H$}, which is the left $H$-module structure on $H$ itself determined by the action $\rho_X : H \to \End_\kk(H)$ given by
\[
 \rho_X(h)(x) = h_{(2)}xS^{-1}(h_{(1)})
\]
for all $h,x \in H$. The dual of $\X$ is the coend for the functor mapping every pair $(V,V')$ of objects of $\calC$ to $V^* \otimes V'$, see \cite{K96,L95b,V06}.

\begin{lemma}\label{L:Drinfeld_and_Radford_for_X}
 The Radford map $\varphi$ and the Drinfeld map $\psi$ induce isomorphisms 
 \[
  \varphi_X := \varphi \circ S \in \Hom_{\calC}(X,X^*), \quad \psi_X := S^{-1} \circ \psi \in \Hom_{\calC}(X^*,X)
 \]
 satisfying $\bp{\psi_X \circ \varphi_X}^* = \varphi_X \circ \psi_X$.
\end{lemma}

\begin{proof}
 The fact that $\varphi_X$ and $\psi_X$ are invertible follows from the invertibility of $\varphi$, $\psi$ and $S$.
 The $H$-equivariance of $\varphi_X$ follows from the computation
 \begin{align*}
  \varphi_X(\rho_X(h)(x)) &= \lambda \circ L_{S^2(h_{(2)}xS^{-1}(h_{(1)}))} = \lambda \circ L_{S^2(h_{(2)})S^2(x)S(h_{(1)})} \\ 
  &= \lambda \circ L_{S^2(x)} \circ \rho_X(S(h)) = \rho_{X^*}(h)(\varphi_X(x))
 \end{align*}
 for every $x \in X$ and every $h \in H$, where the third equality follows from $\lambda$ being a quantum character. The $H$-equivariance of $\psi_X$ follows from the computation
 \begin{align*}
  \psi_X(\rho_{X^*}(h)(f)) &= \sum_{i,j=1}^r f(S(h_{(1)})b_ja_ih_{(2)}) \cdot S^{-1}(a_jb_i) \\
  &= \sum_{i,j=1}^r f(b_ja_i) \cdot S^{-1}(h_{(1)}a_jb_iS(h_{(2)})) \\
  &= \sum_{i,j=1}^r f(b_ja_i) \cdot h_{(2)}S^{-1}(a_jb_i)S^{-1}(h_{(1)}) = \rho_X(h)(\psi_X(f))
 \end{align*}
 for every $f \in X^*$ and every $h \in H$, where the second equality follows from the properties of the R-matrix using the pivotal structure of $\calC$.
 Finally, for every $f \in X^*$ and every $x \in X$, we have
 \begin{align*}
  (\varphi_X \circ \psi_X)(f)(x) &= \sum_{i,j=1}^r f(b_ja_i) \lambda (S(a_jb_i)x) = \sum_{i,j=1}^r f(S^{-1}(a_ib_j)) \lambda (b_ia_jx) \\
  &= \sum_{i,j=1}^r f(S^{-1}(a_ib_j)) \lambda (S^2(x)b_ia_j) = f((\psi_X \circ \varphi_X)(x)) \\
  &= (\psi_X \circ \varphi_X)^*(f)(x),
 \end{align*}
 where the second equality follows from the identity $(S \otimes S)(R) = R$.
\end{proof}

\begin{proposition}\label{P:HomG}
 For every $n \in \N$ and for all objects $V,V'$ of $\calC$ there exist explicit isomorphisms
 \begin{gather*}
  \Theta : \Hom_{\calC}(V,\X^{\otimes n} \otimes V') \rightarrow \Hom_{[n]\calC}([n]V,[n]V'), \\
  \Theta' : \Hom_{\calC}((\X^*)^{\otimes n} \otimes V,V') \rightarrow \Hom_{[n]\calC}([n]V,[n]V').
 \end{gather*}
\end{proposition}

Before proving Proposition \ref{P:HomG}, we point out that
the explicit isomorphisms we will choose will not be the simplest
possible, but will instead be precisely the ones we will need in the
following for an efficient description of TQFT vector spaces.  
In particular, we will use the morphisms introduced in the following lemma.

\begin{lemma}\label{L:important_morphisms} 
 The linear maps
 \[
  \begin{array}{rccc}
   \alpha : & H \otimes X & \rightarrow & H \\
   & h \otimes x & \mapsto & xh,
  \end{array} \quad
  \begin{array}{rccc}
   \beta : & H & \rightarrow & H \otimes X \\ 
   & h & \mapsto & \Lambda_{(1)}h \otimes S^{-1}(\Lambda_{(2)})
  \end{array}
 \]
 define morphisms in $\Hom_{[1]\calC}([1]X,[1]\one)$ and in $\Hom_{[1]\calC}([1]\one,[1]X)$ satisfying
 \begin{gather*}
  \int_{\calC} \bp{\beta \circ \alpha} = \int_{\calC} \bp{\ell_{\beta} \circ \ell_{\alpha}} = \id_\X, \\
  \alpha \circ \left( \id_H \otimes \left( (\lambda \otimes \id_X) \circ \beta \circ L_h \circ \eta \right) \right) = 
  \ell_{\alpha} \circ \left( \id_H \otimes \left( (\lambda \otimes \id_X) \circ \ell_{\beta} \circ L_h \circ \eta \right) \right) = L_h
 \end{gather*}
 for every $h \in H$, where $\ell_{\alpha}$ and $\ell_{\beta}$ are the morphisms 
 \begin{align*}
  \ell_{\alpha} &:= \alpha \circ (\id_H \otimes (\psi_X \circ \varphi_X)) \in \Hom_{[1]\calC}([1]X,[1]\one), \\
  \ell_{\beta} &:= (\id_{H} \otimes\ (\psi_X \circ \varphi_X)^{-1}) \circ \beta \in \Hom_{[1]\calC}([1]\one,[1]X).
 \end{align*}
 Moreover, the morphisms
 \begin{align*}
  \alpha' &:= (\alpha \otimes \id_{X^*}) \circ (\id_H \otimes \lcoev_X) \in \Hom_{[1]\calC}([1]\one,[1]X^*), \\
  \beta' &:= (\id_H \otimes \rev_X) \circ (\beta \otimes \id_{X^*}) \in \Hom_{[1]\calC}([1]X^*,[1]\one), \\
  \ell_{\alpha'} &:= (\id_H \otimes \ (\varphi_X \circ \psi_X)) \circ \alpha' \in \Hom_{[1]\calC}([1]\one,[1]X^*), \\
  \ell_{\beta'} &:= \beta' \circ (\id_{H} \otimes \ (\varphi_X \circ \psi_X)^{-1}) \in \Hom_{[1]\calC}([1]X^*,[1]\one)
 \end{align*}
 satisfy
 \begin{gather*}
  \int_{\calC} \bp{\alpha' \circ \beta'} = \int_{\calC} \bp{\ell_{\alpha'} \circ \ell_{\beta'}} = \id_{\X^*}, \\
  \left( \id_H \otimes \left( \lambda \! \circ \! L_h \! \circ \beta' \! \circ (\eta \otimes \id_{X^*}) \right) \right) \circ \alpha' \! = 
  \left( \id_H \otimes \left( \lambda \! \circ \! L_h \! \circ \ell_{\beta'} \! \circ (\eta \otimes \id_{X^*}) \right) \right) \circ \ell_{\alpha'} \! = L_h
 \end{gather*}
 for every $h \in H$.
 Finally, the isomorphism $h_\X := \varphi_X \circ \psi_X \circ \varphi_X \in \Hom_\cat(\X,\X^*)$ satisfies
 \[
  h_\X = \int_{\calC} \bp{\ell_{\alpha'} \circ \alpha} = \int_{\calC} \bp{\alpha' \circ \ell_{\alpha}}.
 \]
\end{lemma}

\begin{figure}[htb]
 \centering
 \includegraphics{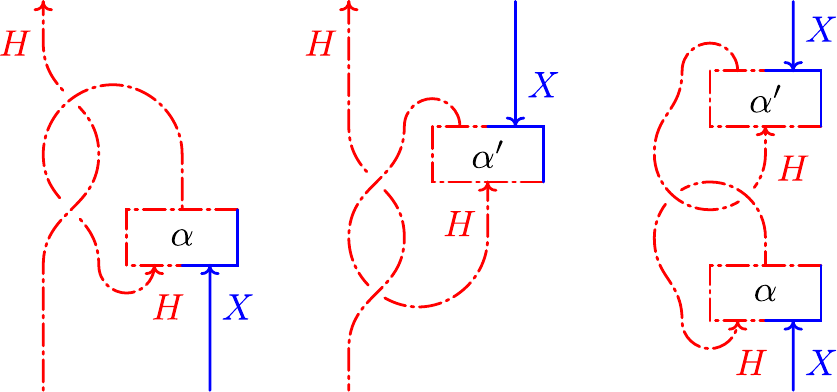}
 \caption{$\calC$-Colored bichrome graphs representing the morphisms $\ell_{\alpha}$, $\ell_{\alpha'}$ and $h_X$ respectively.}
 \label{F:long_Hopf_morphisms}
\end{figure}

\begin{proof}
 First, let us prove $\alpha \in \Hom_{[1]\calC}([1]\X,[1]\one)$. For every $h \in H$ we have
 \begin{align*}
  \alpha \circ (L_{h_{(1)}} \otimes \rho_X(h_{(2)})) &= m \circ \tau \circ (L_{h_{(1)}} \otimes \rho_X(h_{(2)})) \\
  &= m \circ (L_{h_{(3)}} \otimes L_{S^{-1}(h_{(2)})h_{(1)}}) \circ \tau \\
  &= m \circ (L_{h_{(3)}} \otimes L_{S^{-1}(S(h_{(1)})h_{(2)})}) \circ \tau \\
  &= \epsilon(h_{(1)}) \cdot \left( m \circ (L_{h_{(2)}} \otimes \id_H) \circ \tau \right) \\
  &= m \circ (L_h \otimes \id_H) \circ \tau = L_h \circ m \circ \tau = L_h \circ \alpha,
 \end{align*}
 where $\tau(h \otimes x) := x \otimes h$ for every $h \otimes x \in H \otimes X$ and where $m$ is the multiplication map. Furthermore, $\alpha$ can be written as $\sum_{i = 1}^k L_{x_i} \otimes f_i$ where $\{ x_1, \ldots, x_k \}$ is a basis of $X$ and $\{ f_1, \ldots, f_k \}$ is the corresponding dual basis of $X^*$.  
 Next, in order to prove $\beta \in \Hom_{[1]\calC}([1]\one,[1]\X)$, we will show that for every $h \in H$ we have the equality
 $(L_{h_{(1)}} \otimes \rho_X(h_{(2)}) ) \circ \beta \circ L_{S(h_{(3)})} = \epsilon(h) \cdot \beta$,
 which is equivalent to the $H$-equivariance of $\beta$ through the pivotal structure of $\calC$.
 Indeed, for every $h \in H$ the left-hand side of the equality is given by
 \begin{align*}
  L_{h_{(1)}\Lambda_{(1)}S(h_{(4)})} &\otimes (h_{(3)}S^{-1}(\Lambda_{(2)})S^{-1}(h_{(2)})) \\
  &= L_{h_{(1)}\Lambda_{(1)}S(h_{(4)})} \otimes S^{-1}(h_{(2)}\Lambda_{(2)}S(h_{(3)})) \\
  &= L_{(\id_H \otimes\ S^{-1})(\Delta(h_{(1)}\Lambda S(h_{(2)})))} \circ (\id_H \otimes\ \eta) \\
  &= \epsilon(h) \cdot L_{\Lambda_{(1)}} \otimes S^{-1}(\Lambda_{(2)})
 \end{align*}
 where the last equality follows from the fact that $h_{(1)}\Lambda S(h_{(2)}) = \epsilon(h) \cdot \Lambda$. 
 Now recall that the inverse $\varphi^{-1}$ of the Radford map $\varphi$ is given by $\varphi^{-1}(f) = f(\Lambda_{(1)}) \cdot \Lambda_{(2)}$ for every $f \in H^*$. Then for every $x \in \X$ we have
 \begin{align*}
  \int_{\calC} (\beta \circ \alpha) (x) &= \lambda(\Lambda_{(1)}x) \cdot S^{-1}(\Lambda_{(2)}) = \lambda(S^2(x) \Lambda_{(1)}) \cdot S^{-1}(\Lambda_{(2)}) \\
  &= S^{-1} \left( \varphi(S(x))(\Lambda_{(1)}) \cdot  \Lambda_{(2)} \right) = S^{-1} \left( \varphi^{-1}(\varphi(S(x))) \right)  = x.
 \end{align*}
 This means $\int_{\calC} (\beta \circ \alpha) = \id_\X$. Analogously, for all $h,h' \in H$ we have
 \begin{align*}
  \alpha \left( h' \otimes (\lambda \otimes \id_X)(\beta(h)) \right) &= \lambda(\Lambda_{(1)}h) \cdot \alpha(h' \otimes S^{-1}(\Lambda_{(2)})) \\
  &= \lambda(S^2(h) \Lambda_{(1)}) \cdot \alpha(h' \otimes S^{-1}(\Lambda_{(2)})) \\
  &= \alpha \left( h' \otimes S^{-1} \left( \varphi(S(h))(\Lambda_{(1)}) \cdot \Lambda_{(2)} \right) \right) \\
  &= \alpha \left( h' \otimes S^{-1} \left( \varphi^{-1}(\varphi(S(h))) \right) \right) = hh'.
 \end{align*}
 Thus $\alpha \circ \left( \id_H \otimes \left( (\lambda \otimes \id_X) \circ \beta \circ L_h \circ \eta \right) \right) = L_h$. 
 But now $\int_{\calC} (\ell_{\beta} \circ \ell_{\alpha})$ is given by
 \[
  (\psi_X \circ \varphi_X)^{-1} \circ \left( \int_{\calC} (\beta \circ \alpha) \right) \circ (\psi_X \circ \varphi_X) = \id_\X,
 \]
 and analogously $\ell_{\alpha} \circ \left( \id_H \otimes \left( (\lambda \otimes \id_X) \circ \ell_{\beta} \circ L_h \circ \eta \right) \right)$ is given by
 \[
   \alpha \circ \left( \id_H \otimes \left( (\psi_X \circ \varphi_X) \circ (\psi_X \circ \varphi_X)^{-1} \circ (\lambda \otimes \id_X) \circ \beta \circ L_h \circ \eta \right) \right)  = L_h.
 \]
 The proof of the corresponding equalities for $\alpha'$, $\beta'$, $\ell_{\alpha'}$, and $\ell_{\beta'}$ is similar. Now the equalities involving $h_X$ follow from the fact that 
 $\int_{\calC}(\alpha' \circ \alpha) = \varphi_X$, which in turn follows from the computation
 \[
  \int_{\calC}(\alpha' \circ \alpha)(x)(x') = \lambda(x'x) = \lambda(S^2(x)x') = \varphi_X(x)(x')
 \]
 for all $x,x' \in X$. Indeed, this means that
 \begin{align*}
  \int_{\calC}(\alpha' \circ \ell_{\alpha}) &= \int_{\calC}(\alpha' \circ \alpha) \circ (\psi_X \circ \varphi_X) = \varphi_X \circ \psi_X \circ \varphi_X = (\varphi_X \circ \psi_X) \circ \int_{\calC}(\alpha' \circ \alpha) \\ &= \int_{\calC}(\ell_{\alpha'} \circ \alpha).
 \end{align*}
 Finally, to see that $\ell_{\alpha}$ is the image under the Hennings-Reshetikhin-Turaev functor $F_{\lambda}$ of the first $\calC$-colored bichrome graph represented in Figure \ref{F:long_Hopf_morphisms}, remark that 
 for every $h \otimes x \in [1]X$ we have
 \[
  \ell_{\alpha}(h \otimes x) = \sum_{i,j = 1}^r \lambda(S^2(x)b_ja_i) \cdot S^{-1}(a_jb_i)h
  = \sum_{i,j = 1}^r \lambda(S(a_jb_i)x) \cdot b_ja_ih.
 \]
 An analogous computation shows that $\ell_{\alpha'}$ is the image
 under $F_{\lambda}$ of the second $\calC$-colored bichrome graph
 represented in Figure \ref{F:long_Hopf_morphisms}.
\end{proof}

\begin{proof}[Proof of Proposition \ref{P:HomG}.]
 We can define $\Theta$ as the map that sends every morphism $f$ of $\Hom_{\calC}(V,\X^{\otimes n} \otimes V')$ to the morphism $\Theta(f)$ of $\Hom_{[n]\calC}([n]V,[n]V')$ given by the image under $F_{\lambda}$ of the $\calC$-colored bichrome graph represented in Figure \ref{isomorphism_1}.

 \begin{figure}[ht]
  \centering
  \includegraphics{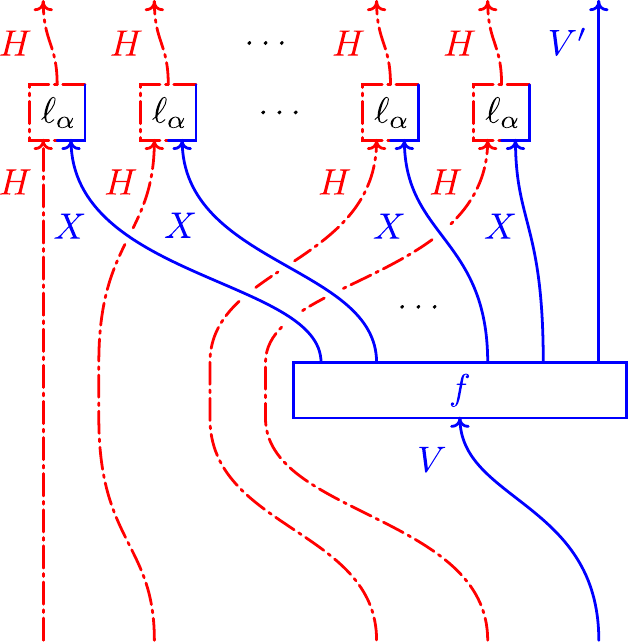}
  \caption{$\calC$-Colored $n$-string link graph representing the morphism $\Theta(f)$ of $\Hom_{[n]\calC}([n]V,[n]V')$.}
  \label{isomorphism_1}
 \end{figure}

 \begin{figure}[t]
  \centering
  \includegraphics{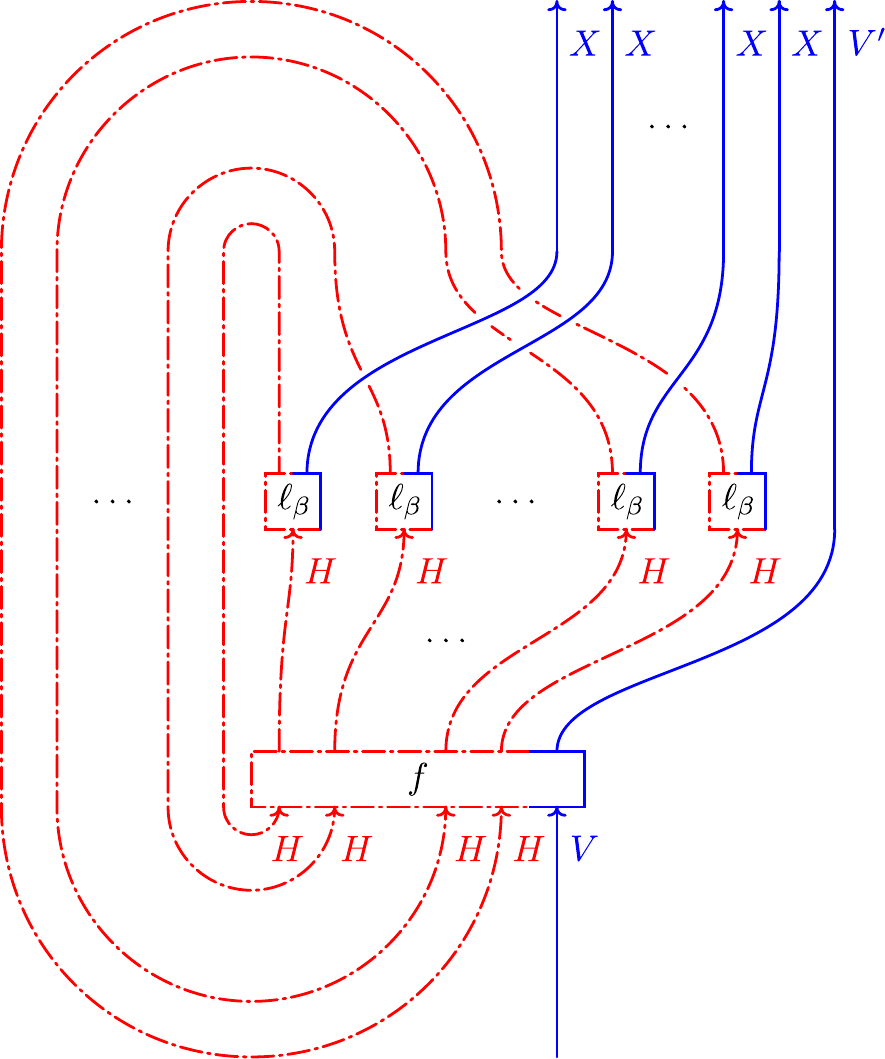}
  \caption{$\calC$-Colored bichrome graph representing the morphism $\Theta^{-1}(f)$ 
  of $\Hom_{\calC}(V,\X^{\otimes n} \otimes V')$.}
  \label{isomorphism_2}
 \end{figure}

 It is now relatively easy to see that $\Theta^{-1}$ is the map that sends every
 morphism $f$ of $\Hom_{[n]\calC}([n]V,[n]V')$ to the morphism
 $\Theta^{-1}(f)$ of $\Hom_{\calC}(V,\X^{\otimes n} \otimes V')$ given
 by the image under $F_{\lambda}$ of the $\calC$-colored bichrome
 graph represented in Figure \ref{isomorphism_2}.

 \begin{figure}[hb]
  \centering
  \includegraphics{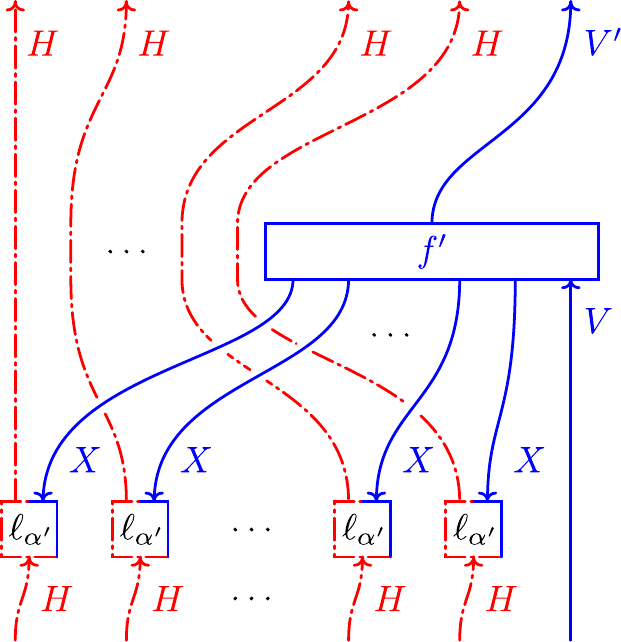}
  \caption{$\calC$-Colored $n$-string link graph representing the morphism $\Theta'(f')$ of $\Hom_{[n]\calC}([n]V,[n]V')$.}
  \label{isomorphism_3}
 \end{figure}

 Analogously, we can define $\Theta'$ as the map that sends every
 morphism $f'$ of $\Hom_{\calC}((\X^*)^{\otimes n} \otimes V,V')$ to
 the morphism $\Theta'(f')$ of $\Hom_{[n]\calC}([n]V,[n]V')$ given by
 the image under $F_{\lambda}$ of the $\calC$-colored bichrome graph
 represented in Figure \ref{isomorphism_3}.

 It is now relatively easy to see that $\Theta'^{-1}$ is the map that sends every
 morphism $f'$ of $\Hom_{[n]\calC}([n]V,[n]V')$ to the morphism
 $\Theta'^{-1}(f')$ of $\Hom_{\calC}(V,\X^{\otimes n} \otimes V')$
 given by the image under $F_{\lambda}$ of the $\calC$-colored
 bichrome graph represented in Figure \ref{isomorphism_4}.
 
 \begin{figure}[ht]
  \centering
  \includegraphics{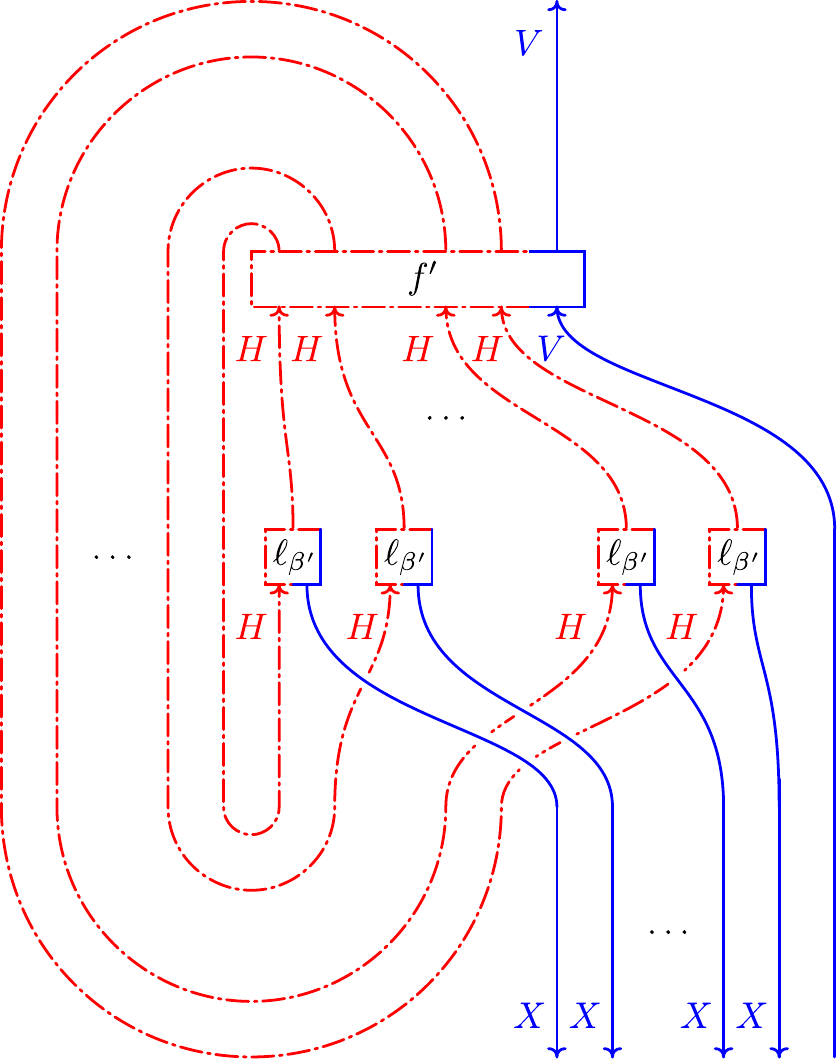}
  \caption{$\calC$-Colored bichrome graph representing the morphism $\Theta'^{-1}(f')$
  of $\Hom_{\calC}((\X^*)^{\otimes n} \otimes V,V')$.}
  \label{isomorphism_4}
 \end{figure}

\end{proof}

\FloatBarrier

We are now ready to define our algebraic TQFT spaces. They will be constructed as (quotients of) certain morphism spaces of the type we studied before. For every $g \in \N$ and for every object $V$ of $\calC$ we consider the vector spaces
\[
 \tilde{\calX}_{g,V} := \Hom_\calC(H,\X^{\otimes g} \otimes V), \quad
 \calX'_{g,V} := \Hom_\calC((\X^*)^{\otimes g} \otimes V,\one).
\]
Remark that the space $\tilde{\calX}_{g,V}$ is isomorphic to $\X^{\otimes g} \otimes V$ via the isomorphism mapping every $\tilde{f} \in \tilde{\calX}_{g,V}$ to $\tilde{f}(1_H)$.

We define now a bilinear pairing $\brk{\cdot,\cdot}_{\calX} : \calX'_{g,V} \times \tilde{\calX}_{g,V} \to \kk$ 
as follows: for every $f' \in \calX'_{g,V}$ and every $f \in \tilde{\calX}_{g,V}$ we set 
\[
 \brk{f',f}_{\calX} := \rmt_H(F_{\lambda}(T_{\calX,f',f}))
\]
where $T_{\calX,f',f}$ is the $\calC$-colored bichrome graph represented in Figure \ref{F:algebraic_pairing_X}. 

\begin{figure}[htb]
 \centering
 \includegraphics{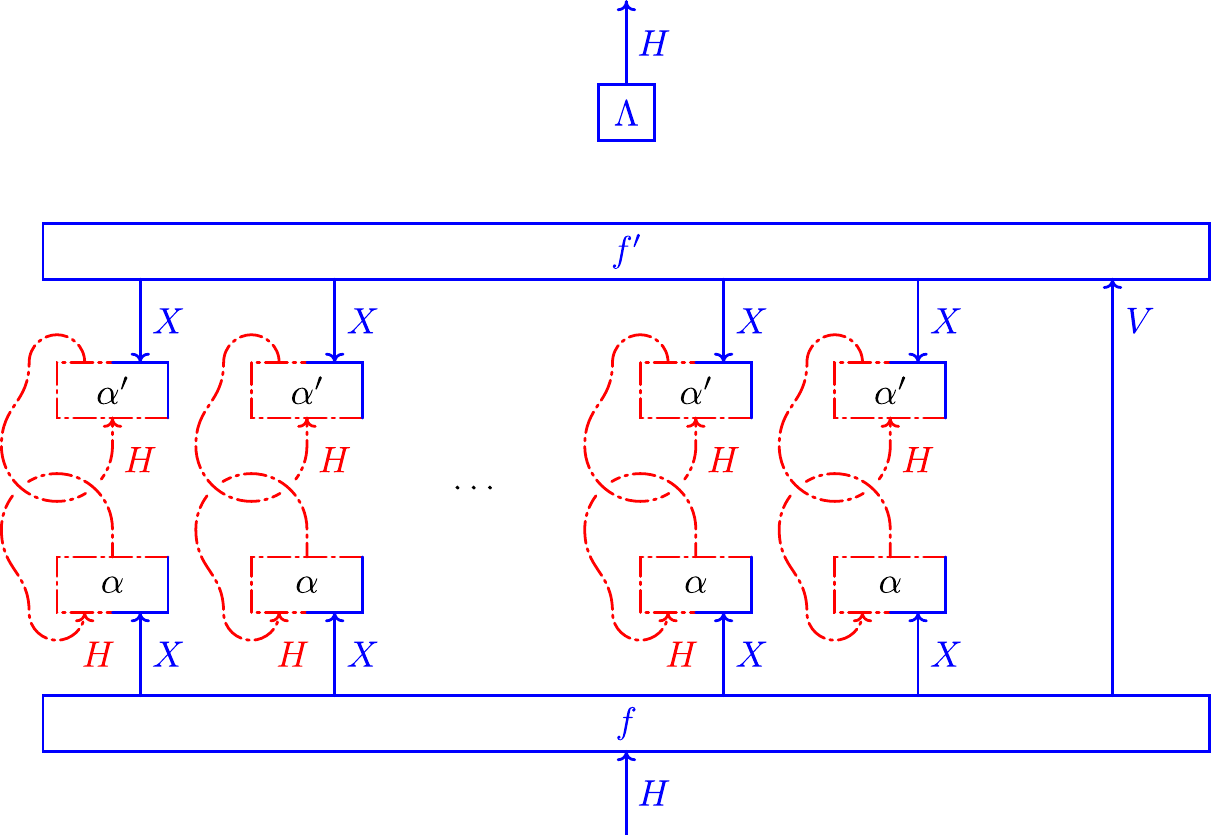}
 \caption{$\calC$-Colored bichrome graph $T_{\calX,f',f}$ defining the pairing $\brk{f',f}_{\calX}$.}
 \label{F:algebraic_pairing_X}
\end{figure}

We define $\calX_{g,V}$ to be the quotient of $\tilde{\calX}_{g,V}$ with respect to the right radical of this pairing.
Then we can induce a bilinear pairing between $\calX'_{g,V}$ and $\calX_{g,V}$ which we still denote by $\brk{\cdot,\cdot}_{\calX}$.

\begin{proposition} 
 The pairing $\brk{\cdot,\cdot}_{\calX} : \calX'_{g,V} \times \calX_{g,V} \to \kk$ is non-degenerate.
\end{proposition}

\begin{proof}
 If $f' \in \calX'_{g,V}$ is non-zero then there exists some
 $\underline{x} \otimes v \in \X^{\otimes g}\otimes V$ such that $f'(h_\X^{\otimes g}(\underline{x}) \otimes v) \neq 0$.  
 Let $f_{\underline{x} \otimes v} \in \tilde{\calX}_{g,V}$ be the unique $H$-module morphism mapping $1_H$ to 
 $\underline{x} \otimes v$. Then $f' \circ (h_\X^{\otimes g} \otimes \id_V) \circ f_{\underline{x} \otimes v}$ 
 is a non-zero $H$-module morphism from $H$ to $\Bbbk$. This means it is a non-zero multiple of $\epsilon$, which gives 
 \[
  \rmt_H \left( \Lambda \circ f' \circ (h_\X^{\otimes g} \otimes \id_V) \circ f_{\underline{x} \otimes v} \right) \neq 0.
 \]
 Thus the left radical of $\brk{\cdot,\cdot}_{\calX} : \calX'_{g,V} \times \calX_{g,V} \to \kk$ is trivial.
 Since the right radical is trivial by definition, we can conclude.
\end{proof}

It will be useful for the following to have another model of these algebraic TQFT spaces, corresponding to the isomorphism provided by Proposition \ref{P:HomG}.
For every $g \in \N$ and for every object $V$ of $\calC$ we consider vector spaces
\[
 \tilde{\calS}_{g,V} := \Hom_{[g]\calC}([g]H,[g]V), \quad
 \calS'_{g,V} := \Hom_{[g]\calC}([g]V,[g]\one).
\]
Remark that, thanks to Proposition \ref{P:HomG}, we have explicit isomorphisms between $\tilde{\calX}_{g,V}$ and $\tilde{\calS}_{g,V}$ and between $\calX'_{g,V}$ and $\calS'_{g,V}$.

We define now a bilinear pairing $\brk{\cdot,\cdot}_{\calS} : \calS'_{g,V} \times \tilde{\calS}_{g,V} \to \kk$ 
as follows: for every $f' \in \calS'_{g,V}$ and every $f \in \tilde{\calS}_{g,V}$ we set 
\[
 \brk{f',f}_{\calS} := \rmt_H(F_{\lambda}(T_{\calS,f,f'}))
\]
where $T_{\calS,f,f'}$ is the $\calC$-colored bichrome graph represented in Figure \ref{F:algebraic_pairing_S}.

\begin{figure}[htb]
 \centering
 \includegraphics{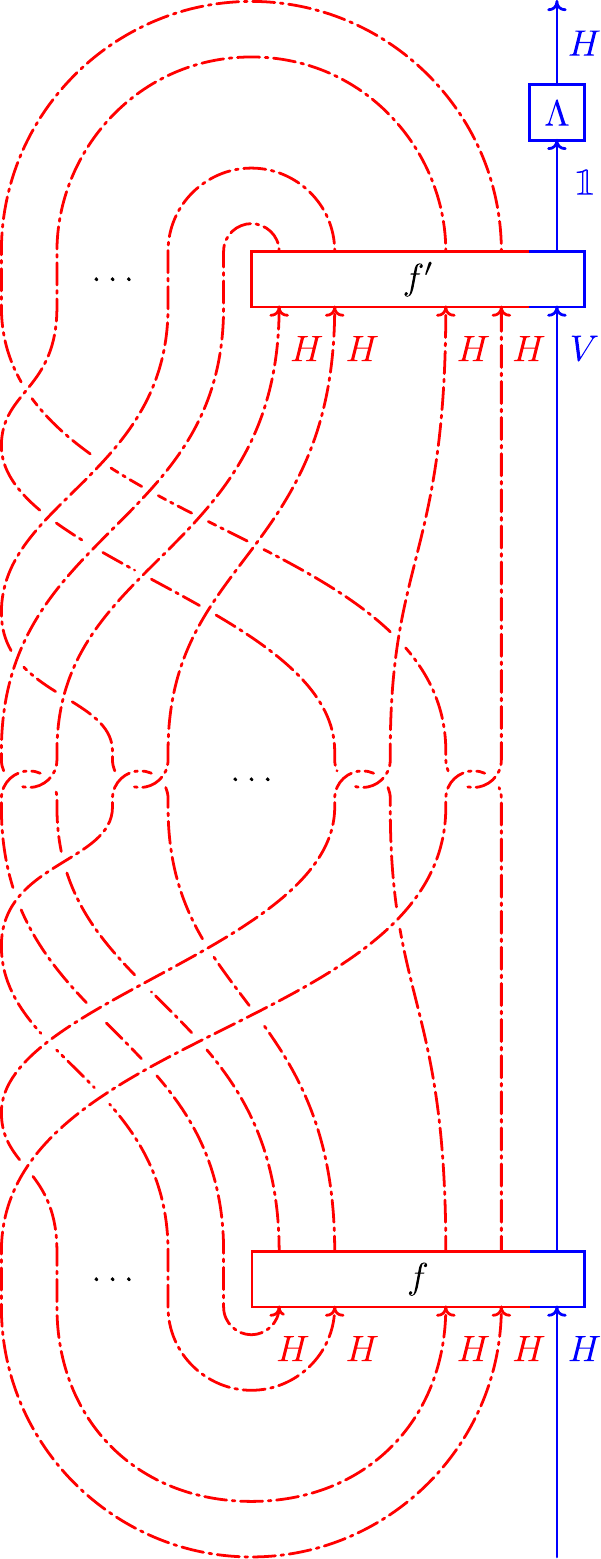}
 \caption{$\calC$-Colored bichrome graph $T_{\calS,f,f'}$ defining the pairing $\brk{f',f}_{\calS}$.}
 \label{F:algebraic_pairing_S}
\end{figure}

We define $\calS_{g,V}$ to be the quotient of $\tilde{\calS}_{g,V}$ with respect to the right radical of this pairing.
Then we can induce a bilinear pairing between $\calS'_{g,V}$ and $\calS_{g,V}$ which we still denote by $\brk{\cdot,\cdot}_{\calS}$.

\begin{proposition}\label{P:translation_of_pairings}
 Every morphism $f' \in \calX'_{g,V}$ and every morphism $f \in \tilde{\calX}_{g,V}$ satisfy
 \[
  \brk{f',f}_{\calX} = \brk{\Theta'(f'),\Theta(f)}_{\calS}
 \]
 where $\Theta$ and $\Theta'$ are the explicit isomorphisms given by Proposition \ref{P:HomG}.
\end{proposition}

\begin{proof}
 The pairing $\brk{\Theta'(f'),\Theta(f)}_{\calS}$ is given by the modified trace of the endomorphism of $H$ defined as the evaluation of the functor $F_{\lambda}$ against the $\calC$-colored bichrome graph represented in Figure \ref{F:translation_of_pairings}. 
 But remark now that every red strand meeting an $\alpha$-colored coupon can be slid upwards using the topmost component of the Hopf link that is tangled to it, and analogously every red strand meeting an $\alpha'$-colored coupon can be slid downwards using the bottommost component of the Hopf link that is tangled to it. This way, we can disentangle all Hopf links from the rest of the bichrome graph. Then, since red Hopf links provide trivial contribution to the image under the functor $F_{\lambda}$, we can just remove them, and we are left with a $\calC$-colored ribbon graph representing a morphism whose modified trace gives precisely $\brk{f',f}_{\calX}$.
 
 \begin{figure}[bt]
  \centering
  \includegraphics{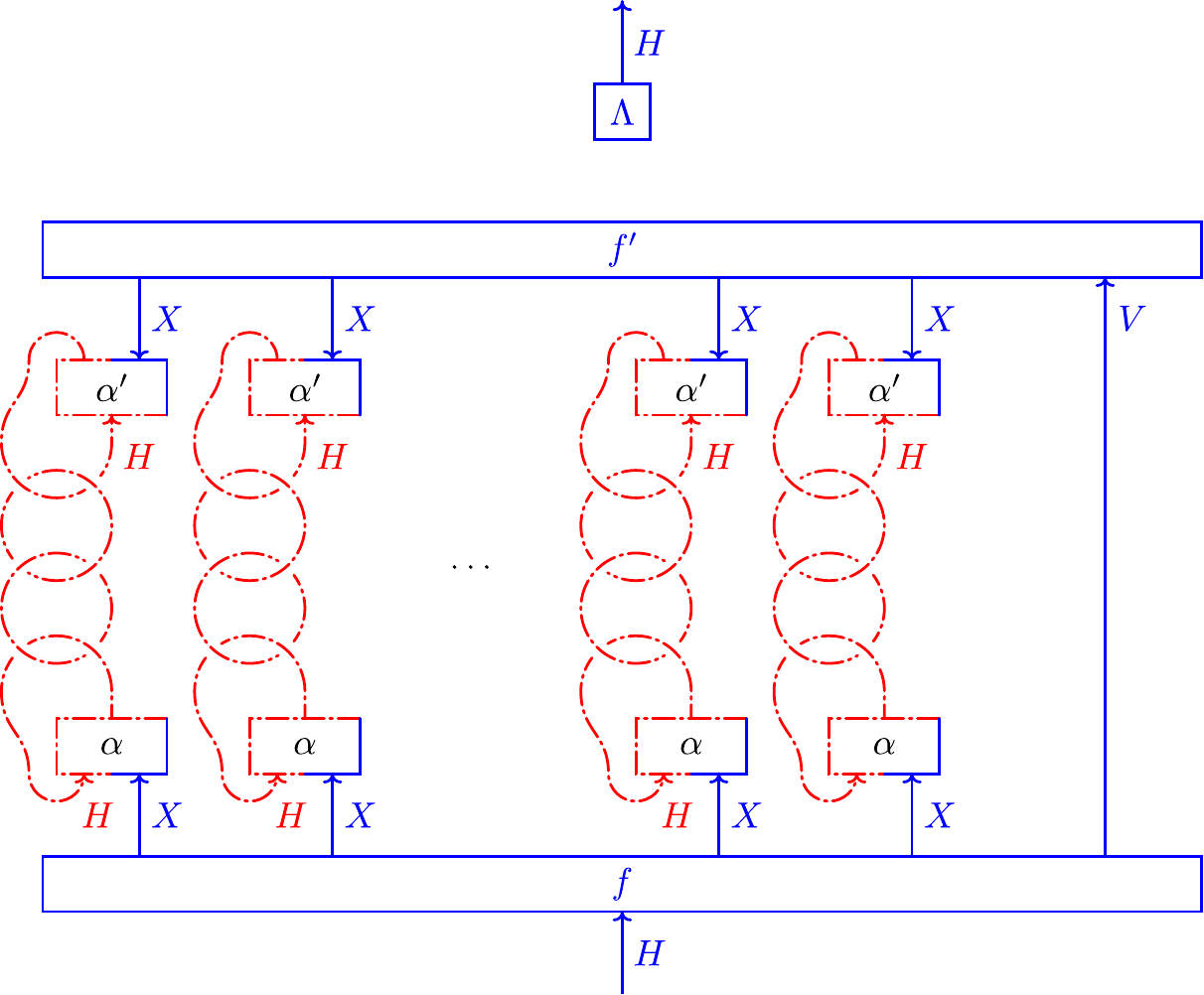}
  \caption{$\calC$-Colored bichrome graph computing the pairing $\brk{\Theta'(f'),\Theta(f)}_{\calS}$.}
  \label{F:translation_of_pairings}
 \end{figure}
\end{proof}

\FloatBarrier

\subsection{Skein equivalence}\label{S:skein_equivalence}

In this subsection we introduce the concept of skein equivalence for morphisms of $[n]\calC$ and we establish some related properties
of $\calC$-colored bichrome graphs.
If $[n](\underline{\epsilon},\underline{V})$ and $[n](\underline{\epsilon'},\underline{V'})$ are objects of 
$[n]\calR_{\lambda}$ then we say two formal linear combinations $\sum_{i=1}^m \alpha_i \cdot T_i$ and $\sum_{i'=1}^{m'} \alpha'_{i'} \cdot T'_{i'}$ of morphisms of $\Hom_{[n]\calR_{\lambda}}([n](\underline{\epsilon},\underline{V}),[n](\underline{\epsilon'},\underline{V'}))$ 
are \textit{skein equivalent} if
\[
 \sum_{i=1}^m \alpha_i \cdot F_{\lambda}(T_i) = \sum_{i'=1}^{m'} \alpha'_{i'} \cdot F_{\lambda}(T'_{i'}).
\]
Such a skein equivalence will be denoted
\[
 \sum_{i=1}^m \alpha_i \cdot T_i \doteq \sum_{i'=1}^{m'} \alpha'_{i'} \cdot T'_{i'}.
\]

\begin{lemma}\label{cutting_lemma}
 There exists a \textit{modularity parameter} $\zeta \in \Bbbk^*$ realizing the skein equivalence of
 Figure \ref{cutting}.
\end{lemma}

\begin{figure}[ht]
 \centering
 \includegraphics{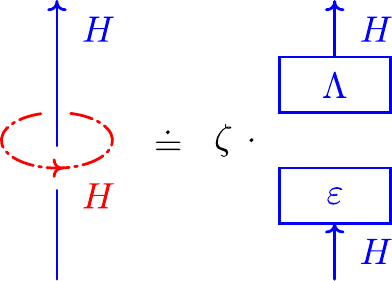}
 \caption{Cutting property for red meridians.}
 \label{cutting}
\end{figure}

\begin{proof}
 The morphism of $\calC$ obtained by applying the
 functor $F_{\lambda}$ to the $\calC$-colored bichrome graph represented in the left hand part of Figure \ref{cutting}
 is equal to $L_z$ for the non-trivial central element $z = \psi(\lambda) \in \rmZ(H)$. 
 We claim that $z = \zeta \cdot \Lambda$ for some $\zeta \in \Bbbk^*$. Indeed,
 Proposition \ref{P:slide} implies that the two $\calC$-colored bichrome graphs represented in Figure \ref{transparent} 
 are skein equivalent.
 
 \begin{figure}[htb]
  \centering
  \includegraphics{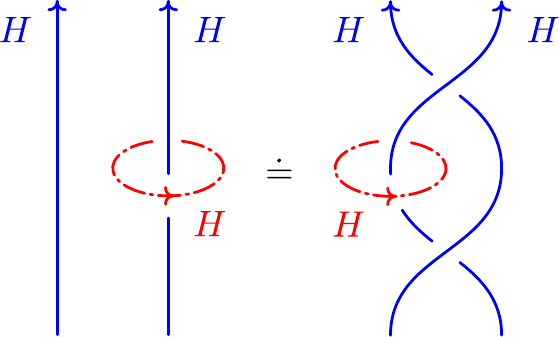}
  \caption{Transparency of $L_z$.}
  \label{transparent}
 \end{figure}

 The morphism of $\calC$ determined by the left-hand part of Figure \ref{transparent}
 maps $1_H \otimes 1_H$ to $1_H \otimes z$, while
 the one determined by the right-hand part of Figure \ref{transparent} maps $1_H \otimes 1_H$ to $\sum_{i,j=1}^r b_ja_i \otimes a_jb_iz$. 
 Now, since the Drinfeld map $\psi$ is an isomorphism, every element $x \in H$ can be written as
 $\sum_{i,j = 1}^r \psi^{-1}(x)(b_ja_i) \cdot a_jb_i$. Then 
 \[
  xz = \sum_{i,j = 1}^r \psi^{-1}(x)(b_ja_i) \cdot a_jb_iz = \psi^{-1}(x)(1_H) \cdot z.
 \]
 This means $z$ spans a 1-dimensional $H$-submodule in $H$, which has to coincide with the ideal of two-sided cointegrals.
\end{proof}

If $V$ is a projective object of $H$ then let us choose a section 
$s_V: V \rightarrow H \otimes V$
of the epimorphism
$\epsilon \otimes \id_V : H \otimes V \rightarrow V$, i.e. an $H$-module morphism satisfying 
$(\epsilon \otimes \id_V) \circ s_V = \id_V$. 
 We define now an operation on admissible $n$-string link graphs
which we call \textit{turning a red cycle blue}. Let $T$ be an $n$-string link graph containing an $H$-colored red cycle $C$ 
and a $V$-colored blue edge $e$ for some projective object $V$ of $\calC$. Then Figure \ref{red_turns_blue_1}
explains how to replace $C$ and $e$ with a set of blue edges and coupons to obtain a new $n$-string link graph
$T'$.

\begin{figure}[ht]
 \centering
 \includegraphics{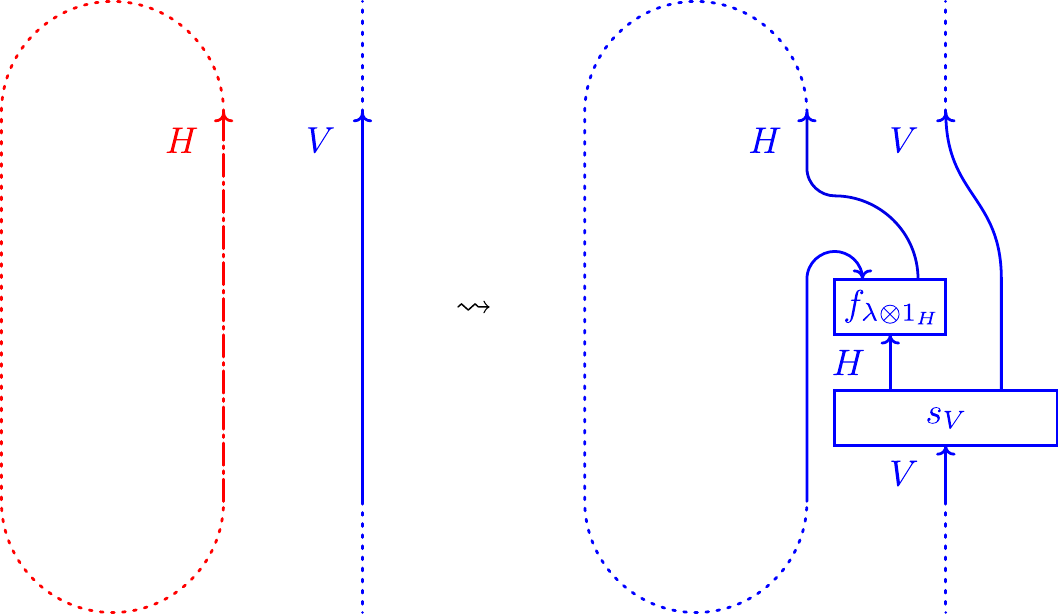}
 \caption{Turning a red cycle blue. The map $f_{\lambda \otimes 1_H}$ is the unique morphism of 
 $H$-modules sending the generator $1_H$ of $H$ to the vector $\lambda \otimes 1_H$ of $H^* \otimes H$.}
 \label{red_turns_blue_1}
\end{figure}

\begin{remark}
 When performing the operation we just described to a red cycle in a bichrome graph
 we have to be extremely careful with coupons. Indeed a direct
 switch of the color of some edges may not result in a bichrome graph.
 To be precise we have to replace coupons as shown in Figure \ref{red_turns_blue_2}.
\end{remark}

\begin{figure}[htb]
 \centering
 \includegraphics{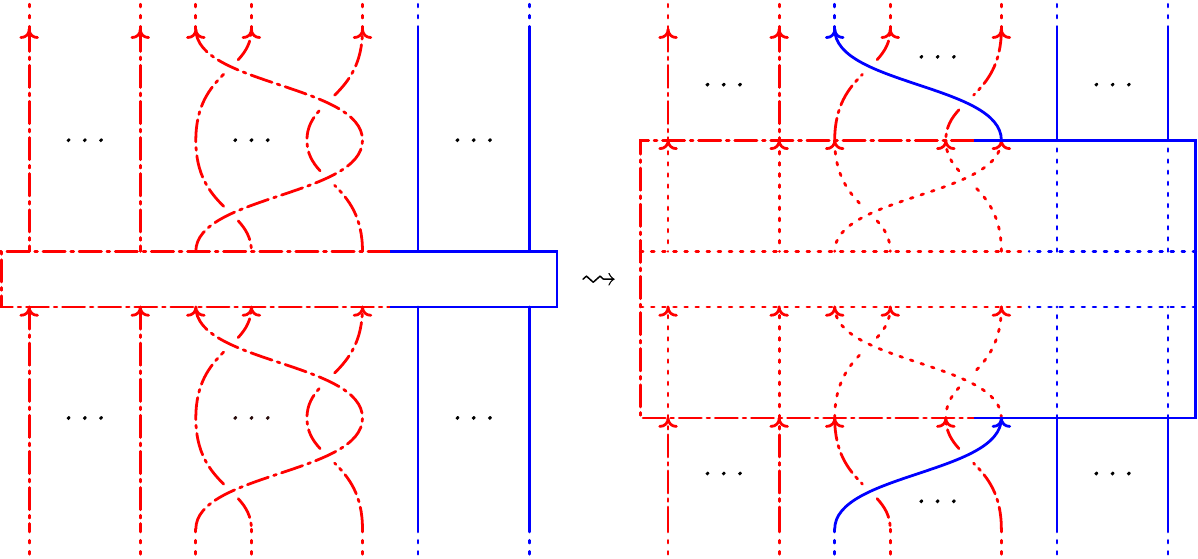}
 \caption{Recipe for replacing a coupon when turning a red cycle blue. The new coupon is to be colored with the image under
 $F_{\lambda}$ of the dashed graph it contains.}
 \label{red_turns_blue_2}
\end{figure}

\begin{lemma}\label{red_turns_blue_lemma}
 If $T$ is an admissible $n$-string link graph and $T'$ is obtained from $T$ by turning a red cycle blue then
 \[
  F_{\lambda}(T) = F_{\lambda}(T').
 \]
\end{lemma}

\begin{proof}
 The proof follows from the equality
 \[
  \lev_H \circ\ (\id_{H^*} \otimes\ L_x) \circ f_{\lambda \otimes 1_H} = \lambda(x) \cdot \epsilon,
 \]
 which holds for every $x \in H$. To establish the identity let us consider $y \in H$. Then
 \begin{align*}
  \lev_H((\id_{H^*} \otimes\ L_x)(f_{\lambda \otimes 1_H}(y))) &= 
  \ \lev_H((\id_{H^*} \otimes\ L_x)((y_{(1)} \cdot \lambda) \otimes y_{(2)})) \\
  &= \ \lev_H((y_{(1)} \cdot \lambda) \otimes xy_{(2)}) \\
  &= \lambda(S(y_{(1)})xy_{(2)}) \\
  &= \lambda(x)\epsilon(y). \qedhere
 \end{align*}
\end{proof}

We can now give a new easy proof of the non-degeneracy of $H$ using skein methods.

\begin{corollary}\label{non-degeneracy_factorizable_lemma}
 If $H$ is a finite-dimensional factorizable ribbon Hopf algebra then $\Hmod$
 satisfies the non-degeneracy condition $\Delta_- \Delta_+ = \zeta \neq 0$.
\end{corollary}

\begin{proof}
 Thanks to Lemma \ref{red_turns_blue_lemma} we have the skein equivalence of Figure \ref{non-degeneracy_factorizable}. Now the functor $F_{\lambda}$ maps the left-hand side of Figure \ref{non-degeneracy_factorizable} to $\Delta_- \Delta_+ \cdot \id_V$ because the red link is obtained by sliding a $+1$-framed unknot over a $-1$-framed unknot, while it maps the right-hand side of Figure \ref{non-degeneracy_factorizable} to $\zeta \cdot \id_V$. \qedhere
 
 \begin{figure}[ht]
  \centering
   \includegraphics{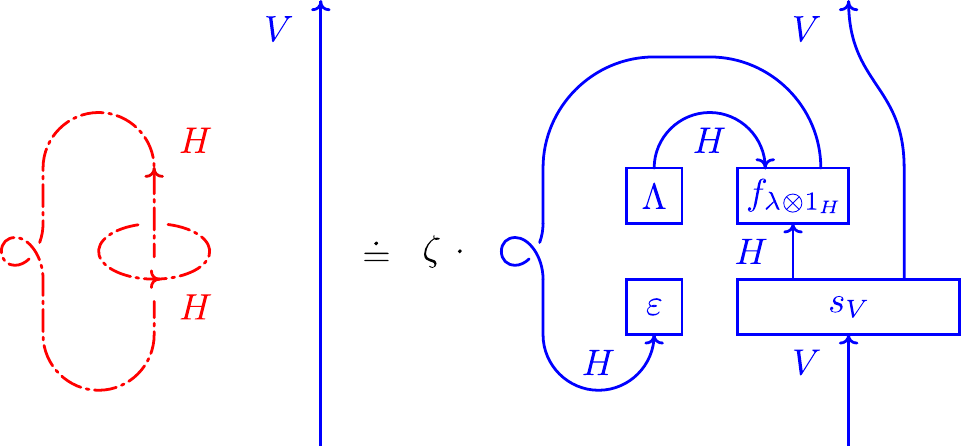}
   \caption{Skein equivalence witnessing $\Delta_- \Delta_+ = \zeta$.}
   \label{non-degeneracy_factorizable}
 \end{figure}
\end{proof}

\subsection{Cobordism category and universal construction}\label{admissible_cobordisms}

In this subsection we introduce the cobordism category we will work with and we apply the universal construction 
of \cite{BHMV95} to obtain a functorial extension of the invariant $\Hm$. To do this, we first
fix some terminology. At the beginning of Subsection
\ref{string_link_graphs} we introduced \textit{bichrome graphs}, which
are ribbon graphs with edges divided into two groups, red and blue,
and with special coupons satisfying
a certain condition concerning the partition of edges.  Now we need to
extend the definition to a more general setting.  A
\textit{$\calC$-colored blue set $P$ inside a surface $\Sigma$} is a
discrete set of blue points of $\Sigma$ endowed with orientations,
framings and colors given by objects of $\calC$.  A
\textit{$\calC$-colored bichrome graph $T$ inside a 3-dimensional
  cobordism $M$} is a $\calC$-colored bichrome graph embedded inside $M$
whose boundary vertices are given by 
$\calC$-colored blue sets inside the boundary of the cobordism.
We can now define the symmetric monoidal category $\Cob_{\calC}$.

An \textit{object $\bbSigma$ of $\Cob_{\calC}$} is a triple $(\Sigma,P,\calL)$ where: 
\begin{enumerate}
 \item $\Sigma$ is a closed surface;
 \item $P \subset \Sigma$ is a $\calC$-colored blue set;
 \item $\calL \subset H_1(\Sigma;\R)$ is a Lagrangian subspace.
\end{enumerate}

A \textit{morphism $\bbM : \bbSigma \rightarrow \bbSigma'$ of $\Cob_{\calC}$}
is an equivalence class of triples $(M,T,n)$ where:
\begin{enumerate}
 \item $M$ is a 3-dimensional cobordism from $\Sigma$ to $\Sigma'$;
 \item $T \subset M$ is a $\calC$-colored bichrome graph from $P$ to $P'$;
 \item $n \in \Z$ is a signature defect.
\end{enumerate}
Two triples $(M,T,n)$ and $(M',T',n')$ are equivalent if $n = n'$ and 
if there exists an isomorphism of cobordisms $f : M \rightarrow M'$ satisfying $f(T) = T'$.

The \textit{identity morphism $\id_{\bbSigma} : \bbSigma \rightarrow \bbSigma$ associated with an object
$\bbSigma = (\Sigma,P,\calL)$ of $\Cob_{\calC}$} is the equivalence class of the triple
\[
 (\Sigma \times I,P \times I,0).
\]

The \textit{composition $\bbM' \circ \bbM : \bbSigma \rightarrow \bbSigma''$ of morphisms 
$\bbM' : \bbSigma' \rightarrow \bbSigma''$, $\bbM : \bbSigma \rightarrow \bbSigma'$ of $\Cob_{\calC}$}
is the equivalence class of the triple
\[
 \left( M \cup_{\Sigma'} M', T \cup_{P'} T', n + n' - \mu(M_* \calL,\calL',M'^* \calL'') \right)
\]
for the Lagrangian subspaces
\begin{gather*}
 M_*\calL := \{ x' \in H_1(\Sigma';\R) \mid i_{M_+ *} x' \in i_{M_- *} (\calL) \} \subset H_1(\Sigma';\R) \\
 M'^*\calL'' := \{ x' \in H_1(\Sigma';\R) \mid i_{M'_- *} x' \in i_{M'_+ *}(\calL'') \} \subset H_1(\Sigma';\R).
\end{gather*}
where
\[
 i_{M_-} : \Sigma \hookrightarrow M, \quad i_{M_+} : \Sigma' \hookrightarrow M, \quad
 i_{M'_-} : \Sigma' \hookrightarrow M', \quad i_{M'_+} : \Sigma'' \hookrightarrow M'
\]
are the embeddings induced by the structure of the cobordisms $M$ and $M'$. 
Here $\mu$ denotes the Maslov index, see \cite{T94} for a detailed account of its properties.

The \textit{unit of $\Cob_{\calC}$} is the unique object whose surface is empty, and it will be denoted $\varnothing$.
The \textit{tensor product $\bbSigma \otimes \bbSigma'$ 
of objects $\bbSigma$, $\bbSigma'$ of $\Cob_{\calC}$} is 
the triple 
\[
 (\Sigma \sqcup \Sigma',P \sqcup P',\calL + \calL').
\]
The \textit{tensor product $\bbM \otimes \bbM' : \bbSigma \otimes \bbSigma' \rightarrow \bbSigma'' \otimes \bbSigma'''$ 
of morphisms $\bbM : \bbSigma \rightarrow \bbSigma''$, $\bbM' : \bbSigma' \rightarrow \bbSigma'''$ of $\Cob_{\calC}$} is 
the equivalence class of the triple
\[
 (M \sqcup M',T \sqcup T',n+n').
\]

We will now construct a TQFT extending the renormalized Hennings invariant $\Hm$. Its domain however will not be the whole 
symmetric monoidal category $\Cob_{\calC}$, as there is no way of defining $\Hm$ for every closed morphism of $\Cob_{\calC}$.
Indeed, we will have to consider a strictly smaller subcategory. We define 
$\adCob_{\calC}$ to be the symmetric monoidal subcategory of $\Cob_{\calC}$ having the same objects but featuring only morphisms 
$\bbM = (M,T,n)$ which satisfy the following condition: every connected component of $M$ disjoint from the incoming boundary contains an
admissible $\calC$-colored bichrome subgraph of $T$.

We can now extend the renormalized Hennings invariant to closed morphisms of $\adCob_{\calC}$ by setting
\[
 \Hm(\bbM) := \delta^n\Hm(M,T)
\]
for every closed connected morphism $\bbM = (M,T,n)$ and then by setting
\[
 \Hm(\bbM_1 \otimes \ldots \otimes \bbM_k) := \prod_{i=1}^k \Hm(\bbM_i)
\]
for every tensor product of closed connected morphisms $\bbM_1, \ldots, \bbM_k$.

\begin{remark}
 It is clear that this definition only works for closed morphisms of $\adCob_{\calC}$, as in general closed morphisms of 
 $\Cob_{\calC}$ do not feature admissible $\calC$-colored bichrome graphs.
\end{remark}

We apply now the universal construction of \cite{BHMV95}, which allows a functorial extension of $\Hm$.  
If $\bbSigma$ is an object of $\adCob_{\calC}$ then
let $\calV(\bbSigma)$ be the free vector space generated by the set of morphisms 
$\bbM_{\bbSigma} : \varnothing \rightarrow \bbSigma$ of $\adCob_{\calC}$, 
and let $\calV'(\bbSigma)$ be the free vector space generated by the set of morphisms 
$\bbM'_{\bbSigma} : \bbSigma \rightarrow \varnothing$ of $\adCob_{\calC}$.
Consider the bilinear form
\[
 \begin{array}{rccc}
  \langle \cdot , \cdot \rangle_{\bbSigma} : & \calV'(\bbSigma) \times \calV(\bbSigma) & \rightarrow & \Bbbk \\
  & (\bbM'_{\bbSigma},\bbM_{\bbSigma}) & \mapsto & \Hm(\bbM'_{\bbSigma} \circ \bbM_{\bbSigma}).
 \end{array}
\]

Let $\rmV_{\calC}(\bbSigma)$ be the quotient of the vector space 
$\calV(\bbSigma)$ with respect to the right radical 
of the bilinear form $\langle \cdot , \cdot \rangle_{\bbSigma}$, and similarly let $\rmV'_{\calC}(\bbSigma)$ be the 
quotient of the vector space 
$\calV'(\bbSigma)$ with respect to the left radical 
of the bilinear form $\langle \cdot , \cdot \rangle_{\bbSigma}$.
Then the pairing $\langle \cdot , \cdot \rangle_{\bbSigma}$ induces a non-degenerate pairing 
\[
 \langle \cdot , \cdot \rangle_{\bbSigma} :  \rmV'_{\calC}(\bbSigma) \otimes \rmV_{\calC}(\bbSigma) \rightarrow \Bbbk. 
\]

Now if $\bbM : \bbSigma \rightarrow \bbSigma'$ is a morphism of $\adCob_{\calC}$, then let 
$\rmV_{\calC}(\bbM)$ be the linear map defined by
\[
 \begin{array}{rccc}
  \rmV_{\calC}(\bbM) : & \rmV_{\calC}(\bbSigma) & \rightarrow & \rmV_{\calC}(\bbSigma') \\
  & [\bbM_{\bbSigma}] & \mapsto & [\bbM \circ \bbM_{\bbSigma}],
 \end{array}
\]
and similarly let $\rmV'_{\calC}(\bbM)$ be the linear map defined by
\[
 \begin{array}{rccc}
  \rmV'_{\calC}(\bbM) : & \rmV'_{\calC}(\bbSigma') & \rightarrow & \rmV'_{\calC}(\bbSigma) \\
   & [\bbM'_{\bbSigma'}] & \mapsto & [\bbM'_{\bbSigma'} \circ \bbM].
 \end{array}
\]

The construction we just provided clearly defines functors
\[
 \rmV_{\calC} : \adCob_{\calC} \rightarrow \Vect_{\Bbbk}, \quad \rmV'_{\calC} : \adCob_{\calC}^{\op} \rightarrow \Vect_{\Bbbk}.
\]

\begin{proposition}\label{lax_monoidality}
 The natural transformation $\mu : \otimes \circ \rmV_{\calC} \Rightarrow \rmV_{\calC} \circ \otimes$ 
 associating with every pair of objects
 $\bbSigma$, $\bbSigma'$ of $\adCob_{\calC}$ the linear map 
 \[
  \begin{array}{rccc}
   \mu_{\bbSigma,\bbSigma'} : & \rmV_{\calC}(\bbSigma) \otimes \rmV_{\calC}(\bbSigma') & \rightarrow & 
   \rmV_{\calC}(\bbSigma \otimes \bbSigma') \\
   & [\bbM_{\bbSigma}] \otimes [\bbM_{\bbSigma'}] & \mapsto & [\bbM_{\bbSigma} \otimes \bbM_{\bbSigma'}]
  \end{array}
 \]
 is monic.
\end{proposition}

\begin{proof}
 A trivial vector in $\rmV_{\calC}(\bbSigma \otimes \bbSigma')$ of the form $\sum_{i=1}^m \alpha_i \cdot [\bbM_{\bbSigma} \otimes \bbM_{\bbSigma'}]$ satisfies
 \[
  \sum_{i=1}^m \alpha_i \Hm(\bbM'_{\bbSigma \otimes \bbSigma'} \circ (\bbM_{\bbSigma} \otimes \bbM_{\bbSigma'})) = 0
 \]
 for every vector $[\bbM'_{\bbSigma \otimes \bbSigma'}]$ of $\rmV'_{\calC}(\bbSigma \otimes \bbSigma')$. In particular its pairing 
 with every vector of the form $[\bbM'_{\bbSigma} \otimes \bbM'_{\bbSigma'}]$ of $\rmV'_{\calC}(\bbSigma \otimes \bbSigma')$
 for some $[\bbM'_{\bbSigma}]$ in $\rmV'_{\calC}(\bbSigma)$ and for some $[\bbM'_{\bbSigma'}]$ in $\rmV'_{\calC}(\bbSigma')$ must be zero
 too. This means $\sum_{i=1}^m \alpha_i \cdot [\bbM_{\bbSigma}] \otimes [\bbM_{\bbSigma'}]$ is a trivial vector in $\rmV_{\calC}(\bbSigma) \otimes \rmV_{\calC}(\bbSigma')$.
\end{proof}

\subsection{Surgery axioms}\label{surgery_axioms_section}
In this subsection we study the behaviour of $\Hm$  
under decorated index $k$ surgery for $k=0,1,2$. In order to do this, we first introduce this topological operation.
For every $k = 0,1,2$ the \textit{index $k$ surgery surface} is 
the object $\bbSigma_k$ of $\Cob_{\calC}$ given by
\begin{gather*}
 \bbSigma_0 := \left( S^{-1} \times S^{3}, \varnothing, \{ 0 \} \right) = \varnothing, \\
 \bbSigma_1 := \left( S^{0} \times S^{2}, P_{\Sigma_1}, \{ 0 \} \right), \\
 \bbSigma_2 := \left( S^{1} \times S^{1}, \varnothing, \calL_{\Sigma_2} \right)
\end{gather*}
with the convention $S^{-1} := \varnothing$, where the $H$-colored blue ribbon set $P_{\Sigma_1}$ 
is given by $S^0 \times \{ (0,0,1) \}$ with orientation induced by $S^0$ and with framing obtained by pulling back a non-trivial 
tangent vector to $(0,0,1)$ along the projection onto the second factor of $S^{0} \times S^{2}$, 
and where the Lagrangian subspace $\calL_{\Sigma_2}$ is generated by the homology class of the curve $\{ (1,0) \} \times S^1$.

For every $k = 0,1,2$ the \textit{index $k$ attaching tube} is 
the morphism $\bbA_k : \varnothing \rightarrow \bbSigma_k$ of $\Cob_{\calC}$ given by
\begin{gather*}
 \bbA_0 := (S^{-1} \times \overline{D^{4}},\varnothing,0) = \id_{\varnothing}, \\
 \bbA_1 := (S^{0} \times D^{3},T_{A_1},0), \\
 \bbA_2 := (S^{1} \times \overline{D^{2}},K_{A_2},0)
\end{gather*}
where the $\calC$-colored blue ribbon graph $T_{A_1}$ is represented in Figure \ref{index_1_attaching} and
where the $H$-colored red knot $K_{A_2}$ is given by $S^1 \times \{ (0,0) \}$ with orientation induced by $S^1$
and with framing obtained by pulling back a non-trivial tangent vector to $(0,0)$ along the projection onto the second factor of 
$S^{1} \times \overline{D^{2}}$.

\begin{figure}[htb]
 \centering
 \includegraphics{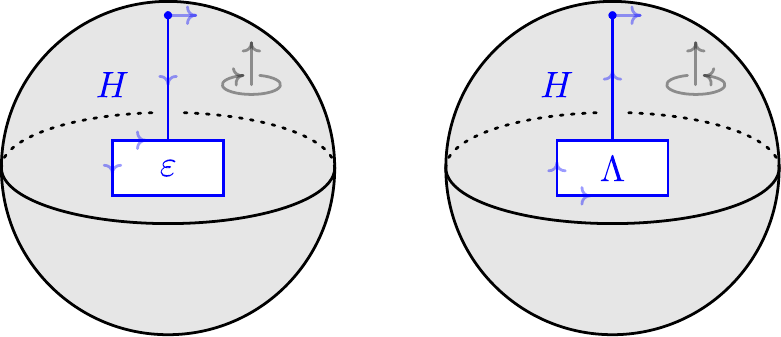}
 \caption{The morphism $\bbA_1$. Arrows on horizontal boundaries of coupons 
 are directed according to the orientations of bases, while arrows on vertical boundaries are directed from bottom bases to top bases.}
 \label{index_1_attaching}
\end{figure}

For every $k = 0,1,2$ the \textit{index $k$ belt tube} is 
the morphism $\bbB_k : \varnothing \rightarrow \bbSigma_k$ of $\Cob_{\calC}$ given by
\begin{gather*}
 \bbB_0 := (D^{0} \times S^{3},T_{B_0},0), \\
 \bbB_1 := (D^{1} \times S^{2},T_{B_1},0), \\
 \bbB_2 := (D^{2} \times S^{1},\varnothing,0)
\end{gather*}
with the convention $D^{0} := \{ 0 \}$, 
where the $\calC$-colored blue ribbon graph $T_{B_0}$ is represented in Figure \ref{index_0_belt},
where the $H$-colored blue tangle $T_{B_1}$ is given by $D^1 \times \{ (0,0,1) \}$ with orientation induced by 
$D^1$ and with framing obtained by pulling back a non-trivial 
tangent vector to $(0,0,1)$ along the projection onto the second factor of $D^{1} \times S^{2}$.

\begin{figure}[htb]
 \centering
 \includegraphics{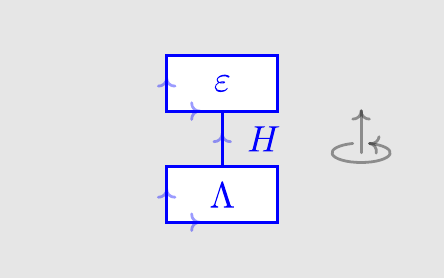}
 \caption{The morphism $\bbB_0$.}
 \label{index_0_belt}
\end{figure}

For $k \in \{0,1,2\}$ and for a morphism $\bbM_k : \bbSigma_k \rightarrow \varnothing$ of $\adCob_{\calC}$  
the morphism $\bbM_k \circ \bbB_k$ is said to be obtained from $\bbM_k \circ \bbA_k$ by an \emph{index $k$ surgery}.

\begin{proposition}\label{surgery_axioms_proposition}
 For $k \in \{0,1,2\}$ let $\bbM_k : \bbSigma_k \rightarrow \varnothing$ be a morphism of $\adCob_{\calC}$.  
 If $\bbM_k \circ \bbA_k$ is in $\adCob_{\calC}$ then
 \[
  \Hm(\bbM_k \circ \bbB_k) = \lambda_k \Hm(\bbM_k \circ \bbA_k)
 \]
 with $\lambda_0 = \lambda_1^{-1} = \lambda_2 = \calD^{-1}$.  
\end{proposition}

\begin{proof}
 If $k = 0$ then the property reduces to the computation
 \[
  \Hm(\bbB_0) = \calD^{-1 - 0} \delta^{0 - 0} F'_{\lambda}(T_{B_0}) = \calD^{-1} \rmt_H(\Lambda \circ \epsilon) = \calD^{-1}.
 \]
 
 If $k = 1$ then we have two cases, according to whether or not the surgery involves two different connected components of the closed morphism. Let us start from the first case, and let us begin by 
 decomposing $\bbSigma_1$ as a tensor product $\smash{\overline{\bbS^2_{(-,H)}}} \otimes \bbS^2_{(+,H)}$ and by decomposing $\bbA_1$ as a tensor product $\smash{\overline{\bbD^3_{\epsilon}}} \otimes \bbD^3_{\Lambda}$ with respect to the morphisms $\smash{\overline{\bbD^3_{\epsilon}} : \varnothing \rightarrow \overline{\bbS^2_{(-,H)}}}$ and $\bbD^3_{\Lambda} :  \varnothing \rightarrow \bbS^2_{(+,H)}$ represented in the left-hand part and in the right-hand part of Figure \ref{index_1_attaching} respectively.
 Let us consider connected morphisms $\bbM_1 : \smash{\overline{\bbS^2_{(-,H)}}} \rightarrow \varnothing$ and $\bbM'_1 : \bbS^2_{(+,H)} \rightarrow \varnothing$ of $\adCob_{\calC}$.  
 If $L = L_1 \cup \ldots \cup L_{\ell}$ is a surgery link for $\smash{M_1 \cup_{\overline{S^2}} \overline{D^3}}$ and 
 if $L' = L'_1 \cup \ldots \cup L'_{\ell'}$ is a surgery link for $M'_1 \cup_{S^2} D^3$ then
 \begin{gather*}
  \Hm(\bbM_1 \circ \overline{\bbD^3_{\epsilon}}) = \calD^{-1-\ell} \delta^{n - \sigma(L)} F'_{\lambda}(L \cup T), \\
  \Hm(\bbM'_1 \circ \bbD^3_{\Lambda}) = \calD^{-1-\ell'} \delta^{n' - \sigma(L')} F'_{\lambda}(L' \cup T').
 \end{gather*}
 If $(L \cup T)_H : \varnothing \rightarrow (+,H)$ and $(L' \cup T')_H : (+,H) \rightarrow \varnothing$ are morphisms of
 $\calR_{\lambda}$ satisfying 
 \begin{gather*}
  L \cup T = T_{\epsilon} \circ (L \cup T)_H, \\
  L' \cup T' = (L' \cup T')_H \circ T_{\Lambda},
 \end{gather*}
 for the elementary morphisms $T_{\epsilon} : (+,H) \rightarrow \varnothing$ and $T_{\Lambda} : \varnothing \rightarrow (+,H)$ of
 $\calR_{\lambda}$ featuring a single blue strand and a single blue coupon with colors specified by the subscripts, then $(L \cup T)_H \circ T_{\epsilon}$ is a cutting presentation for $L \cup T$, and $T_{\Lambda} \circ (L' \cup T')_H$ is a cutting presentation for $L' \cup T'$. This means
 \begin{gather*}
  F'_{\lambda}(L \cup T) = \rmt_H \left( F_{\lambda} \left( (L \cup T)_H \circ T_{\epsilon} \right) \right), \\
  F'_{\lambda}(L' \cup T') = \rmt_H \left( F_{\lambda} \left( T_{\Lambda} \circ (L' \cup T')_H \right) \right).
 \end{gather*}
 Now we remark that 
 \[
  \Hm((\bbM_1 \otimes \bbM'_1) \circ \bbB_1) = \calD^{-1-\ell-\ell'} \delta^{n + n' - \sigma(L) - \sigma(L')} 
  F'_{\lambda} \left( (L' \cup T')_H \circ (L \cup T)_H \right),
 \]
 and that 
 \begin{gather*}
  F_{\lambda} \left( (L \cup T)_H \right) = \rmt_H \left( F_{\lambda} \left( (L \cup T)_H \circ T_{\epsilon} \right) \right) \cdot \Lambda, \\
  F_{\lambda} \left( (L' \cup T')_H \right) = \rmt_H \left( F_{\lambda} \left( T_{\Lambda} \circ (L' \cup T')_H \right) \right) \cdot \varepsilon,
 \end{gather*}
 because $\Hom_{\calC}(\one,H)$ and $\Hom_{\calC}(H,\one)$ are 1-dimensional. This means
 \begin{align*}
  F'_{\lambda} &\left( (L' \cup T')_H\circ (L \cup T)_H \right) \\
  &= \rmt_H \left( F_{\lambda} \left( (L \cup T)_H \circ (L' \cup T')_H \right) \right) \\
  &= \rmt_H \left( F_{\lambda} \left( (L \cup T)_H \circ T_{\epsilon} \right) \right) \
  \rmt_H \left( F_{\lambda} \left( T_{\Lambda} \circ (L' \cup T')_H \right) \right) \
  \rmt_H \left( \Lambda \circ \varepsilon \right) \\
  &= F'_{\lambda}(L \cup T) \ F'_{\lambda}(L' \cup T').
 \end{align*}
 But now
 \begin{align*}
  \Hm((\bbM_1 \otimes \bbM'_1) \circ \bbB_1) &= \calD^{-1-\ell-\ell'} \delta^{n + n' - \sigma(L) - \sigma(L')} 
  F'_{\lambda}((L \cup T)_H \circ (L' \cup T')_H) \\
  &= \calD^{-1-\ell-\ell'} \delta^{n + n' - \sigma(L) - \sigma(L')} 
  F'_{\lambda}(L \cup T) \ F'_{\lambda}(L' \cup T') \\
  &= \calD \ \Hm(\bbM_1 \circ\overline{\bbD^3_{\epsilon}}) \ \Hm(\bbM'_1 \circ \bbD^3_{\Lambda}) \\
  &= \calD \ \Hm((\bbM_1 \otimes \bbM'_1) \circ \bbA_1).
 \end{align*}

 Now let us move on to the second case, and let us consider a connected morphism $\bbM_1 : \bbSigma_1 \rightarrow \varnothing$ of $\adCob_{\calC}$. If $L = L_1 \cup \ldots \cup L_{\ell}$ is a surgery link for $M_1 \cup_{(S^{0} \times S^{2})} (S^{0} \times D^{3})$
 then
 \[
  \Hm(\bbM_1 \circ \bbSigma_1) = \calD^{-1-\ell} \delta^{n - \sigma(L)} F'_{\lambda}(L \cup T).
 \]
 If $(L \cup T)_H : (+,H) \rightarrow (+,H)$ is a morphism of 
 $\calR_{\lambda}$ satisfying 
 \[
  L \cup T = T_{\epsilon} \circ (L \cup T)_H \circ T_{\Lambda},
 \]
 then $T_{\Lambda} \circ T_{\epsilon} \circ (L \cup T)_H$ is a cutting presentation for $L \cup T$.
 This means
 \[
  F'_{\lambda}(L \cup T) = \rmt_H \left( F_{\lambda} \left( T_{\Lambda} \circ T_{\epsilon} \circ (L \cup T)_H \right) \right).
 \]
 Now remark that
 \[
  \Hm(\bbM_1 \circ \bbB_1) = \calD^{-1-(\ell+1)} \delta^{n - \sigma(L)} F'_{\lambda}(K \cup L \cup \hat{T})
 \]
 where the admissible $\calC$-colored bichrome graph $K \cup L \cup \hat{T}$ is represented in Figure \ref{F:proof_1-surgery}.
 
 \begin{figure}[ht]
  \centering
  \includegraphics{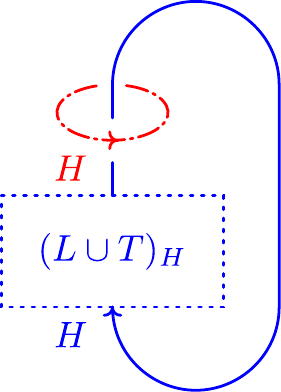}
  \caption{The admissible $\calC$-colored bichrome graph $K \cup L \cup \hat{T}$.}
  \label{F:proof_1-surgery}
 \end{figure}
 
 But now, thanks to Lemma \ref{cutting_lemma}, we have
 \[
  F'_{\lambda} \left( K \cup L \cup \hat{T} \right) = 
  \zeta \ F'_{\lambda}(L \cup T).
 \]
 This means
 \begin{align*}
  \Hm(\bbM_1 \circ \bbB_1) &= \calD^{-2-\ell} \delta^{n - \sigma(L)} F'_{\lambda}(K \cup L \cup \hat{T}) \\
  &= \zeta \calD^{-2-\ell} \delta^{n - \sigma(L)} F'_{\lambda}(L \cup T) \\
  &= \calD^{-\ell} \delta^{n - \sigma(L)} F'_{\lambda}(L \cup T) \\
  &= \calD \ \Hm(\bbM_1 \circ \bbA_1).
 \end{align*}

 If $k = 2$ let us consider a connected morphism
 $\bbM_2 : \bbSigma_2 \rightarrow \varnothing$ of $\adCob_{\calC}$.
 If $L = L_1 \cup \ldots \cup L_{\ell}$ is a surgery link for $M_2 \cup_{(S^1 \times S^1)} (S^{1} \times \overline{D^{2}})$ 
 then $L \cup K_2$ is a surgery link for $M_2 \cup_{(S^1 \times S^1)} (D^{2} \times S^{1})$, where $K_2$ denotes the pull back 
 of the $H$-colored red knot coming from $\bbA_2$ to $S^3$. Now if the signature defect of $\bbM_2 \circ \bbA_2$ is $n$, then the 
 signature defect of $\bbM_2 \circ \bbA_2$ is $n + \sigma(L \cup K_2) - \sigma(L)$. Therefore
 \begin{gather*}
  \Hm(\bbM_2 \circ \bbA_2) = \calD^{-1 - \ell} \delta^{n - \sigma(L)} F'_{\lambda}(L \cup K_2), \\
  \Hm(\bbM_2 \circ \bbB_2) = \calD^{-1 - (\ell + 1)} \delta^{(n + \sigma(L \cup K_2) - \sigma(L)) - 
  \sigma(L \cup K_2)} F'_{\lambda}(L \cup K_2).
 \end{gather*}
\end{proof}

\subsection{Consequences of skein equivalence and surgery axioms} 

In this subsection we establish some useful properties of the functor $\rmV_{\calC}$ which will be used for 
the proof of its monoidality and for the computation of its image. Loosely speaking, they can be summarized as follows:
 \begin{enumerate}
\item If $\bbSigma$ is an object of $\adCob_{\calC}$ then $\rmV_{\calC}(\bbSigma)$ is generated by graphs inside a fixed connected cobordism;
 \item To test if a vector $[\bbM]$ in $\rmV_{\calC}(\bbSigma)$ is trivial it is enough to pair it with all covectors in 
  $\rmV'_{\calC}(\bbSigma)$ whose support is given by a fixed connected cobordism.
\end{enumerate}

For every $k \geq 0$ let us consider a standard embedding
$f_k : D^3 \hookrightarrow \R^2 \times I$ mapping the point
$(\cos(\frac{t}{k+1} \pi),0,\sin(\frac{t}{k+1} \pi)) \in D^3$ to the
point $((t,0),1) \in \R^2 \times I$ 
for every $t \in [0,k+1]$.  
Then if
$(\underline{\epsilon},\underline{V}) =
((\epsilon_1,V_1),\ldots,(\epsilon_k,V_k))$
is an object of $\calR_{\calC}$ we can use the embedding $f_k$ to
define by pull back a standard $\calC$-colored blue set
$P_{(\underline{\epsilon},\underline{V})}$ inside $S^2$.  Let
$\bbS^2_{(\underline{\epsilon},\underline{V})}$ denote the object of
$\adCob_{\calC}$ given by
\[
 (S^2,P_{(\underline{\epsilon},\underline{V})},\{ 0 \}).
\]
We can now generalize the notion of skein equivalence we gave in Subsection \ref{pairing_section} for morphisms of
$[n]\calR_{\lambda}$.
Indeed, we say two formal linear combinations $\sum_{i=1}^m \alpha_i \cdot T_i$ and $\sum_{i'=1}^{m'} \alpha'_{i'} \cdot T'_{i'}$
of $\calC$-colored bichrome graphs inside $D^3$ from $\varnothing$ to $P_{(\underline{\epsilon},\underline{V})}$
are \textit{skein equivalent} if 
\[
 \sum_{i=1}^m \alpha_i \cdot f_k(T_i) \doteq \sum_{i'=1}^{m'} \alpha'_{i'} \cdot f_k(T'_{i'})
\]
in 
$\Hom_{\calR_{\lambda}}([0]\varnothing,[0](\underline{\epsilon},\underline{V}))$.
Now let $\bbSigma = (\Sigma,P,\calL)$ be an object of $\adCob_{\calC}$, let 
$M$ be a connected 3-dimensional cobordism from $\varnothing$ to $\Sigma$ and let us fix 
an isomorphism of cobordisms $f_M : M \rightarrow D^3 \cup_{S^2} \hat{M}$ for some cobordism $\hat{M}$ 
from $S^2$ to $\Sigma$. 
In general, we say two linear combinations 
of $\calC$-colored bichrome graphs inside $M$ from $\varnothing$ to $P$
are \textit{skein equivalent} if, up to isotopy, their images under $f_M$
are of the form $\sum_{i=1}^m \alpha_i \cdot ( T_i \cup \hat{T} )$ and $\sum_{i'=1}^{m'} \alpha'_{i'} \cdot ( T'_{i'} \cup \hat{T})$
for some object $(\underline{\epsilon},\underline{V})$ of $\calR_{\calC}$,
for some $\calC$-colored bichrome graph $\hat{T}$ inside $\hat{M}$ from $P_{(\underline{\epsilon},\underline{V})}$ to $P$, and for some
skein equivalent linear combinations
\[
 \sum_{i=1}^m \alpha_i \cdot T_i \doteq \sum_{i'=1}^{m'} \alpha'_{i'} \cdot T'_{i'}
\]
of $\calC$-colored bichrome graphs inside $D^3$ from $\varnothing$ to $P_{(\underline{\epsilon},\underline{V})}$.

If $\bbSigma = (\Sigma,P,\calL)$ is an object of $\adCob_{\calC}$ and 
$M$ is a connected 3-dimensional cobordism from $\varnothing$ to $\Sigma$,
we denote with $\calV(M;\bbSigma)$ the vector space generated by isotopy classes of admissible
$\calC$-colored bichrome graphs inside $M$ from $\varnothing$ to $P$.

\begin{proposition}\label{connection+skein}
 If $\bbSigma = (\Sigma,P,\calL)$ is an object of $\adCob_{\calC}$ and if $M$ is a connected 
 3-dimensional cobordism from $\varnothing$ to $\Sigma$ then the linear map 
 \[
  \begin{array}{rccc}
   \pi_{\bbSigma} : & \calV(M;\bbSigma) & \rightarrow & \rmV_{\calC}(\bbSigma) \\
   & T & \mapsto & [M,T,0]
  \end{array}
 \]
 is surjective, and skein equivalent vectors of $\calV(M;\bbSigma)$ have the same image in $\rmV_{\calC}(\bbSigma)$.
\end{proposition}

\begin{proof}
 First of all we remark that if we have a skein equivalence
 \[
  \sum_{i=1}^m \alpha_i \cdot T_i \doteq \sum_{i'=1}^{m'} \alpha'_{i'} \cdot T'_{i'}
 \]
 between vectors of $\calV(M;\bbSigma)$, then 
 \[
  \sum_{i=1}^m \alpha_i \Hm( \bbM'_{\bbSigma} \circ (M,T_i,0) ) = 
  \sum_{i'=1}^{m'} \alpha'_{i'} \Hm( \bbM'_{\bbSigma} \circ (M,T'_{i'},0) )
 \]
 for every morphism $\bbM'_{\bbSigma} : \bbSigma \rightarrow \varnothing$ of $\adCob_{\calC}$.
 This follows directly from the very definition of $\Hm$ in terms of the Hennings-Reshetikhin-Turaev functor $F_{\lambda}$.
 Therefore skein equivalent vectors of $\calV(M;\bbSigma)$ have the same image in $\rmV_{\calC}(\bbSigma)$.
 
 What we just proved implies in particular, up to skein equivalence, we can assume every connected component of every vector in $\calV(\bbSigma)$ features an $\epsilon$-colored coupon, or a $\Lambda$-colored coupon, or both. In order to show this, the idea is to use the properties of projective objects of $\calC$. Indeed, if $V$ is a projective $H$-module, then we can always find a section $s_V: V \rightarrow H \otimes V$ for the epimorphism $\epsilon \otimes \id_V : H \otimes V \rightarrow V$, i.e. an $H$-module morphism satisfying $(\epsilon \otimes \id_V) \circ s_V = \id_V$, just like we did for turning red components blue in Section \ref{S:skein_equivalence}. Remark that, thanks to the pivotal structure of $\calC$, projective $H$-modules are also injective, and we can always find a retraction $r_V: H \otimes V \rightarrow V$ for the monomorphism $\Lambda \otimes \id_V : V \rightarrow H \otimes V$, i.e. an $H$-module morphism satisfying $r_V \circ (\Lambda \otimes \id_V) = \id_V$. This means that every time a vector of $\calV(\bbSigma)$ features a blue edge colored with some projective object $V$, we can replace a small portion of it with one of the $\calC$-colored bichrome graphs represented in Figure \ref{projective_trick} without altering the vector in the 
 quotient $\rmV_{\calC}(\bbSigma)$. We call this operation \textit{projective trick}, and we will use it in the following argument.

 \begin{figure}[hbt]
  \centering
  \includegraphics{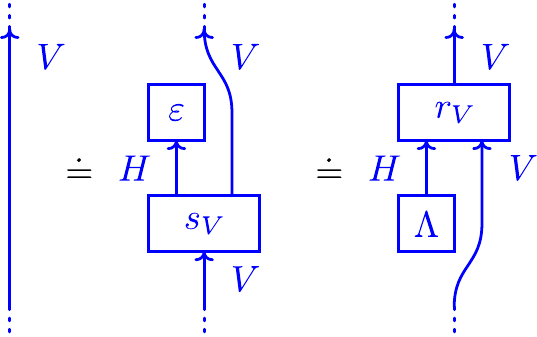}
  \caption{Projective trick along a $V$-colored blue edge.}
  \label{projective_trick}
 \end{figure}

 Now, in order to prove that $\pi_{\bbSigma}$ is surjective, we have to show that for every vector
 $[M_{\Sigma},T,n] \in \rmV_{\calC}(\bbSigma)$
 there exist admissible $\calC$-colored bichrome graphs $T_1, \ldots, T_m \subset M$ 
 and coefficients $\alpha_1, \ldots, \alpha_m \in \Bbbk$ such that
 \[
  \sum_{i = 1}^m \alpha_i \cdot [M,T_i,0] = [M_{\Sigma},T,n].
 \]
 We do this in two steps.
 First, we can assume that $M_{\Sigma}$ is connected: 
 indeed every time we have distinct connected components we can suppose, 
 up to skein equivalence, one of them contains an $\epsilon$-colored coupon while the other one contains
 a $\Lambda$-colored coupon. Then, thanks to Proposition \ref{surgery_axioms_proposition},
 the 1-surgery connecting them will determine a vector of $\rmV_{\calC}(\bbSigma)$ which is a non-zero scalar multiple of 
 $[M_{\Sigma},T,n]$. Second, assuming now $M_{\Sigma}$ is connected, we know there exists 
 a surgery link $L = L_1 \cup \ldots \cup L_{\ell}$ for $M_{\Sigma}$ inside $M$. Then,
 thanks to Proposition \ref{surgery_axioms_proposition} with $k = 2$, there exists some signature defect $n' \in \Z$ such that 
 \[
  [M_{\Sigma},T,n] = \lambda_2^{\ell} \cdot [M,L \cup T,n'] = \lambda_2^{\ell} \delta^{n'} \cdot 
  [M,L \cup T,0]
 \]
 where, once again, we adopt a slightly abusive notation for the pull back of the $\calC$-colored bichrome graph $T$ along 
 the embedding of the exterior of $L$ into $M_{\Sigma}$.
\end{proof}

If $\bbSigma = (\Sigma,P,\calL)$ is an object of $\adCob_{\calC}$ and 
$M'$ is a connected 3-dimensional cobordism from $\Sigma$ to $\varnothing$
then we denote with $\calV'(M';\bbSigma)$ the vector space generated by isotopy classes of
$\calC$-colored bichrome graphs inside $M'$ from $P$ to $\varnothing$.

\begin{proposition}\label{triviality}
 If $\bbSigma = (\Sigma,P,\calL)$ is a connected object of $\adCob_{\calC}$ and if $M'$ is a 
 connected 3-dimensional cobordism from $\Sigma$ to $\varnothing$ then a vector 
 $\sum_{i=1}^m \alpha_i \cdot [\bbM_{i,\bbSigma}]$ in $\rmV_{\calC}(\bbSigma)$ is trivial if and only if
 \[
  \sum_{i=1}^m \alpha_i \Hm\left( (M',T',0) \circ \bbM_{i,\bbSigma} \right) = 0
 \]
 for every $T'$ in $\calV'(M';\bbSigma)$.
\end{proposition}

\begin{proof}
 Connected morphisms of $\adCob_{\calC}$ from $\bbSigma$ to $\varnothing$ 
 are sufficient in order to detect non-triviality of vectors of $\rmV_{\calC}(\bbSigma)$ because
 $\Hm$ is multiplicative with respect to disjoint union. Then, just like in the proof of 
 Proposition \ref{connection+skein}, we can trade index 2 surgery for red links inside $M'$. 
\end{proof}

\subsection{Monoidality}

We use the results of the previous two subsections in order to prove that the functor
$\rmV_{\calC} : \adCob_{\calC} \rightarrow \Vect_{\Bbbk}$ is a TQFT.

\begin{theorem}
 The natural transformation $\mu : \otimes \circ \rmV_{\calC} \Rightarrow \rmV_{\calC} \circ \otimes$ 
 associating with every pair of objects
 $\bbSigma$, $\bbSigma'$ of $\adCob_{\calC}$ the linear map 
 \[
  \begin{array}{rccc}
   \mu_{\bbSigma,\bbSigma'} : & \rmV_{\calC}(\bbSigma) \otimes \rmV_{\calC}(\bbSigma') & \rightarrow & 
   \rmV_{\calC}(\bbSigma \otimes \bbSigma') \\
   & [\bbM_{\bbSigma}] \otimes [\bbM_{\bbSigma'}] & \mapsto & [\bbM_{\bbSigma} \otimes \bbM_{\bbSigma'}]
  \end{array}
 \]
 is an isomorphism.
\end{theorem}

\begin{proof}
 Thanks to Proposition \ref{lax_monoidality} we just need to prove that $\mu_{\bbSigma,\bbSigma'}$ is surjective for every
 pair of objects $\bbSigma$, $\bbSigma'$ of $\adCob_{\calC}$.
 Let $M_{\Sigma}$ be a connected cobordism from $\overline{S^2}$ to $\Sigma$ and
 let $M_{\Sigma'}$ be a connected cobordism from $S^2$ to $\Sigma'$.
 Thanks to Proposition \ref{connection+skein} we know $\rmV_{\calC}(\bbSigma \otimes \bbSigma')$ is generated by vectors of the form
 \[
  [(D^{1} \times S^{2}) \cup_{S^0 \times S^2} (M_{\Sigma} \sqcup M_{\Sigma'}),T,0]
 \]
 with $T$ a $\calC$-colored bichrome graph inside $(D^{1} \times S^{2}) \cup_{S^0 \times S^2} (M_{\Sigma} \sqcup M_{\Sigma'})$
 from $\varnothing$ to $P \sqcup P'$. Let us choose such a $T$ and let us show that the corresponding vector of
 $\rmV_{\calC}(\bbSigma \otimes \bbSigma')$ lies in the image of $\mu_{\bbSigma,\bbSigma'}$.  
 Thanks to Lemma \ref{red_turns_blue_lemma} we can suppose $D^{1} \times S^{2}$
 intersects only blue edges of $T$. Up to isotopy we can furthermore suppose $D^{1} \times S^{2}$
 intersects a projective edge of $T$. Then, up to skein equivalence, $D^{1} \times S^{2}$ intersects
 a single edge whose color $V$ is a projective object of $\calC$. 
 Therefore, there exist morphisms $f_1, \ldots, f_m$ in $\Hom_{\calC}(V,H)$ and
 $f'_1, \ldots, f'_m$ in $\Hom_{\calC}(H,V)$ satisfying
 \[
  \id_V = \sum_{i=1}^m f'_i \circ f_i.
 \]
 Indeed, $H$ splits as a direct sum with multiplicity of all the indecomposable projective modules of $\calC$. 
 This means that, up to skein equivalence,
 \[
  [(D^{1} \times S^{2}) \cup_{(S^0 \times S^2)} (M_{\Sigma} \sqcup M_{\Sigma'}),T,0]
  = \sum_{i=1}^m \left[ \left( (M_{\Sigma},T_i,0) \otimes (M_{\Sigma'},T'_i,0) \right) \circ \bbB_1 \right]
 \]
 for the index 1 belt tube $\bbB_1 : \varnothing \rightarrow \bbSigma_1$ introduced in Subsection \ref{surgery_axioms_section}.
 But Proposition \ref{surgery_axioms_proposition} with $k = 1$ yields the equality $[\bbB_1] = \calD \cdot [\bbA_1]$
 between vectors of $\rmV_{\calC}(\bbSigma_1)$, where $\bbA_1 : \varnothing \rightarrow \bbSigma_1$ is the 
 the index 1 attaching tube introduced in Subsection \ref{surgery_axioms_section}. Then we have the chain of equalities 
 \begin{align*}
  &\sum_{i=1}^m \left[ \left( (M_{\Sigma},T_i,0) \otimes (M_{\Sigma'},T'_i,0) \right) \circ \bbB_1 \right] \\
  = &\sum_{i=1}^m \alpha_i \calD \cdot \left[ \left( (M_{\Sigma},T_i,0) \otimes (M_{\Sigma'},T'_i,0) \right) \circ \bbA_1 \right] \\
  = &\sum_{i=1}^m \alpha_i \calD \cdot \left[ \left( (M_{\Sigma},T_i,0) \circ \overline{\bbD^3_{\epsilon}} \right) \otimes 
  \left( (M_{\Sigma'},T'_i,0) \circ \bbD^3_{\Lambda} \right) \right] \\
  = &\sum_{i=1}^m \alpha_i \calD \cdot \mu_{\bbSigma,\bbSigma'} 
  \left( \left[ (M_{\Sigma},T_i,0) \circ \overline{\bbD^3_{\epsilon}} \right] \otimes 
  \left[ (M_{\Sigma'},T'_i,0) \circ \bbD^3_{\Lambda} \right] \right)
 \end{align*}
 for the morphisms
 $\overline{\bbD^3_{\epsilon}} : \varnothing \rightarrow \overline{\bbS^2_{(-,H)}}$ and 
 $\bbD^3_{\Lambda} :  \varnothing \rightarrow \bbS^2_{(+,H)}$ of $\adCob_{\calC}$ represented in the left-hand part and in the 
 right-hand part of Figure \ref{index_1_attaching} respectively.
\end{proof}

\begin{remark}\label{R:Verlinde}
 As a consequence of monoidality we get a kind of Verlinde formula for dualizable surfaces: 
 if $\bbSigma = (\Sigma,P,\calL)$ is an object of $\adCob_{\calC}$ then we denote with $\overline{\bbSigma} = (\overline{\Sigma},\overline{P},\calL)$ the object obtained from $\bbSigma$ by reversing the orientation of both $\Sigma$ and $P$. If $P$ contains a point with projective color in every connected component of $\Sigma$ then $\bbSigma^* = \overline{\bbSigma}$. Duality morphisms are given by cylinders, with $\lev_{\bbSigma} : \bbSigma^* \otimes \bbSigma \rightarrow \varnothing$ and $\lcoev_{\bbSigma} : \varnothing \rightarrow \bbSigma \otimes \bbSigma^*$ both realized by the same decorated 3-manifold realizing the identity $\id_{\overline{\bbSigma}} : \overline{\bbSigma} \rightarrow \overline{\bbSigma}$, although seen as different cobordisms. Furthermore, the braiding morphism $c_{\bbSigma,\bbSigma^*} : \bbSigma \otimes \bbSigma^* \rightarrow \bbSigma^* \otimes \bbSigma$ is realized by the same decorated 3-manifold realizing the identity $\id_{\bbSigma \otimes \overline{\bbSigma}} : \bbSigma \otimes \overline{\bbSigma} \rightarrow \bbSigma \otimes \overline{\bbSigma}$ seen as a different cobordism. Therefore, if we set $\bbS^1 \times \bbSigma := \lev_{\bbSigma} \circ \ c_{\bbSigma,\bbSigma^*} \circ \lcoev_{\bbSigma}$, we get
 \begin{align*}
  \rmH'_{\calC}(\bbS^1 \times \bbSigma) &= \rmV_{\calC} \left( \lev_{\bbSigma} \circ \ c_{\bbSigma,\bbSigma^*} \circ \lcoev_{\bbSigma} \right) = \rev_{\rmV_{\calC}(\bbSigma)} \circ \ \tau \ \circ \lcoev_{\rmV_{\calC}(\bbSigma)} \\ &= \dim_{\Bbbk} \left( \rmV_{\calC}(\bbSigma) \right),
 \end{align*}
 where $\tau([\bbM_{\bbSigma}] \otimes [\bbM_{\bbSigma^*}]) := [\bbM_{\bbSigma^*}] \otimes [\bbM_{\bbSigma}]$ for every $[\bbM_{\bbSigma}] \otimes [\bbM_{\bbSigma^*}] \in \rmV_{\calC}(\bbSigma) \otimes \rmV_{\calC}(\bbSigma^*)$.
\end{remark}

\subsection{Identification of TQFT spaces}\label{S:identification_TQFT_spaces}
In this subsection we show that TQFT vector spaces can be identified with the algebraic vector spaces defined 
in Subsection \ref{pairing_section}. 
Indeed, recall that we introduced for every $g \in \N$ and for every object $V$ of $\calC$ the spaces 
\begin{gather*}
 \tilde{\calX}_{g,V} = \Hom_\calC(H,\X^{\otimes g} \otimes V), \quad
 \calX'_{g,V} = \Hom_\calC((\X^*)^{\otimes g} \otimes V,\one), \\
 \tilde{\calS}_{g,V} = \Hom_{[g]\calC}([g]H,[g]V), \quad
 \calS'_{g,V} = \Hom_{[g]\calC}([g]V,[g]\one),
\end{gather*}
as well as the quotient $\calX_{g,V}$ of $\tilde{\calX}_{g,V}$ given by the right radical of the bilinear pairing $\brk{\cdot,\cdot}_{\calX}$, and
the quotient $\calS_{g,V}$ of $\tilde{\calS}_{g,V}$ given by the right radical of the bilinear pairing $\brk{\cdot,\cdot}_{\calS}$. 
In Proposition \ref{P:HomG} we gave explicit isomorphisms
\[
 \tilde{\calX}_{g,V} \cong \tilde{\calS}_{g,V}, \quad
 \calX'_{g,V} \cong \calS'_{g,V},
\]
which, thanks to Proposition \ref{P:translation_of_pairings}, also induce explicit isomorphisms
\[
 \calX_{g,V} \cong \calS_{g,V}.
\]

Let us consider a genus $g$ Heegaard splitting $M_g \cup_{\Sigma_g} M'_g$ 
of $S^3$. Let $P_{V}$ denote a $\calC$-colored blue set 
inside $\Sigma_g$ composed of a single point with positive orientation and color $V$ and let $\overline{P}_H$
denote another 
$\calC$-colored blue set inside $\Sigma_g$ composed of a single point with negative orientation and color $H$.
Let $\calL_g$ denote the Lagrangian subspace of $H_1(\Sigma_g;\R)$
given by the kernel of the inclusion of $\Sigma_g$ into $M_g$. 
We denote with $\bbSigma_{g,V}$ the object $(\Sigma_g,P_{V},\calL_g)$
of $\adCob_{\calC}$ and with $\bbSigma_{g,V,H}$ the object $(\Sigma_g,P_{V} \cup \overline{P}_{H},\calL_g)$
of $\adCob_{\calC}$. 
Then the quotient of $\calV(M_g;\bbSigma_{g,V})$ with respect to the right radical of the bilinear form 
\[
 \begin{array}{rccc}
  \langle \cdot , \cdot \rangle_{S^3} : & 
  \calV'(M'_g;\bbSigma_{g,V}) \times 
  \calV(M_g;\bbSigma_{g,V}) & \rightarrow & \Bbbk \\
  & (T',T) & \mapsto & F'_{\lambda}(T \cup_{P_{V}} T')
 \end{array}
\]
is isomorphic to $\rmV_{\calC}(\bbSigma_{g,V})$.

Let us consider the linear map 
\[
 \begin{array}{rccc}
  \calE : & \calV(M_g;\bbSigma_{g,V,H}) & \rightarrow & \rmV_{\calC}(\bbSigma_{g,V}) \\
  & T & \mapsto & [ M_g \cup_{\Sigma_g} (I \times \Sigma_g), T \cup_{(P_V \cup \overline{P}_H)} T_{\Lambda},0]
 \end{array}
\] 
where $T_{\Lambda} \subset I \times \Sigma_g$ is the $\calC$-colored bichrome graph from 
$P_{V} \cup \overline{P}_H$ to $P_V$ represented in Figure \ref{open_skein} inside $I \times N(P_V \cup \overline{P}_H) 
\subset I \times \Sigma_g$ for a tubular neighborhood $N(P_V \cup \overline{P}_H)$ of $P_V \cup \overline{P}_H$ inside $\Sigma_g$.

\begin{figure}[htb]
 \centering
 \includegraphics{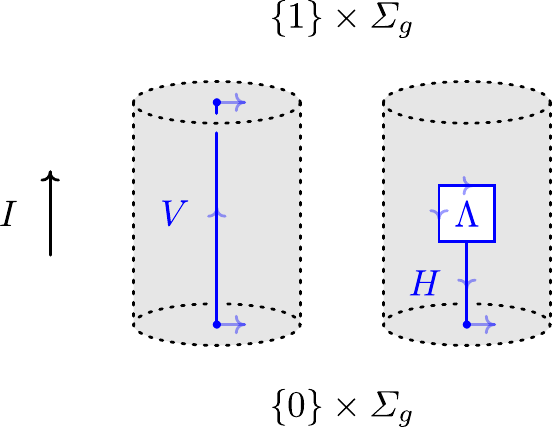}
 \caption{The $\calC$-colored bichrome graph $T_{\Lambda}$.}
 \label{open_skein}
\end{figure}

\begin{proposition}\label{open_skein_compute_TQFT}
 The linear map $\calE : \calV(M_g;\bbSigma_{g,V,H}) \rightarrow \rmV_{\calC}(\bbSigma_{g,V})$ is surjective.
\end{proposition}

\begin{proof}
 Every vector of $\rmV_{\calC}(\bbSigma_{g,V})$ is of the form $[ M_g \cup_{\Sigma_g} (I \times \Sigma_g), T,0]$
 for some admissible $\calC$-colored bichrome graph $T \in M_g \cup (I \times \Sigma_g)$ from $\varnothing$ to $P_V$ thanks to 
 Proposition \ref{connection+skein}. Up to skein equivalence we can suppose $T$ features a $\Lambda$-colored coupon, 
 and up to isotopy we can conclude.
\end{proof}

Let us fix now a standard embedding $\iota_g$ of ${M}_g$ into
$\R^2 \times I$ and let us consider the linear map
\[
 \Phi : \calV(M_g;\bbSigma_{g,V,H}) \rightarrow \tilde{\calS}_{g,V}
\]
sending 
an
admissible $\calC$-colored bichrome graph $T \subset M_g$ from $\varnothing$ to 
$P_{V} \cup \overline{P}_H$ to the morphism of $\calC$ 
obtained by evaluating the Hennings-Reshetikhin-Turaev functor $F_{\lambda}$ against the $\calC$-colored 
$g$-string link graph obtained as described by Figure \ref{skein_to_morphism}.

\begin{figure}[htb]
 \centering
 \includegraphics{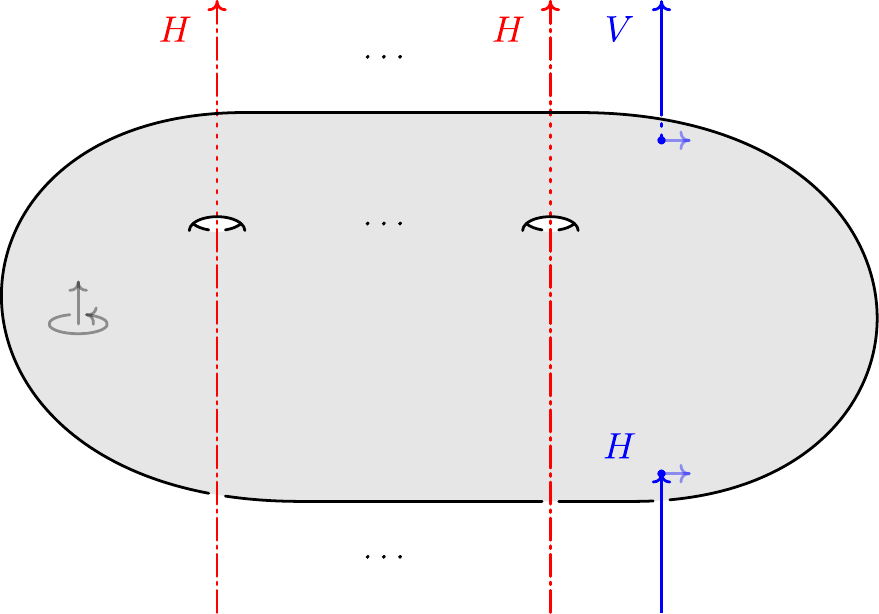}
 \caption{How to obtain a $g$-string link graph from a $\calC$-colored
   bichrome graph $T \subset M_g$ from $\varnothing$ to
   $P_{V} \sqcup \overline{P}_H$ using the standard embedding
   $\iota_g$.}
 \label{skein_to_morphism}
\end{figure}

\begin{proposition}\label{surjectivity_proposition}
 The linear map $\Phi : \calV(M_g;\bbSigma_{g,V,H}) \rightarrow \tilde{\calS}_{g,V}$ is surjective.
\end{proposition}

\begin{proof}
 Let us consider the linear map $\Psi : \tilde{\calX}_{g,V} \rightarrow \calV(M_g;\bbSigma_{g,V,H})$ sending
 every $f$ in $\tilde{\calX}_{g,V}$ to the admissible $\calC$-colored bichrome graph 
 $T_{f}$ represented in Figure \ref{morphism_to_skein}. 
 Now, using the notation of Proposition \ref{P:HomG}, we have 
 \[
  \Phi \circ \Psi = \Theta.
 \]
 Then $\Phi$ is surjective because $\Theta$ is an isomorphism. \qedhere
 
 \begin{figure}[htb]
  \centering
  \includegraphics{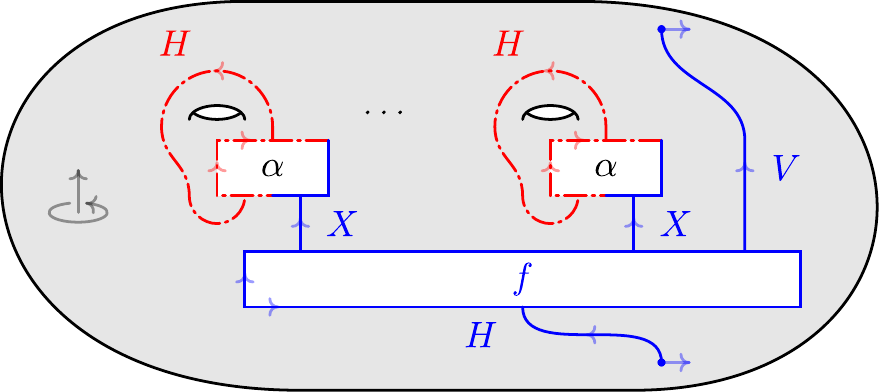}
  \caption{The admissible $\calC$-colored bichrome graph $T_{f}$.}
  \label{morphism_to_skein}
 \end{figure}
 
\end{proof}

\FloatBarrier

Let us fix now a standard embedding $\iota'_g$ of $M'_g$ into
$\R^2 \times I$ and let us consider the linear map
\[
 \Phi' : \calV'(M'_g;\bbSigma_{g,V}) \rightarrow \calS'_{g,V}
\]
sending every $\calC$-colored bichrome graph $T' \subset M'_g$ from $P_{V}$ to $\varnothing$
to the morphism of $\calC$ obtained by evaluating the Hennings-Reshetikhin-Turaev functor $F_{\lambda}$ against the $\calC$-colored 
$g$-string link graph obtained as described by Figure \ref{skein_to_morphism_prime}.

\begin{figure}[ht]
 \centering
 \includegraphics{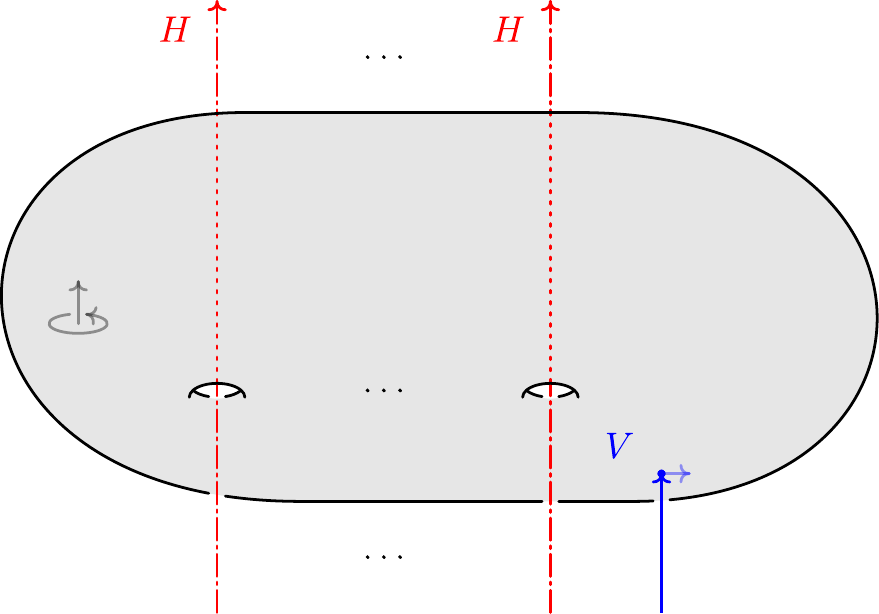}
 \caption{How to obtain a $g$-string link graph from a $\calC$-colored
 bichrome graph $T' \subset M'_g$ from $P_{V}$ to $\varnothing$ using
 the standard embedding $\iota'_g$.}
 \label{skein_to_morphism_prime}
\end{figure}

\begin{proposition}\label{surjectivity_proposition'}
 The linear map $\Phi' : \calV'(M'_g;\bbSigma_{g,V}) \rightarrow \calS'_{g,V}$ is surjective.
\end{proposition}

\begin{proof} 
 Let us consider the linear map $\Psi' : \calX'_{g,V} \rightarrow \calV'(M'_g;\bbSigma_{g,V})$ sending
 every $f'$ in $\calX'_{g,V}$ to the admissible $\calC$-colored bichrome graph 
 $T_{f'}$ represented in Figure \ref{morphism_to_skein_prime}.
 Now, using the notation of Proposition \ref{P:HomG}, we have 
 \[
  \Phi' \circ \Psi' = \Theta'.
 \]
 Then $\Phi'$ is surjective because $\Theta'$ is an isomorphism. \qedhere
 
 \begin{figure}[htb]
  \centering
  \includegraphics{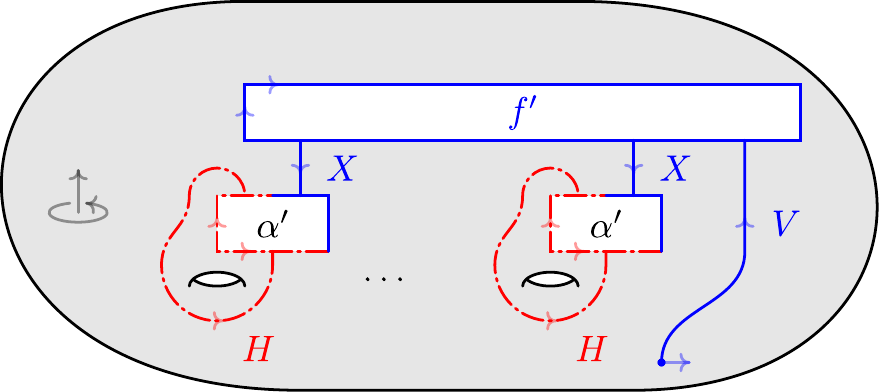}
  \caption{The admissible $\calC$-colored bichrome graph $T_{f'}$.}
  \label{morphism_to_skein_prime}
 \end{figure}
 
\end{proof}

\FloatBarrier

\begin{proposition}\label{P:pair}
 Every $\calC$-colored bichrome graph $T'$ in $\calV'(M'_g;\bbSigma_{g,V})$ and every 
 admissible $\calC$-colored bichrome graph $T$ in $\calV(M_g;\bbSigma_{g,V,H})$ satisfy
 \[
  \brk{T',\calE(T)}_{S^3} = \brk{\Phi' (T'),\Phi(T)}_{\calS} = \brk{\Theta'^{-1}(\Phi' (T')),
  \Theta^{-1}(\Phi(T))}_{\calX}.
 \]
\end{proposition}

\begin{proof}
 We can compute $\brk{T',\calE(T)}_{S^3}$ using the surgery presentation of Figure \ref{surgery_presentation} for the morphism 
 \[
  (M_g \cup_{\Sigma_g} (I \times \Sigma_g) \cup_{\Sigma_g} M'_g ,\calE(T) \cup_{P_V} T',0)
 \]
 of $\adCob_{\calC}$.  
 
 \begin{figure}[ht]
  \centering
  \includegraphics{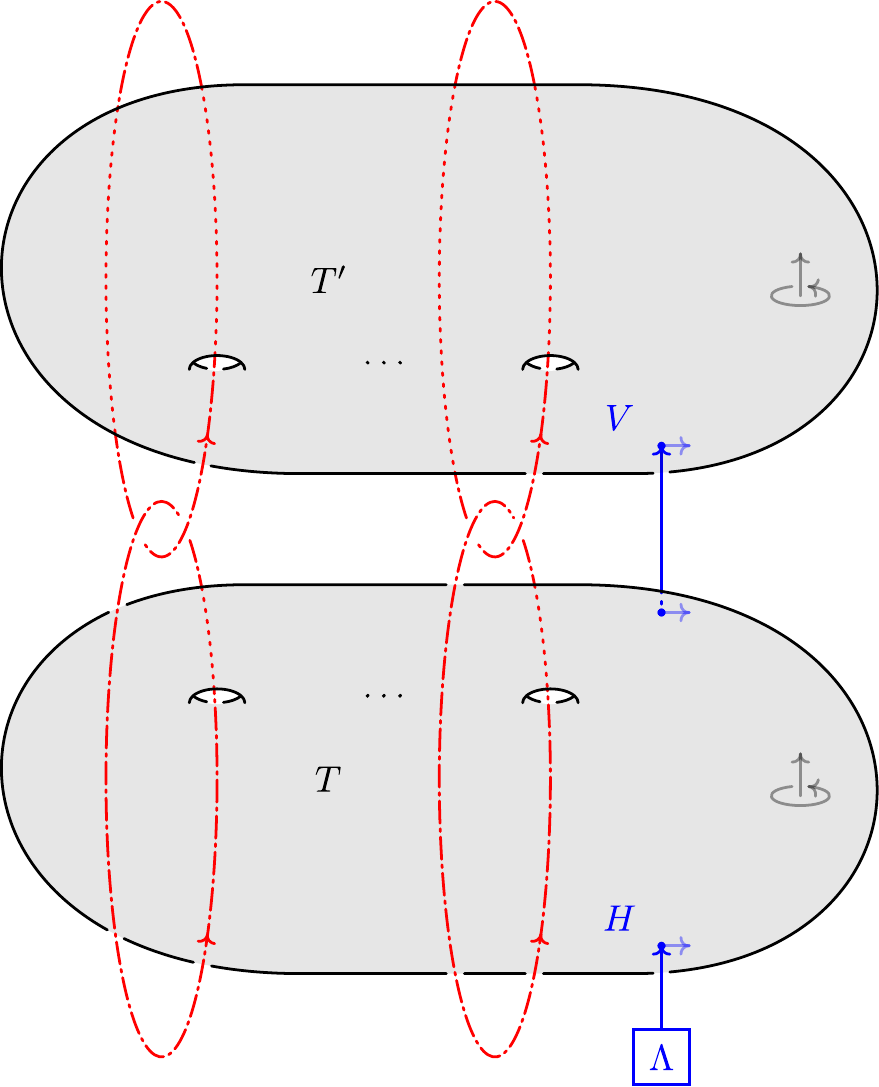}
  \caption{Surgery presentation computing $\brk{T',\calE(T)}_{S^3}$.}
  \label{surgery_presentation}
 \end{figure}
 
 To see this is a surgery presentation of $S^3$ we can start with the Heegaard splitting $M_g \cup_{\Sigma_g} M'_g$ and remark that surgery on the red Hopf links allows us to disentangle
 the two handlebodies, see Figure \ref{F:Heegaard_splitting}.
 
 \begin{figure}[ht]
  \centering
  \includegraphics{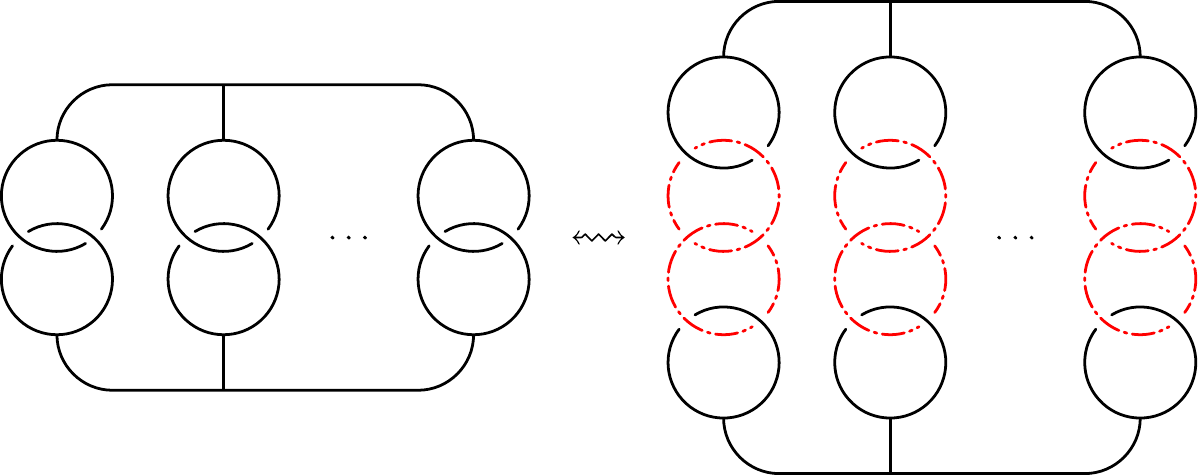}
  \caption{The complementary handlebodies $M_g$ and $M'_g$ can be realized as tubular neighborhoods of the two trivalent graphs in $S^3$ represented on the left. Surgery on the red Hopf links allows us to disentangle them.}
  \label{F:Heegaard_splitting}
 \end{figure}
 
 Then, in order to compute 
 \[
  \Hm \left( M_g \cup_{\Sigma_g} (I \times \Sigma_g) \cup_{\Sigma_g} M'_g ,\calE(T) \cup_{P_V} T',0 \right),
 \]
 we can consider $\calE(T) \cup_{P_V} T'$ and cut open the bottom $H$-colored blue edge in Figure \ref{surgery_presentation}. This will result in a $\calC$-colored bichrome graph from $(H,+)$ to itself which, in the notation of Figure \ref{F:algebraic_pairing_S}, is precisely $T_{\calS,\Phi(T),\Phi' (T')}$. Now the definition of the pairing $\brk{\cdot,\cdot}_{\calS}$, together with Proposition \ref{P:translation_of_pairings}, gives the result.
\end{proof}

If $\Sigma$ is a connected surface then let $g(\Sigma)$ denote its genus.
If $P = P_1 \cup \ldots \cup P_k$ is a $\calC$-colored blue set inside a $\Sigma$ 
with $P_i$ given by a single point having orientation $\epsilon_i$ and color $V_i$ for every $i \in \{1,\ldots,k\}$, then 
let $F_{\calC}(P)$ denote the object $V_1^{\epsilon_1} \otimes \cdots \otimes V_k^{\epsilon_k}$ of $\calC$ where
$V_i^{+} := V_i$ and $V_i^{-} := V_i^*$. Remark that the isomorphism class of
$F_{\calC}(P)$ is independent of the ordering of $P$.

\begin{corollary}\label{C:identification}
 If $\bbSigma = (\Sigma,P,\calL)$ is a connected object of $\adCob_{\calC}$ 
 then there exist isomorphisms $\rmV_{\calC}(\bbSigma) \cong \calX_{g(\Sigma),F_{\calC}(P)}$ and
 $\rmV'_{\calC}(\bbSigma) \cong \calX'_{g(\Sigma),F_{\calC}(P)}$ which are compatible with
 the pairings $\langle \cdot,\cdot \rangle_{\bbSigma}$ and $\langle \cdot,\cdot \rangle_{\calX}$.  
\end{corollary}

\begin{proof}
 Let $\gamma \subset \Sigma$ be a separating curve which cuts a disc $D$ containing $P$ from $\Sigma$. Let $\tilde{P}$
 denote a $\calC$-colored blue set inside $D$ given by a single point with positive orientation and color $F_{\calC}(P)$.
 Let $\tilde{\bbSigma}$ denote the object $(\Sigma,\tilde{P},\calL)$ and let us consider the morphism 
 $\bbI \times \bbSigma : \bbSigma \rightarrow \tilde{\bbSigma}$ given by $(I \times \Sigma,T_{\id},0)$
 where $T_{\id} \subset I \times \Sigma$ is the $\calC$-colored bichrome graph from $P$ to $\tilde{P}$ represented in Figure \ref{cylinder}
 inside $I \times D$. 
 Then $\rmV_{\calC}(\bbI \times \bbSigma) : \rmV_{\calC}(\bbSigma) \rightarrow \rmV_{\calC}(\tilde{\bbSigma})$ is an isomorphism.
 But now if $M \cup_{\Sigma} M'$ is a Heegard splitting of $S^3$ and if
 \[
  \langle \cdot , \cdot \rangle_{S^3} : \calV'(M';\tilde{\bbSigma}) \times \calV(M;\tilde{\bbSigma}) \rightarrow \Bbbk
 \]
 is the associated bilinear form then both $\rmV_{\calC}(\tilde{\bbSigma})$ and $\calX_{g(\Sigma),F_{\calC}(P)}$ 
 are isomorphic to the quotient of $\calV(M;\tilde{\bbSigma})$ with respect to the right radical of $\langle \cdot , \cdot \rangle_{S^3}$
 and both $\rmV'_{\calC}(\tilde{\bbSigma})$ and $\calX'_{g(\Sigma),F_{\calC}(P)}$ 
 are isomorphic to the quotient of $\calV'(M';\tilde{\bbSigma})$ with respect to the left radical of $\langle \cdot , \cdot \rangle_{S^3}$.

 \begin{figure}[tb]
  \centering
  \includegraphics{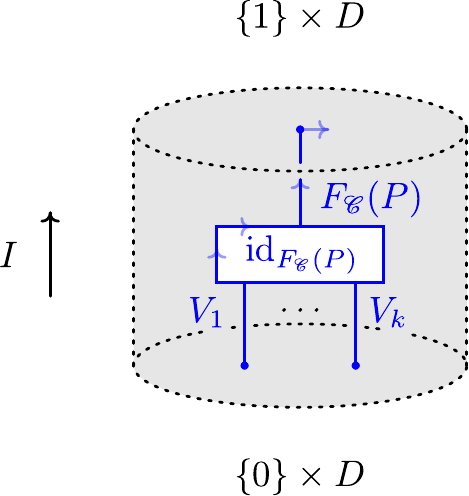}
  \caption{The $\calC$-colored bichrome graph $T_{\id}$.}
  \label{cylinder}
 \end{figure}
\end{proof}

\begin{remark}\label{R:Dehn_twist}
 Dehn twists act on TQFT spaces like curve operators. Indeed, let $\bbSigma = (\Sigma,P,\calL)$ be an object of $\adCob_{\calC}$ and let $\gamma \subset \Sigma$ be an oriented simple closed curve. If $I \times \Sigma$ denotes the cylinder cobordism from $\Sigma$ to itself, let $K_{\gamma} \subset I \times \Sigma$ denote the red knot $\{ \frac 12 \} \times \gamma$ with framing determined by the homology class $[\ell + m] \in H_1(\partial N)$, where then $N$ is a tubular neighborhood of $\{ \frac 12 \} \times \gamma$, where $m$ is a positive meridian of $\partial N$ in $N$, and where $\ell$ is a positive longitude of $\partial N$ contained in $\partial N \cap (I \times \gamma)$. Then the action of a Dehn twist along $\gamma$ on a vector $[\bbM'_{\bbSigma}]$ in $\rmV'_{\calC}(\bbSigma)$ is given by 
 \[
  [\bbM'_{\bbSigma} \circ (I \times \Sigma,K_{\gamma},0)].
 \]
\end{remark}

\FloatBarrier

\section{Examples and related constructions}\label{SS:Examples3Man}
In Section \ref{S:3-manifold_invariants} we constructed a 3-manifold invariant from a finite-dimensional non-degenerate unimodular 
ribbon Hopf algebra. If this Hopf algebra is factorizable then in Section \ref{S:TQFT} we showed the 3-manifold invariant extends to 
a $2+1$-TQFT. As discussed in the introduction, such Hopf algebras have been studied at length. In this section we highlight a 
few examples. We also relate our 3-manifold invariant to the logarithmic Hennings invariant of \cite{BBG17a} and to the
Generalized Kashaev invariant of \cite{M17}.

\subsection{Drinfeld doubles} 

An important family of examples is provided by Drinfeld doubles of finite-dimensional Hopf algebras. 
If $H$ is a finite-dimensional Hopf algebra, then it is well known that its Drinfeld double $D(H)$
is always factorizable, and in particular non-degenerate unimodular. See \cite{D88,EGNO15,H96,R11} 
for references. Furthermore, by Kauffman and Radford \cite{KR93}, it is also known precisely when a Drinfeld double is a ribbon 
Hopf algebra. Indeed, a right integral $\lambda \in H^*$ and a two-sided cointegral $\Lambda \in H$ satisfying $\lambda(\Lambda) = 1$
uniquely determine elements $a \in H$ and $\alpha \in H^*$ satisfying
\begin{gather*}
 f \lambda = f(a) \cdot \lambda, \quad \Lambda x = \alpha(x) \cdot \Lambda, \quad \epsilon(a) = \alpha(1_H) = 1, \\
 S^4(x) = \alpha^{-1}(x_{(1)})\alpha(x_{(3)}) \cdot ax_{(2)}a^{-1}
\end{gather*}
for all $f \in H^*$ and $x \in H$. Then $D(H)$ is ribbon if and only if there exist 
elements $g \in H$ and $\gamma \in H^*$ satisfying
\begin{gather*}
 g^2 = a, \quad \gamma^2 = \alpha, \quad \epsilon(g) = \gamma(1_H) = 1, \\
 S^2(x) = \gamma^{-1}(x_{(1)})\gamma(x_{(3)}) \cdot gx_{(2)}g^{-1}
\end{gather*}
for every $x \in H$. 

\subsection{Quantum groups}\label{SS:small_quantum_groups}

Many examples of the type of Hopf algebras we are considering come from quantum groups.  
In particular, to each simple Lie algebra one can associate several finite-dimensional quantum groups depending on many ingredients,
including the choice of a root of unity. These factors determine whether the quantum group is ribbon and/or factorizable, 
for details see \cite{L95b,LN15,LO16,KS97}. In this subsection, we will discuss finite-dimensional quantum groups associated to $\sl_2$, 
which have different properties depending on the order of the root of unity.

Let us fix $r \geq 3$ and let us choose the primitive $r$-th root of unity $q = e^{\frac{2 \pi i}{r}}$.
Let us set $\bar{r} := r$ if $r$ is odd and $\bar{r} := \frac{r}{2}$ if $r$ is even.
We also introduce for all $k \geq \ell \in \N$ the standard notation 
\[
 \{ k \} := q^k - q^{-k},
 \quad [k] := \frac{\{ k \}}{\{ 1 \}},
 \quad [k]! := [k][k-1]\cdots[1].
\]
We denote with $\bar{U}_q \sl_2$ the $\C$-algebra with generators
$\{ E,F,K,K^{-1} \}$ and relations
\begin{gather*}
 K K^{-1} = K^{-1} K = 1, \quad K E K^{-1} = q^2 E, \quad K F K^{-1} = q^{-2} F, \\
 [E,F] = \frac{K - K^{-1}}{q-q^{-1}}, \quad E^{\bar{r}} = F^{\bar{r}} = 0, \quad K^r = 1.
\end{gather*}
We can make $\bar{U}_q \sl_2$ into a Hopf algebra by setting
\begin{align*}
 \Delta(E) &= E \otimes K + 1 \otimes E, & \varepsilon(E) &= 0, & S(E) &= -E K^{-1}, \\
 \Delta(F) &= F \otimes 1 + K^{-1} \otimes F, & \varepsilon(F) &= 0, & S(F) &= - K F, \\
 \Delta(K) &= K \otimes K, & \varepsilon(K) &= 1, & S(K) &= K^{- 1}.
\end{align*}

A Poincaré-Birkhoff-Witt basis for $\bar{U}_q \sl_2$ is given by
\[
 \left\{ E^b F^c K^m \left| \ 0 \leq b,c \leq \bar{r} - 1, \ 0 \leq m \leq r - 1 \right. \right\}.
\]
A right integral $\lambda$ of $\bar{U}_q \sl_2$ is determined by
\[
\lambda \left( E^b F^c K^m \right) = \xi \delta_{b,\bar r-1} \delta_{c,\bar r-1} \delta_{m,r-\bar r+1}
\]
where $\xi$ is any non-zero constant.  
A two-sided cointegral $\Lambda$ of $\bar{U}_q \sl_2$ is given by
\[
  \frac{1}{\xi} \cdot \sum_{m = 0}^{r - 1}E^{\bar{r}-1} F^{\bar{r}-1} K^m =
  \frac{1}{\xi} \cdot \sum_{m = 0}^{r - 1} F^{\bar{r}-1}E^{\bar{r}-1} K^m.
\]
In particular, $\bar{U}_q \sl_2$ is always unimodular.
A pivotal element $g \in \bar{U}_q \sl_2$ is given by $K^{\bar{r}+1}$.
Furthermore, as shown in \cite{M95}, when $r$ is odd $\bar{U}_q \sl_2$ is also ribbon and factorizable.
Indeed, an R-matrix $R \in \bar{U}_q \sl_2 \otimes \bar{U}_q \sl_2$ is given by
\[
 \frac{1}{r} \cdot \sum_{b,\ell,m = 0}^{r - 1} \frac{\{ 1 \}^b}{[b]!}
 q^{\frac{b(b-1)}{2} + 2(b(\ell - m) - \ell m)} \cdot E^b K^\ell \otimes F^b K^m,
\]
and a ribbon element $v \in \bar{U}_q \sl_2$ is given by
\[
 \frac{1}{r} \cdot \sum_{b,\ell,m = 0}^{r - 1} \frac{\{ 1 \}^b}{[b]!} q^{\frac{b(b-1)}{2} - \frac{(r+1)(b-m-1)^2}{2} + 2(\ell^2-bm)} 
 \cdot E^b F^b K^m.
\]
Thus, when $r$ is odd, $\bar{U}_q\sl_2$ gives rise to a TQFT as in Section \ref{S:TQFT}.  

For odd values of $r$ we call $\bar{U}_q \sl_2$ the \textit{small quantum group of $\sl_2$}, while for even values of $r$ we call it the 
\textit{restricted quantum group of $\sl_2$}.

\subsection{Logarithmic Hennings invariants}\label{SS:Log_Hennings}
The logarithmic Hennings invariant constructed in \cite{BBG17a} is based on the restricted version of the quantum group of 
$\sl_2$, that is $\bar{U}_q\sl_2$ when $r$ is even. 
This Hopf algebra is not quasi-triangular and so the construction of this paper does not immediately apply. 
However, the restricted quantum group of $\sl_2$ admits a ribbon extension $D$ which is unimodular and non-degenerate,
and which therefore allows for the construction of our 3-manifold invariant. As we will explain next, this 3-manifold invariant 
recovers the logarithmic Hennings invariant. Note that, since $D$ is not factorizable,
our $2+1$-TQFT construction does not directly apply to the restricted case. 

Let us recall the definition of the ribbon extension $D$ of $\bar{U}_q \sl_2$ when $r = 2p$ for the choice of the 
primitive $2p$-th root of unity $q = e^{\frac{\pi i}{p}}$: let $D$ be the $\C$-algebra with generators $\{ e,f,k,k^{-1} \}$ and relations
\begin{gather*}
 k k^{-1} = k^{-1} k = 1, \quad k e k^{-1} = q e, \quad k f k^{-1} = q^{-1}f, \\
 [e,f] = \frac{k^2 - k^{-2}}{q-q^{-1}}, \quad e^p = f^p = 0, \quad k^{4p} = 1.
\end{gather*}
It can be made into a Hopf algebra by setting
\begin{align*}
 \Delta(e) &= e \otimes k^2 + 1 \otimes e, & \varepsilon(e) &= 0, & S(e) &= -e k^{-2}, \\
 \Delta(f) &= f \otimes 1 + k^{-2} \otimes f, & \varepsilon(f) &= 0, & S(f) &= - k^2 f, \\
 \Delta(k) &= k \otimes k, & \varepsilon(k) &= 1, & S(k) &= k^{- 1}.
\end{align*}
The restricted quantum group $\bar{U}_q \sl_2$ embeds into $D$ by sending $E$ to $e$, $F$ to $f$ and $K$ to $k^2$.  
We denote $U$ as the  image of $\bar{U}_q \sl_2$ in $D$.
A Poincaré-Birkhoff-Witt basis for $D$ is given by
\[
 \left\{ e^b f^c k^m \left| \ 0 \leq b,c \leq p - 1, \ 0 \leq m \leq 4p - 1 \right. \right\}.
\]
A right integral $\lambda$ of $D$ is determined by
\[
 \lambda \left( e^b f^c k^m  \right) = \xi \delta_{b,p-1} \delta_{c,p-1} \delta_{m,2p + 2}.
\]
Following \cite{M17,BBG17a}, we choose the normalization 
 \[
  \xi = \sqrt{\frac{2}{p}}([p-1]!)^2.
 \]
A two-sided cointegral $\Lambda$ of $D$ is given by
\[
 \frac{1}{\xi} \cdot \sum_{m = 0}^{4p - 1} e^{p-1} f^{p-1} k^m .
\]
In particular, $D$ is unimodular. A pivotal element $g \in D$ is given by $k^{2p+2}$.
An R-matrix $R \in D \otimes D$ is given by 
\[
 \frac{1}{4p} \cdot \sum_{b=0}^{p-1} \sum_{\ell,m = 0}^{4p - 1} 
 \frac{\{ 1 \}^b}{[b]!} q^{\frac{b(b-1)}{2} + b(\ell - m) - \frac{\ell m}{2}} \cdot e^b k^\ell \otimes f^b k^m,
\]
where $q^{\frac{1}{2}} := e^{\frac{\pi i}{2p}}$. A ribbon element $v \in D$ is given by 
\[
 \frac{1 - i}{2 \sqrt{p}} \cdot \sum_{b = 0}^{p - 1} \sum_{m=0}^{2p-1} \frac{\{ 1 \}^b}{[b]!} 
 q^{-\frac{b(2m+1)}{2} + \frac{(m-p-1)^2}{2}} \cdot e^b f^b k^{2m}.
\]
An easy computation shows that
\[
 \Delta_- = \lambda(v) = \frac{1 - i}{\sqrt{2} p} \{ 1 \}^{p-1}[p-1]! q^{-\frac{(p-1)(2p+3)}{2}} \neq 0.
\]
An analogous computation shows that $\Delta_+ = \lambda(v^{-1}) \neq 0$.
Therefore, $D$ is non-degenerate.  

Summarizing the above we have that $D$ is a finite-dimensional non-degenerate unimodular ribbon Hopf algebra.  
Thus, Theorem \ref{T:3ManInvMainT} implies there exists a renormalized Hennings invariant $\Hm$ associated with $\calC=D$-mod.  
By restricting the integral of $D$ to the subalgebra $U$ we recover the formulas
given in the previous subsection and in \cite{BBG17a}.

Next, we will explain how $\Hm$ is related to the logarithmic Hennings invariant $\rmH^{\log}$ of \cite{BBG17a}.
The latter is an invariant of closed 3-manifolds endowed with links which are labeled with elements of the Hopf algebra $U$ as follows.
Let $L = L_+ \cup L_-$ be a link inside a closed 3-manifold $M$.
Each component of $L_-$ is labeled with an element of the 0th-Hochschild homology $\HH_0(U) = U/[U,U]$ of $U$, 
where $[U,U]=\Span \{ xy - yx | x,y \in U \}$.
Each component of $L_+$ is labeled with an element of the center $\rmZ(U)$ of $U$.  
The invariant $\rmH^{\log}(M,L_+ \cup L_-)$ is defined using the underlying ribbon structure of $D$, but the result only 
depends on $U$ since the M-matrix $R_{21}R_{12}$ is in $U \otimes U$ and the ribbon element $v$ is in $U$.  
We will now associate with $L_+ \cup L_-$ a $\calC$-colored $n$-string link $L'_+ \cup L'_-$.   
Remark that every $z \in \rmZ(U)$ is also central in $D$ under inclusion, so that $L_z : D \to D$ is a morphism in the category $\calC=D$-$\mod$.   
Let $L'_+$ be the $\calC$-colored red graph obtained from $L_+ \subset L_+ \cup L_-$ by adding to every $z$-colored 
component a $(1,1)$-coupon colored with $L_z$.
On the other hand, for any $x \in D$ the right multiplication $R_x : D \to D$ is a morphism in $\calC$.   
Let $L'_-$ be the $\calC$-colored blue graph obtained from $L_- \subset L_+\cup L_-$ by adding to every $[x]$-colored 
component a $(1,1)$-coupon colored with $R_x$ where $x \in U$ is any representative of $[x] \in \HH_0(U)=U/[U,U]$.
All edges of $L'_+ \cup L'_-$ are colored with the regular representation $D$.  
Properties of the modified trace imply that $\Hm(M,L'_+ \cup L'_-)$ does not depend on the choice of the representative $x$ of 
$[x] \in \HH_0(U)$.  
From the definition of the invariants we have the equality  
\[
 \Hm(M,L'_+ \cup L'_-)=\rmH^{\log}(M,L_+\cup L_-).
\]

\subsection{Generalized Kashaev invariants}\label{SS:Gen_Kashaev}

In \cite{M17}, Jun Murakami defines a generalized Kashaev invariant of a link $K$ in a 3-manifold 
$M$ by combining the Hennings invariant associated with the restricted version $\bar{U}_q \sl_2$ of quantum $\sl_2$ 
and the ADO invariant associated with the medium version $\tilde{U}_q \sl_2$ of quantum $\sl_2$, see \cite{ADO92, M17}. 
These two theories overlap with the use of the Steinberg-Kashaev
module $V_0$. At a primitive $2p$-th root of unity, this is a $p$-dimensional simple projective module of both $\bar{U}_q \sl_2$ and 
$\tilde{U}_q \sl_2$. 

Our functors $F_{\lambda}$ are a generalization of Murakami's invariant $\wt{\GK}_p $:  
first, let $M$ be represented by surgery along a framed link $L$ and let $\widehat{K}$ be the pre-image of $K$ in $S^3$. 
Let $T$ be the $(1,1)$-tangle obtained from $L \cup \widehat{K}$ by cutting open one of the components of
$\widehat{K}$.
Let $T'$ be the $\calC$-colored $n$-string tangle determined by $T$ by declaring every component of $L$ to be red and 
every component of $\widehat{K}$ to be blue and colored with $V_0$. Then by construction
\[
 \wt{\GK}_p(T)=F_{\lambda}(T'),
\]
and $\wt{\GK}_p(T)$ does not depend on the component of $K$ which is cut, see Theorem 5 of \cite{M17}.  
Our Theorem \ref{T:F'Exists} gives a new proof of this fact.  
Note that in this case the renormalization given by the modified trace only changes $F_{\lambda}$ by a global constant 
because each component of $\widehat{K}$ is colored with the same module $V_0$.  

As in Subsection \ref{3-manifold_invariants}, Murakami uses standard techniques to scale $\wt{\GK}_p(T)$ 
and construct an invariant ${\GK}_p$ of the pair $(M,K)$.    
Thus, after a global normalization, we have that ${\GK}_p(M,K)$ is equal to $\Hm(M,T')$.

\end{document}